\documentclass[10pt,letterpaper]{amsart}
\usepackage{dslihead}

\title{Local-global compatibility over function fields}

\author{Siyan Daniel Li-Huerta}
\email{sli@math.harvard.edu}
\address{Department of Mathematics\\Harvard University\\1 Oxford Street\\Cambridge, MA 02138}

\usepackage{amsmath}

\swapnumbers
\theoremstyle{plain}
\newtheorem{thm}[subsection]{Theorem}
\newtheorem*{thm*}{Theorem}
\newtheorem{lem}[subsection]{Lemma}
\newtheorem*{lem*}{Lemma}
\newtheorem{prop}[subsection]{Proposition}
\newtheorem*{prop*}{Proposition}

\newtheorem*{conj*}{Conjecture}

\newtheorem*{cor*}{Corollary}

\newtheorem*{thmA}{Theorem A}
\newtheorem*{thmB}{Theorem B}
\newtheorem*{thmC}{Theorem C}
\newtheorem*{thmD}{Theorem D}

\theoremstyle{definition}
\newtheorem{defn}[subsection]{Definition}
\newtheorem*{defn*}{Definition}

\theoremstyle{remark}

\newtheorem*{rem*}{Remark}
\newtheorem*{rems*}{Remarks}

\newcommand{\fLocSht}{\fL\mathrm{oc}\fS\mathrm{ht}}
\newcommand{\cGr}{\cG\mathrm{r}}
\newcommand{\cLocSht}{\cL\mathrm{oc}\cS\mathrm{ht}}
\newcommand{\fFr}{\fF\mathrm{r}}
\newcommand{\cFr}{\cF\mathrm{r}}

\begin{document}

\begin{abstract}
We prove that V. Lafforgue's global Langlands correspondence is compatible with Fargues--Scholze's semisimplified local Langlands correspondence. As a consequence, we canonically lift Fargues--Scholze's construction to a non-semisimplified local Langlands correspondence for positive characteristic local fields. We also deduce that Fargues--Scholze's construction agrees with that of Genestier--Lafforgue, answering a question of Fargues--Scholze, Hansen, Harris, and Kaletha. The proof relies on a uniformization morphism for moduli spaces of shtukas.
\end{abstract}

\maketitle
\tableofcontents

\section*{Introduction}
The Langlands program predicts a relationship between automorphic forms and Galois representations. More precisely, in the case of a connected reductive group $\mathbf{G}$ over a global function field $\mathbf{F}$ of characteristic $p>0$, the Langlands program posits a canonical map
\begin{align*}
\GLC_{\mathbf{G}}: \left\{
  {\begin{tabular}{c}
    cuspidal automorphic\\
    representations of $\mathbf{G}(\bA_{\mathbf{F}})$
  \end{tabular}}
\right\}\ra \left\{
  {\begin{tabular}{c}
  $L$-parameters \\
  for $\mathbf{G}$ over $\mathbf{F}$
  \end{tabular}}
    \right\},
\end{align*}
where $\bA_{\mathbf{F}}$ denotes the adele ring of $\mathbf{F}$, and all representations are taken with $\ov\bQ_\ell$-coefficients for some $\ell\neq p$. In a landmark result, such a map $\GLC_{\mathbf{G}}$ was constructed by V. Lafforgue \cite{Laf16}.

In the case of a connected reductive group $G$ over a nonarchimedean local field $F$, the Langlands program predicts a similar map
\begin{align*}\label{eqn:LLC}
\LLC_G: \left\{
  {\begin{tabular}{c}
    irreducible smooth\\
    representations of $G(F)$
  \end{tabular}}
\right\}\ra \left\{
  {\begin{tabular}{c}
  $L$-parameters \\
  for $G$ over $F$
  \end{tabular}}
    \right\}.\tag{$\dagger$}
\end{align*}
Recent breakthrough work of Fargues--Scholze \cite{FS21} constructs such a map up to semisimplification; namely, they construct a map
\begin{align*}\label{eqn:LLCss}
\LLC_G^{\semis}: \left\{
  {\begin{tabular}{c}
    irreducible smooth\\
    representations of $G(F)$
  \end{tabular}}
\right\}\ra \left\{
  {\begin{tabular}{c}
  \emph{semisimple} $L$-parameters \\
  for $G$ over $F$
  \end{tabular}}
    \right\}.\tag{$\ddagger$}
\end{align*}

Our main result is that V. Lafforgue's global Langlands correspondence is compatible with Fargues--Scholze's semisimplified local Langlands correspondence.
\begin{thmA}
  Let $v$ be a place of $\mathbf{F}$. Then the square
  \begin{align*}
        \xymatrix{
\left\{
  {\begin{tabular}{c}
    cuspidal automorphic\\
    representations of $\mathbf{G}(\bA_{\mathbf{F}})$
  \end{tabular}}
\right\}\ar[r]^-{\GLC_{\mathbf{G}}}\ar[d]^-{(-)_v} & \left\{
  {\begin{tabular}{c}
  $L$-parameters \\
  for $\mathbf{G}$ over $\mathbf{F}$
  \end{tabular}}
    \right\}\ar[d]^-{(-)|_{W_{\mathbf{F}_v}}^{\semis}} \\
\left\{
  {\begin{tabular}{c}
    irreducible smooth\\
    representations of $\mathbf{G}(\mathbf{F}_v)$
  \end{tabular}}
\right\}\ar[r]^-{\LLC^{\semis}_{\mathbf{G}_{\mathbf{F}_v}}} & \left\{
  {\begin{tabular}{c}
  semisimple $L$-parameters \\
  for $\mathbf{G}_{\mathbf{F}_v}$ over $\mathbf{F}_v$
  \end{tabular}}
\right\}
    }
  \end{align*}
commutes.
\end{thmA}
Since $\GLC_{\mathbf{G}}$ \cite[Th\'eor\`eme 12.3]{Laf16} and $\LLC_G^{\semis}$ \cite[Theorem IX.0.5]{FS21} are compatible with the Satake isomorphism at unramified places, for a given cuspidal automorphic representation this is already known at \emph{unramified} places.

We actually prove a refinement of Theorem A on the level of \emph{excursion algebras}; see Theorem \ref{ss:localglobalcompatibilityalgebras}.
\begin{rems*}\hfill
  \begin{enumerate}[(1)]
  \item V. Lafforgue \cite[Th\'eor\`eme 13.2]{Laf16} and Fargues--Scholze \cite[Proposition IX.4.1]{FS21} prove a version of their results with $\ov\bF_\ell$-coefficients, and the analogous version of Theorem A also holds in this mod-$\ell$ context. See Theorem \ref{ss:BernsteincenterGLFS}.

  \item Once one constructs a non-semisimplified local Langlands correspondence as in Equation (\ref{eqn:LLC}) (e.g. see Theorem B below), one can ask whether Theorem A holds before semisimplification. The answer is already negative when $\mathbf{G}$ is the units of a quaternion algebra \cite[Remarque 0.3]{GL17}. More generally, Arthur's conjecture \cite{Art89} predicts that the answer is negative precisely for global $A$-packets where a local $A$-packet $\Pi_{\psi_v}$ contains a representation whose $L$-parameter does not equal the $L$-parameter associated with $\Pi_{\psi_v}$. For instance, examples of Howe--Piatetski-Shapiro \cite{HPS79} show that the answer is also negative when $\mathbf{G}$ is $\Sp_4$. 
  \end{enumerate}
\end{rems*}
We now turn to some consequences of Theorem A. When $\Char{F}>0$, Theorem A enables us to remove the ``up to semisimplification'' ambiguity in Fargues--Scholze's construction.
\begin{thmB}
Assume that $\Char{F}>0$. Then $\LLC^{\semis}_G$ canonically lifts to a non-semisimplified local Langlands correspondence $\LLC_G$ as in Equation $($\ref{eqn:LLC}$)$.
\end{thmB}
The proof that Theorem A implies Theorem B is due to Gan--Harris--Sawin \cite{GHSBP21}; roughly, the idea is to maneuver into a situation where Theorem A holds even before semisimplification. This uses a globalization result of Beuzart-Plessis \cite{GHSBP21}, work of Heinloth--Ng\^o--Yun \cite{HNY13} on $\ell$-adic Kloosterman sheaves, results of Xu--Zhu \cite{XZ22} on their $p$-adic companions, and Deligne's purity theorem.

Our next result concerns previous work of Genestier--Lafforgue \cite{GL17}, who also constructed a map as in Equation (\ref{eqn:LLCss}) when $\Char{F}>0$. Genestier--Lafforgue obtained a version of Theorem A for their construction, and since this property basically uniquely characterizes such maps, we deduce the following result.
\begin{thmC}
The Genestier--Lafforgue correspondence agrees with the Fargues--Scholze correspondence.
\end{thmC}
This answers a question of Fargues--Scholze \cite{FS21}, Hansen, Harris, and Kaletha \cite{Kal22}. We also prove a refinement of Theorem C on the level of Bernstein centers; see Theorem \ref{ss:BernsteincenterGLFS}.
\begin{rem*}
Conversely, if we only had Theorem C, then work of Genestier--Lafforgue would imply Theorem A. However, our proof of Theorem A is independent of their results.
\end{rem*}
We conclude by showing that $\LLC^{\semis}_G$ satisfies the expected compatibility with the local Jacquet--Langlands correspondence \cite{Bad02}, which we denote by $\JL$, when $\Char{F}>0$ and $G$ is the units of a central simple algebra over $F$.
\begin{thmD}
Assume that $\Char{F}>0$ and $G$ is the units of a central simple algebra over $F$. For any irreducible essentially $L^2$ representation $\pi$ of $G(F)$, we have $\LLC^{\semis}_G(\pi)=\LLC^{\semis}_{\GL_n}(\JL(\pi))$.
\end{thmD}
When $\Char{F}>0$, Theorem D was previously only known when $G$ is $\GL_n$ or the units of a central division algebra over $F$ \cite[Theorem IX.7.4]{FS21}. The $\Char{F}=0$ analogue of Theorem D is due to Hansen--Kaletha--Weinstein \cite[Theorem 6.6.1]{HKW17} as a consequence of their work on the local Kottwitz conjecture.

Let us discuss our proof of Theorem A. Elements of our strategy go back to Deligne's letter to Piatetski-Shapiro \cite{DelPS}, which proves local-global compatibility for modular forms. Their associated Galois representations are constructed via the cohomology of modular curves, and one of Deligne's key ideas was to restrict to the supersingular locus, using the uniformization of the latter by Lubin--Tate space to relate the local and global Langlands correspondences for $\GL_2$.

Deligne's proof, as well as subsequent works on local-global compatibility using basic uniformization \cite{Car86, HT01, Sch13, LH17}, also crucially relies on arguments specific to the particular group $\mathbf{G}$ in question. However, our proof of Theorem A is uniform in all groups $\mathbf{G}$.

We begin by observing that, since the correspondences of V. Lafforgue and Fargues--Scholze are constructed via \emph{excursion operators}, it suffices to show that said operators are compatible. Let us recall their definition, which involves moduli spaces of shtukas. For simplicity, assume that $\mathbf{G}$ is split, and write $\wh{\mathbf{G}}$ for the dual group of $\mathbf{G}$ over $\ov\bQ_\ell$. For any finite set $I$ and representation $V$ of $\wh{\mathbf{G}}^I$, write $\Sht^I_{\mathbf{G},V}$ for the associated moduli space of \emph{global $\mathbf{G}$-shtukas},\footnote{In the introduction, we ignore convolution data and level structures in our notation.} which is a Deligne--Mumford stack. Work of Xue \cite{Xue20b} naturally endows the compactly supported intersection cohomology $H^I_V$ of its generic fiber with an action of $W_{\mathbf{F}}^I$, where $W_{\mathbf{F}}$ denotes the absolute Weil group. For any $x$ and $\xi$ in $V$ and $V^\vee$, respectively, that are fixed by the image of $\De:\wh{\mathbf{G}}\hra\wh{\mathbf{G}}^I$, and any $\ga_\bullet$ in $W_{\mathbf{F}}^I$, the associated global excursion operator is
\begin{gather*}\label{eqn:introglobalexcursion}
H^{*}_{\mathbf{1}}\lra^xH^{*}_{V|_{\De(\wh{G})}}=H^{I}_V\lra^{\ga_{\bullet}}H^{I}_V=H^{*}_{V|_{\De(\wh{G})}}\lra^\xi H^{*}_{\mathbf{1}},\tag{$\heartsuit$}
\end{gather*}
where $*$ denotes the singleton set, and $\mathbf{1}$ denotes the trivial representation.

In the local setting, write $\cLocSht^I_{\mathbf{G},V}$ for the associated moduli space of \emph{local $\mathbf{G}$-shtukas}, which is an analytic adic space. Work of Fargues--Scholze \cite{FS21} naturally endows the intersection homology $H^{\loc,I}_V$ of its generic fiber with an action of $W_{\mathbf{F}_v}^I$, so when $\ga_\bullet$ lies in $W_{\mathbf{F}_v}^I$, one can form local excursion operators using the same recipe as in Equation (\ref{eqn:introglobalexcursion}).

We compare the local and global excursion operators using a uniformization morphism. To define it, first we construct a formal model $\fLocSht^I_{\mathbf{G},V}$ for $\cLocSht^I_{\mathbf{G},V}$ at hyperspecial level. Stating the formal moduli problem is straightforward, although comparing it with our original definition of local $\mathbf{G}$-shtukas requires an equicharacteristic version of Kedlaya--Liu's results \cite{KL15} on relative $p$-adic Hodge theory, which we prove. Next, we use Beauville--Laszlo gluing to construct a formally \'etale morphism of formal stacks
\begin{align*}
  \wh\Te:\fLocSht^I_{\mathbf{G},V}\ra\wh\Sht^I_{\mathbf{G},V}
\end{align*}
when the level is hyperspecial at $v$, where $\wh\Sht^I_{\mathbf{G},V}$ denotes the formal completion of $\Sht^I_{\mathbf{G},V}$ along $v^I$, and we assume that $\deg{v}=1$ for simplicity. This generalizes results of Arasteh Rad--Hartl \cite{AH13}.

From here, we restrict to a Harder--Narasimhan truncation $\Sht^{I,\leq s}_{\mathbf{G},V}$ of $\Sht^I_{\mathbf{G},V}$ and enlarge the level away from $v$. This yields a scheme that is locally of finite type, so we can use Huber's analytification \cite[(3.8)]{Hub94} to extend $\wh\Te$ to a morphism of analytic adic spaces
\begin{align*}
\Te:\cLocSht^{I,\leq s}_{\mathbf{G},V}\ra(\Sht^{I,\leq s}_{\mathbf{G},V})_{(\Spa\mathbf{F}_v)^I}.
\end{align*}
for deeper levels at $v$. To prove that $\Te$ is \'etale, it suffices to consider the case of hyperspecial level. There, we prove that $\fLocSht^I_{\mathbf{G},V}$ is a formal scheme that is locally formally of finite type, generalizing results of Arasteh Rad--Hartl \cite{EH14}. After restricting to a Harder--Narasimhan truncation, this lets us upgrade the formal \'etaleness of $\Te$ to \'etaleness, as desired.

Since $\Te$ is \'etale, we can form the $!$-pushforward map
\begin{align*}
\Te_!:H^{\loc,I,\leq s}_V\ra H^{I,\leq s}_V.
\end{align*}
After restricting to a Harder--Narasimhan truncation, this induces a morphism from the composition diagram in Equation (\ref{eqn:introglobalexcursion}) to the analogous composition diagram for $H^{\loc,I}_V$. We use this to prove that the global and local excursion operators are compatible, which concludes the proof of Theorem A.

With Theorem A in hand, let us return to the local context and sketch the proofs of Theorem B, Theorem C, and Theorem D. For Theorem B, compatibility with parabolic induction and the Langlands classification reduce us to the case of $L^2$ representations $\pi$. Then the Langlands program predicts $\LLC_G(\pi)$ to be the unique pure $L$-parameter whose semisimplification is $\LLC^{\semis}_G(\pi)$, if it exists. To construct this $L$-parameter, we use a globalization result of Beuzart-Plessis \cite{GHSBP21} to obtain a cuspidal automorphic representation $\Pi$ that has the same cuspidal support as $\pi$ at one place and is isomorphic to the cuspidal representation $\pi'$ considered by Gross--Reeder \cite{GR10} at another place. Combining Theorem A with work of Heinloth--Ng\^o--Yun \cite{HNY13} and Xu--Zhu \cite{XZ22} shows that the Fargues--Scholze parameter of $\pi'$ is irreducible. From here, applying Deligne's purity theorem and Theorem A to $\Pi$ yields the desired result.

For Theorem C, we instead reduce to the case of cuspidal representations. Then a classical Poincar\'e series argument and Theorem A give the desired result. Finally, for Theorem D we use the simple trace formula to construct a cuspidal automorphic representation of $\GL_n$ that globalizes $\JL(\pi)$ and transfers to a suitable central division algebra under the global Jacquet--Langlands correspondence \cite{BR17}. From here, the Chebotarev density theorem and Theorem A imply the desired result.

\subsection*{Outline}
In \S\ref{s:grassmannians}, we recall some facts about loop groups and Beilinson--Drinfeld affine Grassmannians. In \S\ref{s:formallocalshtukas}, we define the formal moduli problem and prove that it is a formal scheme that is locally formally of finite type. In \S\ref{s:zadichodgetheory}, we prove the necessary results on $z$-adic Hodge theory. In \S\ref{s:analyticlocalshtukas}, we define the analytic moduli problem, compare it with the formal moduli problem, and recall results of Fargues--Scholze \cite{FS21} on its intersection homology. In \S\ref{s:uniformizing}, we recall the global moduli problem and construct the uniformization morphism. In \S\ref{s:localglobalcompatibility}, we use this to prove Theorem A. In \S\ref{s:applications}, we use Theorem A to prove Theorem B, Theorem C, and Theorem D.

\subsection*{Notation}
Unless otherwise specified, all products are taken over $\bF_q$. The transition morphisms for our ind-schemes are required to be closed embeddings. We view all functors between derived categories as derived functors.

Starting in \S\ref{s:zadichodgetheory}, we freely use definitions from perfectoid geometry as in \cite{Sch17} and \cite{FS21}. When viewing an adic space $X$ as a locally ringed space, we take $\sO_X$ for its structure sheaf. For any adic space $X$ over $\bZ_p$, write $X^\Diamond$ for the associated v-sheaf over $\bF_p$ as in \cite[Lemma 18.1.1]{SW20}. 

\subsection*{Acknowledgements}
The author thanks Mark Kisin for his patience and advice. The author would also like to thank Michael Harris for giving a talk on \cite{GHSBP21} that motivated him to prove Theorem A, to thank Urs Hartl for helpful corrections, and to thank David Hansen for his interest and encouragement.

\section{Recollections on affine Grassmannians}\label{s:grassmannians}
In this section, we begin by setting up our local context. We then establish some notation on loop groups, Beilinson--Drinfeld affine Grassmannians, and their affine Schubert varieties, as well as recall basic facts about these objects. Nothing in this section is new.

\subsection{}\label{ss:globalizinggroupschemes}

Let $F$ be a local field of characteristic $p>0$, and write $\bF_q$ for its residue field. Fix a separable closure $\ov{F}$ of $F$, and write $\Ga_F$ for $\Gal(\ov{F}/F)$. Choose a uniformizer $z$ of $\cO_F$, which yields an identification $\cO_F=\bF_q\lb{z}$. Let $G$ be a parahoric group scheme over $\cO_F$ as in \cite[5.2.6]{BT84}.

It will be convenient to use the following globalization of our local setup, although we will see that our constructions are independent of this globalization.
\begin{lem*}
  There exists a geometrically connected smooth proper curve $C$ over $\bF_q$, a nonempty open subspace $U\subseteq C$, a parahoric group scheme $G_C$ over $C$ as in \cite[Definition 2.18]{Ric16}, a closed point $v$ of $C$, and an isomorphism $\wh\sO_{C,v}\cong\cO_F$ such that
  \begin{enumerate}[a)]
  \item $G_C|_U$ is reductive over $U$,
  \item $G_C|_{\cO_v}$ is identified with $G$ as group schemes over $\wh\sO_{C,v}\cong\cO_F$.
  \end{enumerate}
  Moreover, there exists an $\SL_h$-bundle $\sV$ on $C$ and a closed embedding
  \begin{align*}
    \io:G_C\ra\ul{\Aut}(\sV)
  \end{align*}
  of group schemes over $C$ such that $\ul{\Aut}(\sV)/G_C$ is quasi-affine over $C$.
\end{lem*}

\begin{proof}
By \cite[Lemma 3.1]{Ric16}, there exists a connected smooth curve $\mathring{C}$ over $\bF_q$, a smooth affine group scheme $\mathring{G}$ over $\mathring{C}$ with geometrically connected fibers, a closed point $v$ of $\mathring{C}$, and an isomorphism $\wh\sO_{C,v}\cong\cO_F$ such that $\mathring{G}|_{\mathring{C}\ssm v}$ is reductive over $\mathring{C}\ssm v$ and $\mathring{G}|_{\wh\sO_{C,v}}$ is identified with $G$ as group schemes over $\wh\sO_{C,v}\cong\cO_F$. Because $\mathring{C}$ has an $\bF_q$-point $v$, it is geometrically connected. Write $C$ for the associated smooth proper curve over $\bF_q$. Fpqc descent and \cite[5.1.9]{BT84} yield a parahoric group scheme $G_C$ over $C$ as in \cite[Definition 2.18]{Ric16} that extends $\mathring{G}$, so we can take $U=\mathring{C}\ssm v$. Finally, the last claim follows from \cite[Proposition 2.2(b)]{AH13}.
\end{proof}

\subsection{}\label{ss:globalBDgrassmannian}
Let us recall some facts about loop groups and affine Grassmannians. Let $I$ be a finite set, and let $S=\Spec{R}$ be an affine scheme over $C^I$. For all $i$ in $I$, write $\Ga_i$ for the graph of its $i$-th projection $S\ra C$, which is a relative effective Cartier divisor on $C\times S$.

Let $I_1,\dotsc,I_k$ be an ordered partition of $I$. Write $\wh\cO_C(S)$ for the ring of global sections of the completion of $\sO_{C\times S}$ along $\sum_{i\in I}\Ga_i$. For all $1\leq j\leq k$, write $\wh\cO^{j,\circ}_C(S)$ for the version that is punctured along $\sum_{i\in I_j}\Ga_i$.
\begin{defn*}\hfill
  \begin{enumerate}[a)]
  \item Write $L^n_I(G_C)$, $L^+_I(G_C)$, and $L^{j,\circ}_I(G_C)$ for the sheaves over $C^I$ given by sending $S$ to $G_C(\sO_{n\sum_{i\in I}\Ga_i})$, $G_C(\wh\cO_C(S))$, and $G_C(\wh\cO^{j,\circ}_C(S))$, respectively.

  \item Write $\Gr^{(I_1,\dotsc,I_k)}_{G_C}$ for the sheaf over $C^I$ whose $S$-points parametrize data consisting of
  \begin{enumerate}[i)]
  \item for all $1\leq j\leq k$, a $G$-bundle $\sG_j$ on $\Spec\wh\cO_C(S)$,
  \item for all $1\leq j\leq k$, an isomorphism of $G$-bundles
    \begin{align*}
      \phi_j:\sG_j|_{\Spec\wh\cO^{j,\circ}_C(S)}\ra^\sim\sG_{j+1}|_{\Spec\wh\cO^{j,\circ}_C(S)},
    \end{align*}
    where $\sG_{k+1}$ denotes the trivial $G$-bundle.
  \end{enumerate}
\end{enumerate}
Write $L^+_zG$ and $L_zG$ for the fiber at $v$ of $L^+_*(G_C)$ and $L^{1,\circ}_*(G_C)$, respectively, where $*$ denotes the singleton set. Also, write $\Gr^k_{z,G}$ for the fiber at $v^I$ of $\Gr^{(\{1\},\dotsc,\{k\})}_{G_C}$.
\end{defn*}
The proof of \cite[Lemma 3.2]{HR20} shows that $L^n_I(G_C)$ is of finite type and affine over $C^I$, so $L^+_I(G_C)=\textstyle\varprojlim_nL^n_I(G_C)$ is affine over $C^I$. Recall that $L^{j,\circ}_I(G_C)$ is ind-affine over $C^I$ \cite[Lemma 3.2(i)]{HR20}, and $\Gr^{(I_1,\dotsc,I_k)}_{G_C}$ is ind-projective over $C^I$ \cite[Proposition 3.12]{AH13}. Also, note that $L_z^+G$, $L_zG$, and $\Gr^k_{z,G}$ are independent of the globalization from Lemma \ref{ss:globalizinggroupschemes}.

\subsection{}\label{ss:formaldisk}
The following lemmas give an alternative description of the Beilinson--Drinfeld affine Grassmannian after completing at a point. Write $\bD$ for the formal scheme $\Spf\cO_F$, and recall that $\Spec$ yields an anti-equivalence from the category of $\bF_q\lb{\ze_i}_{i\in I}$-algebras where the $\ze_i$ are nilpotent to the category of affine schemes over $\bD^I$. Let $S=\Spec{R}$ be an affine scheme over $\bD^I$. 
\begin{lem*}
The direct system $(n\sum_{i\in I}\Ga_i)_{n\geq0}$ of schemes over $C\times S$ is naturally isomorphic to $(nv\times S)_{n\geq0}$. Consequently, $\wh\cO_C(S)$ is naturally isomorphic to $R\lb{z}$, and $\wh\cO^{j,\circ}_C(S)=\wh\cO_C(S)[\textstyle\frac1{z-\ze_i}]_{i\in I_j} = R\lb{z}[\frac1{z-\ze_i}]_{i\in I_j}$ is naturally isomorphic to $R\lp{z}$.
\end{lem*}
\begin{proof}
  As nilpotent thickenings are \'etale-local and $C$ is smooth at $v$, it suffices to replace $C$ with $\bA^1_{\bF_q}=\Spec\bF_q[z]$ and $v$ with the origin. Then $n\sum_{i\in i}\Ga_i$ is the vanishing locus of $\prod_{i\in I}(z-\ze_i)^n$ in $C\times S=\Spec R[z]$, and $nv\times S$ is the vanishing locus of $z^n$ in $C\times S$. Choose positive integers $n_i$ such that $\ze_i^{n_i}=0$ in $R$.

Set $N_1\coloneqq\sum_{i\in I}n+n_i-1$ and $N_2\coloneqq n+\max_{i\in I}\{n_i\}-1$. On $n\sum_{i\in I}\Ga_i$, we have
\begin{align*}
  z^{N_1} = \prod_{i\in I}((z-\ze_i)+\ze_i)^{n+n_i-1} = \prod_{i\in I}\sum_{l=0}^{n+n_i-1}\textstyle\binom{n+n_i-1}l(z-\ze_i)^l\ze_i^{n+n_i-1-l}= 0,
\end{align*}
so $n\sum_{i\in I}\Ga_i$ lies in $N_1v\times S$. Conversely, on $nv\times S$, we see that
\begin{align*}
\prod_{i\in I}(z-\ze_i)^{N_2} = \prod_{i\in I}\sum_{l=0}^{N_2}\textstyle\binom{N_2}lz^l\ze_i^{N_2-l}=0,
\end{align*}
so $nv\times S$ lies in $N_2\sum_{i\in I}\Ga_i$.
\end{proof}

\subsection{}\label{ss:localBDgrassmannian}
Write $\wh\Gr^{(I_1,\dotsc,I_k)}_G$ for the formal completion of $\Gr^{(I_1,\dotsc,I_k)}_{G_C}$ along $v^I$ in $C^I$.
\begin{lem*}
  Our $\wh\Gr^{(I_1,\dotsc,I_k)}_G$ is an ind-projective ind-scheme over $\bD^I$, and it is naturally isomorphic to $\Gr^k_{z,G}|_{\bD^I}$.
\end{lem*}
Thus $\wh\Gr^{(I_1,\dotsc,I_k)}_G$ is independent of the globalization from Lemma \ref{ss:globalizinggroupschemes}.
\begin{proof}
This follows immediately from Lemma \ref{ss:formaldisk}.
\end{proof}

\subsection{}\label{ss:affineschubertvarieties}
We now introduce affine Schubert varieties. Write $\ov{\bF_q(C)}$ for the separable closure of $\bF_q(C)$ in $\ov{F}$, and write $\Ga_{\bF_q(C)}$ for $\Gal(\ov{\bF_q(C)}/\bF_q(C))$. Let $T$ be a maximal subtorus of $G_C|_{\bF_q(C)}$, and write $X_*^+(T)$ for the set of dominant cocharacters of $T_{\ov{\bF_q(C)}}$ with respect to a fixed Borel subgroup $B\subseteq G_C|_{\ov{\bF_q(C)}}$ containing $T_{\ov{\bF_q(C)}}$. Identify $X_*^+(T)$ with the set of conjugacy classes of cocharacters of $G_C|_{\ov{\bF_q(C)}}$.

Let $\mu_\bullet=(\mu_i)_{i\in I}$ be in $X_*^+(T)^I$. Identify the field of definition of $\mu_i$ with $\bF_q(C_i)$ for some finite cover $C_i\ra C$ that is \'etale over $U$, and write $U_i$ for the preimage of $U$ in $C_i$. Note that the closure $F_i$ of $\bF_q(C_i)$ in $\ov{F}$ equals the completion of $\bF_q(C_i)$ at the closed point $v_i$ of $C_i$ above $v$ induced by $\ov{\bF_q(C)}\ra\ov{F}$. Write $\bD_i$ for $\Spf\cO_{F_i}$.

\begin{defn*}\hfill
  \begin{enumerate}[a)]
  \item Write $\Gr^{(I_1,\dotsc,I_k)}_{G_C,\mu_\bullet}|_{\prod_{i\in I}U_i}\subseteq\Gr^{(I_1,\dotsc,I_k)}_{G_C}|_{\prod_{i\in I}U_i}$ for the associated closed affine Schubert variety, and write $\Gr^{(I_1,\dotsc,I_k)}_{G_C,\mu_\bullet}|_{\prod_{i\in I}C_i}$ for its closure in $\Gr^{(I_1,\dotsc,I_k)}_{G_C}|_{\prod_{i\in I}C_i}$.
  \item Write $\wh\Gr^{(I_1,\dotsc,I_k)}_{G,\mu_\bullet}|_{\prod_{i\in I}\bD_i}$ for the formal completion of $\Gr^{(I_1,\dotsc,I_k)}_{G_C,\mu_\bullet}|_{\prod_{i\in I}C_i}$ along $\textstyle\prod_{i\in I}v_i$ in $\textstyle\prod_{i\in I}C_i$.
  \item When $I=*$, write $\Gr^1_{z,G,\mu}|_{v_*}$ for the fiber at $v_*$ of $\Gr^{(*)}_{G_C,\mu}|_{C_*}$.
  \end{enumerate}

\end{defn*}
Recall that $\Gr^{(I_1,\dotsc,I_k)}_{G_C,\mu_\bullet}|_{\prod_{i\in I}C_i}$ is a projective scheme over $\prod_{i\in I}C_i$, and the natural $L^+_I(G_C)$-action on $\Gr^{(I_1,\dotsc,I_k)}_{G_C,\mu_\bullet}|_{\prod_{i\in I}C_i}$ factors through $L^n_I(G_C)$ for large enough $n$ \cite[Proposition 1.10]{Laf16}\footnote{While \cite[Proposition 1.10]{Laf16} only treats the case of split $G$, it extends to the general case. Indeed, this is already implicitly used in \cite[(12.10)]{Laf16}.}. Therefore $\wh\Gr^{(I_1,\dotsc,I_k)}_{G,\mu_\bullet}|_{\prod_{i\in I}\bD_i}$ is a formal scheme that is formally of finite type and adic over $\prod_{i\in I}\bD_i$, and its special fiber is projective over $\textstyle\prod_{i\in I}v_i$. Also, the proof of \cite[Lemma 3.2]{Zhu14} shows that $\wh\Gr^{(I_1,\dotsc,I_k)}_{G,\mu_\bullet}|_{\prod_{i\in I}\bD_i}$ is independent of the globalization from Lemma \ref{ss:globalizinggroupschemes}.

\subsection{}\label{ss:SLrisogenybound}
Recall that we have an isomorphism
\begin{align*}
\Gr^k_{z,G}\ra^\sim(\Gr^1_{z,G})^k
\end{align*}
given by $((\sG_j)_{j=1}^k,(\phi_j)_{j=1}^k)\mapsto((\sG_k,\phi_k),\dotsc,(\sG_1,\phi_k\circ\dotsb\circ\phi_1))$.
\begin{defn*}
  Under this identification, write $\Gr^k_{z,\SL_h,m}$ for the closed subsheaf of $\Gr^k_{z,\SL_h}$ corresponding to $(\Gr^1_{z,\SL_h,m2\rho^\vee})^k\subseteq(\Gr^1_{z,\SL_h})^k$, where $2\rho^\vee$ denotes the sum of positive coroots in $\SL_h$.
\end{defn*}
By \ref{ss:affineschubertvarieties}, we see that $\Gr^k_{z,\SL_h,m}$ is a projective scheme over $\bF_q$.

\subsection{}\label{ss:SLrloopgroup}
We conclude by showing that, after pulling back to the loop group, affine Schubert varieties are affine. Write $L_I(G_C)_{\mu_\bullet}|_{\prod_{i\in I}C_i}$ for the pullback of
\begin{align*}
\Gr^{(I_1,\dotsc,I_k)}_{G_C,\mu_\bullet}|_{\prod_{i\in I}C_i}
\end{align*}
under the natural morphism $\textstyle\prod_{j=1}^kL_I^{j,\circ}(G_C)\ra\Gr^{(I_1,\dotsc,I_k)}_{G_C}$.
\begin{lem*}
Our $L_I(G_C)_{\mu_\bullet}|_{\prod_{i\in I}C_i}$ is affine over $\prod_{i\in I}C_i$.
\end{lem*}
\begin{proof}
  Because $\ul{\Aut}(\sV)/G_C$ is quasi-affine over $C$, the proof of \cite[Proposition 1.2.6]{Zhu17} shows that $\io_*:\Gr^{(I_1,\dotsc,I_k)}_{G_C}\ra\Gr^{(I_1,\dotsc,I_k)}_{\SL_{h,C}}$ is a locally closed embedding. Now \ref{ss:affineschubertvarieties} indicates that $\Gr^{(I_1,\dotsc,I_k)}_{G_C,\mu_\bullet}|_{\prod_{i\in I}C_i}$ is a quasi-compact scheme, so \cite[Lemma 5.4]{HV11} implies that its image under $\io_*$ lies in $\Gr^{(I_1,\dotsc,I_k)}_{\SL_{h,C},(m2\rho^\vee)_{i\in I}}|_{\prod_{i\in I}C_i}$ for large enough $m$. Since $\Gr^{(I_1,\dotsc,I_k)}_{G_C,\mu_\bullet}|_{\prod_{i\in I}C_i}$ is projective over $\prod_{i\in I}C_i$ by \ref{ss:affineschubertvarieties} and $\io_*$ is a monomorphism, we see that $\io_*:\Gr^{(I_1,\dotsc,I_k)}_{G_C,\mu_\bullet}|_{\prod_{i\in I}C_i}\ra\Gr^{(I_1,\dotsc,I_k)}_{\SL_{h,C},(m2\rho^\vee)_{i\in I}}|_{\prod_{i\in I}C_i}$ is a closed embedding. Combined with the fact that $L_I^+(G_C)\ra L_I^+(\SL_{r,C})$ is a closed embedding, this implies that $L_I(G_C)_{\mu_\bullet}|_{\prod_{i\in I}C_i}\ra L_I(\SL_{h,C})_{(m2\rho^\vee)_{i\in I}}|_{\prod_{i\in I}C_i}$ is also a closed embedding. Now the argument in the proof of \cite[Lemma 4.23]{EH14} shows that $L_I(\SL_{h,C})_{(m2\rho^\vee)_{i\in I}}$ is affine over $C^I$, so $L_I(G_C)_{\mu_\bullet}|_{\prod_{i\in I}C_i}$ is affine over $\textstyle\prod_{i\in I}C_i$.
\end{proof}

\section{Formal moduli of local shtukas}\label{s:formallocalshtukas}
To define the uniformization morphism via Beauville--Laszlo gluing in \S\ref{s:uniformizing}, we need a formal variant of the moduli of local shtukas. Moreover, to show that the uniformization morphism is \'etale (after passing to generic fibers), we need some finitude properties of this formal moduli. Accomplishing these tasks is the goal of this section.

We start by defining local shtukas and their quasi-isogenies in the formal setting. After proving a rigidity property for quasi-isogenies, we define the formal moduli problem, and we dedicate the rest of this section to proving that it gives a formal scheme that is locally formally of finite type over $\bD^I$.

Our strategy ultimately harks back to Rapoport--Zink's proof \cite{RZ96} of the analogous property for Rapoport--Zink spaces. The equicharacteristic incarnation of this argument is heavily based on work of Hartl--Viehmann \cite{HV11} and Arasteh Rad--Hartl \cite{EH14}, although we generalize their results to the case of arbitrarily many legs.

\subsection{}\label{ss:localshtukas}
Later, it will be useful to work in the following generality. Let $R$ be a topological $\bF_q\lb{\ze_i}_{i\in I}$-algebra that is adic with finitely generated ideal of definition, and write $S\coloneqq\Spec{R}$. Write $\tau:S\ra S$ for the absolute $q$-Frobenius endomorphism. By abuse of notation, we also write $\tau: R\lb{z}\ra R\lb{z}$ for the canonical lift of absolute $q$-Frobenius. We use $\prescript\tau{}{(-)}$ to denote pullback by $\tau$.

Write $R\ba{z,\textstyle\frac1z}$ for the completion of $R\lp{z}$ with respect to the topology induced from $R$. We now define \emph{local $G$-shtukas}.
\begin{defn*}\hfill
  \begin{enumerate}[a)]
  \item   A \emph{local $G$-shtuka} over $S$ consists of
  \begin{enumerate}[i)]
  \item for all $1\leq j\leq k$, a $G$-bundle $\sG_j$ on $\Spec R\lb{z}$,
  \item for all $1\leq j\leq k$, an isomorphism of $G$-bundles
    \begin{align*}
      \phi_j:\sG_j|_{\Spec R\lb{z}[\frac1{z-\ze_i}]_{i\in I_j}}\ra^\sim\sG_{j+1}|_{\Spec R\lb{z}[\frac1{z-\ze_i}]_{i\in I_j}},
    \end{align*}
    where $\sG_{k+1}$ denotes the $G$-bundle $\prescript\tau{}{\sG_1}$.
  \end{enumerate}
\item Suppose that $\Spf{R}$ lies over $\prod_{i\in I}\bD_i$, and let $\sG=((\sG_j)_{j=1}^k,(\phi_j)_{j=1}^k)$ be a local shtuka over $S$. We say that $\sG$ is \emph{bounded by $\mu_\bullet$} if, for any affine \'etale cover $\Spf\wt{R}\ra\Spf{R}$ such that $\prescript\tau{}{\sG_1}|_{\Spec\wt{R}\lb{z}}$ is trivial and any trivialization $t:\prescript\tau{}{\sG_1}|_{\Spec\wt{R}\lb{z}}\ra^\sim G$, the $\Spf\wt{R}$-point of $\wh\Gr^{(I_1,\dotsc,I_k)}_G|_{\prod_{i\in I}\bD_i}$ given by
  \begin{align*}
        \xymatrixcolsep{0.75in}
    \xymatrix{\sG_1|_{\Spec\wt{R}\lb{z}}\ar@{-->}[r]^-{(\phi_1)_{\wt{R}\ba{z,\frac1z}}} & \dotsm\ar@{-->}[r]^-{(\phi_{k-1})_{\wt{R}\ba{z,\frac1z}}} & \sG_k|_{\Spec\wt{R}\lb{z}}\ar@{-->}[r]^-{(t\circ\phi_k)_{\wt{R}\ba{z,\frac1z}}} & G}
  \end{align*}
lies in $\wh\Gr^{(I_1,\dotsc,I_k)}_{G,\mu_\bullet}|_{\prod_{i\in I}\bD_i}$, using the description of $\wh\Gr^{(I_1,\dotsc,I_k)}_G$ from Lemma \ref{ss:localBDgrassmannian}.
  \end{enumerate}
\end{defn*}
It suffices to check Definition \ref{ss:localshtukas}.b) for a single $\Spf\wt{R}\ra \Spf{R}$ and $t$.

\subsection{}\label{ss:quasiisogenies}
For the rest of this section, assume that $R$ is discrete, so that the $\ze_i$ are nilpotent in $R$. In this setting, we use the following notion of quasi-isogenies.
\begin{defn*}Let $\sG$ and $\sG'$ be local $G$-shtukas over $S$.
  \begin{enumerate}[a)]
  \item A \emph{quasi-isogeny} from $\sG$ to $\sG'$ consists of, for all $1\leq j\leq k$, an isomorphism of $G$-bundles
      \begin{align*}
     \de_j:\sG_j|_{\Spec R\lp{z}}\ra^\sim \sG_j'|_{\Spec R\lp{z}}
    \end{align*}
    such that the diagram
    \begin{align*}
      \xymatrix{\sG_j|_{\Spec R\lp{z}}\ar[r]^-{\phi_j}\ar[d]^-{\de_j} & \sG_{j+1}|_{\Spec R\lp{z}}\ar[d]^-{\de_{j+1}}\\
      \sG_j'|_{\Spec R\lp{z}}\ar[r]^-{\phi_j'} & \sG'_{j+1}|_{\Spec R\lp{z}}}
    \end{align*}
    commutes, where $\de_{k+1}$ denotes the isomorphism $\prescript\tau{}{\de_1}$, and we use Lemma \ref{ss:formaldisk} to identify $R\lb{z}[\textstyle\frac1{z-\ze_i}]_{i\in I_j}$ with $R\lp{z}$.
\item Let $m$ be a non-negative integer, and let $\de$ be a quasi-isogeny from $\sG$ to $\sG'$. We say that $\de$ is \emph{bounded by $m$} if, for all $1\leq j\leq k$, the morphism $\io_*(\de_j)$ yields a point of $[L^+_z\SL_h\!\bs\!\Gr^1_{z,\SL_h,m2\rho^\vee}]$.
  \end{enumerate}
\end{defn*}
Since $L^+_zG$-bundles on $\Spec R$ are trivial after an \'etale cover, \cite[Lemma 5.4]{HV11} implies that any quasi-isogeny is bounded by $m$ for large enough $m$.

\subsection{}\label{ss:rigidisogeny}
We will need the following quantitative version of the rigidity of quasi-isogenies. Let $J$ be an ideal of $R$ satisfying $J^n=0$, and write $\j:\ov{S}\ra S$ for the associated closed embedding.
\begin{prop*}
  For all local $G$-shtukas $\sG$ and $\sG'$ over $S$, pullback yields a bijection
  \begin{align*}
\{\mbox{quasi-isogenies from }\sG\mbox{ to }\sG'\}\ra^\sim\{\mbox{quasi-isogenies from }\j^*\sG\mbox{ to }\j^*\sG'\}.
  \end{align*}
Moreover, suppose that $S$ lies over $\prod_{i\in I}\bD_i$ and that $\sG$ and $\sG'$ are bounded by $\mu_\bullet$. There exists a non-negative integer $B$ such that, if $\j^*\de$ is bounded by $m$, then $\de$ is bounded by $m+B\lceil\log_qn\rceil$.
\end{prop*}
\begin{proof}
  By induction, it suffices to consider the $n=q$ case. There $\tau:S\ra S$ factors as $\j\circ\i$ for a unique morphism $\i:S\ra\ov{S}$, so for any quasi-isogeny $\de$ from $\sG$ to $\sG'$, we get $\prescript\tau{}{\de_1}=\i^*\j^*\de_1$. Hence the commutative square
  \begin{align*}
          \xymatrix{\sG_k|_{\Spec R\lp{z}}\ar[r]^-{\phi_k}\ar[d]^-{\de_k} & \prescript\tau{}{\sG_1}|_{\Spec R\lp{z}}\ar[d]^-{\prescript\tau{}{\de_1}}\\
      \sG_k'|_{\Spec R\lp{z}}\ar[r]^-{\phi_k'} & \prescript\tau{}{\sG'_1}|_{\Spec R\lp{z}}},
  \end{align*}
  enables us to recover $\de_k$ from $\j^*\de_1$, where we use Lemma \ref{ss:formaldisk} to identify $R\lb{z}[\textstyle\frac1{z-\ze_i}]_{i\in I_k}$ with $R\lp{z}$. From here, we similarly recover $\de_j$ for $1\leq j\leq k-1$, showing that pullback by $\j$ is injective on quasi-isogenies. Considering the same squares over $\ov{S}$ also shows that pullback by $\j$ is surjective on quasi-isogenies.

Next, suppose that $S$ lies over $\prod_{i\in I}\bD_i$ and that $\sG$ and $\sG'$ are bounded by $\mu_\bullet$. If $\j^*\de$ is bounded by $m$, then its pullback $\i^*\j^*\de_1=\prescript\tau{}{\de_1}$ is as well. Because $\phi_k$ and $\phi_k'$ are bounded by $(\mu_i)_{i\in I_k}$, where the relative position bound is taken with respect to the $(z-\ze_i)$ for $i$ in $I_k$, a quasi-compactness argument shows that there exists a non-negative integer $B$ such that $\de_k$ is bounded by $m+B$. For $1\leq j\leq k-1$, applying the same argument to $\de_j$ indicates that, after increasing $B$ by an amount depending only on $\mu_\bullet$, our $\de_j$ is also bounded by $m+B$.
\end{proof}

\subsection{}\label{ss:formalmoduli}
We now define the formal moduli of local $G$-shtukas.
\begin{defn*}
  Write $\fLocSht_{G}^{(I_1,\dotsc,I_k)}$ for the sheaf over $\bD^I$ whose $S$-points parametrize data consisting of
  \begin{enumerate}[i)]
  \item a local $G$-shtuka $\sG$ over $S$,
  \item a quasi-isogeny $\de$ from $\sG$ to the trivial local $G$-shtuka $G=((G)_{j=1}^k,(\id)_{j=1}^k)$.
  \end{enumerate}
  Write $\fLocSht^{(I_1,\dotsc,I_k)}_{G,\mu_\bullet}|_{\prod_{i\in I}\bD_i}$ for the subsheaf of $\fLocSht^{(I_1,\dotsc,I_k)}_G|_{\prod_{i\in I}\bD_i}$ whose $S$-points consist of the $(\sG,\de)$ such that $\sG$ is bounded by $\mu_\bullet$.

  Write $f^\fL:\fLocSht^{(I_1,\dotsc,I_k)}_{G,\mu_\bullet}|_{\prod_{i\in I}\bD_i}\ra\textstyle\prod_{i\in I}\bD_i$ for the structure morphism.
\end{defn*}
\begin{prop}\label{lem:unboundedformalmoduli}
Our $\fLocSht^{(I_1,\dotsc,I_k)}_{G}$ is naturally isomorphic to $\Gr^k_{z,G}|_{\bD^I}$ over $\bD^I$, and $\fLocSht^{(I_1,\dotsc,I_k)}_{G,\mu_\bullet}|_{\prod_{i\in I}\bD_i}$ is a closed subsheaf of $\fLocSht^{(I_1,\dotsc,I_k)}_G|_{\prod_{i\in I}\bD_i}$. 
\end{prop}
\begin{proof}
  In Definition \ref{ss:formalmoduli}, note that i) and ii) are uniquely determined by $(\sG_j)_{j=1}^k$, $(\phi_j)_{j=1}^{k-1}$, and $\de_k$. This is precisely the data parametrized by $\Gr^k_{z,G}|_{\bD^I}$, which proves the first claim. The second claim follows from the fact that $\wh\Gr^{(I_1,\dotsc,I_k)}_{G,\mu_\bullet}|_{\prod_{i\in I}\bD_i}$ is a closed subsheaf of $\wh\Gr^{(I_1,\dotsc,I_k)}_G|_{\prod_{i\in I}\bD_i}$; see the proof of \cite[Proposition 4.11]{EH14}. 
\end{proof}

\subsection{}\label{ss:formalisogenybound}
First, we naively stratify $\fLocSht^{(I_1,\dotsc,I_k)}_{G,\mu_\bullet}|_{\prod_{i\in I}\bD_i}$ by bounding the quasi-isogeny. Write $\fLocSht^{(I_1,\dotsc,I_k)}_{G,\mu_\bullet,m}|_{\prod_{i\in I}\bD_i}$ for the subsheaf of $\fLocSht^{(I_1,\dotsc,I_k)}_{G,\mu_\bullet}|_{\prod_{i\in I}\bD_i}$ whose $S$-points consist of the $(\sG,\de)$ such that $\de$ is bounded by $m$.
\begin{prop*}
Our $\fLocSht^{(I_1,\dotsc,I_k)}_{G,\mu_\bullet,m}|_{\prod_{i\in I}\bD_i}$ is a formal scheme that is formally of finite type and adic over $\textstyle\prod_{i\in I}\bD_i$, its reduced subscheme is projective over $\textstyle\prod_{i\in I}v_i$, and $\fLocSht^{(I_1,\dotsc,I_k)}_{G,\mu_\bullet}|_{\prod_{i\in I}\bD_i}$ equals the direct limit $\varinjlim_m\fLocSht^{(I_1,\dotsc,I_k)}_{G,\mu_\bullet,m}|_{\prod_{i\in\bD_i}}$.
\end{prop*}
\begin{proof}
  Note that we have a Cartesian square
  \begin{align*}
  \xymatrix{\fLocSht^{(I_1,\dotsc,I_k)}_{G,\mu_\bullet,m}|_{\prod_{i\in I}\bD_i}\ar[r]\ar[d]^{\io_*} & \fLocSht^{(I_1,\dotsc,I_k)}_{G,\mu_\bullet}|_{\prod_{i\in I}\bD_i}\ar[d]^-{\io_*}\\
  \Gr^k_{z,\SL_h,m}|_{\prod_{i\in I}\bD_i}\ar[r] & \Gr^k_{z,\SL_h}|_{\prod_{i\in I}\bD_i},
}
  \end{align*}
  where we use Proposition \ref{lem:unboundedformalmoduli} to identify $\fLocSht^{(I_1,\dotsc,I_k)}_{\SL_h}|_{\prod_{i\in I}\bD_i}$ with $\Gr^k_{z,\SL_h}|_{\prod_{i\in I}\bD_i}$. Because $\SL_h\!/G$ is quasi-affine over $\cO_F$, Proposition \ref{lem:unboundedformalmoduli} and \cite[Proposition 1.2.6]{Zhu17} show that $\fLocSht^{(I_1,\dotsc,I_k)}_{G,\mu_\bullet}|_{\prod_{i\in I}\bD_i}\ra\Gr^k_{z,\SL_h}|_{\prod_{i\in I}\bD_i}$ is a closed embedding. Therefore its pullback $\fLocSht^{(I_1,\dotsc,I_k)}_{G,\mu_\bullet,m}|_{\prod_{i\in I}\bD_i}\ra\Gr^k_{z,\SL_h,m}|_{\prod_{i\in I}\bD_i}$ is as well. Since
  \begin{align*}
    \Gr^k_{z,\SL_h,m}|_{\prod_{i\in I}nv_i}
  \end{align*}
  is projective over $\prod_{i\in I}nv_i$ for any positive integer $n$ by \ref{ss:SLrisogenybound}, the same holds for $\fLocSht^{(I_1,\dotsc,I_k)}_{G,\mu_\bullet,m}|_{\prod_{i\in I}nv_i}$. Now the underlying topological space of $\Gr^k_{z,\SL_h,m}|_{\prod_{i\in I}nv_i}$ is independent of $n$, so the $\fLocSht^{(I_1,\dotsc,I_k)}_{G,\mu_\bullet,m}|_{\prod_{i\in I}nv_i}$ have this property too. From here, \cite[(\textbf{1}, 10.6.4)]{GD71} indicates that $\fLocSht^{(I_1,\dotsc,I_k)}_{G,\mu_\bullet,m}|_{\prod_{i\in I}\bD_i}$ is a noetherian formal scheme that is adic over $\prod_{i\in I}\bD_i$. Therefore its reduced subscheme equals that of $\fLocSht^{(I_1,\dotsc,I_k)}_{G,\mu_\bullet,m}|_{\prod_{i\in I}v_i}$, which is projective over $\prod_{i\in I}v_i$, so $\fLocSht^{(I_1,\dotsc,I_k)}_{G,\mu_\bullet,m}|_{\prod_{i\in I}\bD_i}$ is formally of finite type over $\prod_{i\in I}\bD_i$. Finally, last statement follows from $\Gr^k_{z,\SL_h}$ equaling the direct limit $\varinjlim_m\Gr^k_{z,\SL_h,m}$.
\end{proof}

\subsection{}\label{ss:nilpotentalgebraization}
To obtain a better stratification of $\fLocSht^{(I_1,\dotsc,I_k)}_{G,\mu_\bullet}|_{\prod_{i\in I}\bD_i}$, we need the following algebraization lemma. Briefly, relax our assumption that $R$ is discrete, since we will also use this lemma later. Let $(S_l)_{l\geq0}$ be a direct system of affine schemes $S_l=\Spec R_l$ over $\prod_{i\in I}\bD_i$ such that
\begin{enumerate}[i)]
\item the morphisms $S_l\ra S_{l'}$ are closed embeddings,
\item the associated ideals $\ker(R_{l'}\ra R_l)$ are nilpotent.
\end{enumerate}
Take $R$ to be the ring $\textstyle\varprojlim_lR_l$, and endow $R$ with a topological ring structure such that $\bF_q\lb{\ze_i}_{i\in I}\ra R$ is continuous, the $R\ra R_l$ are continuous for the discrete topology on $R_l$, and $R$ is adic with finitely generated ideal of definition.
\begin{lem*}
Pullback yields an equivalence of groupoids
\begin{align*}
  \left\{\begin{tabular}{c}
    local $G$-shtukas over \\
    $S$ bounded by $\mu_\bullet$
         \end{tabular}\right\}\lra^\sim \varprojlim_l
  \left\{\begin{tabular}{c}
    local $G$-shtukas over \\
   $S_l$ bounded by $\mu_\bullet$
  \end{tabular}\right\}.
\end{align*}
\end{lem*}
\begin{proof}
  Let $(\sG^l)_{l\geq0}$ be a compatible system of local $G$-shtukas over $S_l$ bounded by $\mu_\bullet$. We can form the $G$-bundles $\sG_j\coloneqq\varprojlim_l\sG^l_j$ on $\Spec R\lb{z}$, so now we just need to form the isomorphisms $\phi_j$.

  Let $\Spec\wt{R}_0\ra S_0$ be an affine \'etale cover where $\sG^0_j|_{\Spec\wt{R}_0\lb{z}}$ is trivial for all $1\leq j\leq k$, and fix trivializations of the $\sG_j^0|_{\Spec\wt{R}_0\lb{z}}$. By ii), there exists a unique affine \'etale cover $\Spec\wt{R}_l\ra S_l$ whose pullback to $S_0$ is $\Spec\wt{R}_0$, and there also exist compatible systems of trivializations of the $\sG_j^l|_{\Spec\wt{R}_l\lb{z}}$ \cite[Proposition 2.2(c)]{HV11}\footnote{While \cite{HV11} only treats split reductive $G$, the proof immediately adapts to any smooth $G$.}. Under these identifications, the $(\phi_j^l)_{\wt{R}_l\lp{z}}$ correspond to compatible systems of $b_j^l$ in $G(\wh\cO^{j,\circ}_C(\Spec\wt{R}_l))$, where we use Lemma \ref{ss:formaldisk} to identify $\wt{R}_l\lp{z}$ with $\wh\cO^{j,\circ}_C(\Spec\wt{R}_l)$.

For all $i$ in $I$, let $V_i$ be an affine neighborhood of $v_i$ in $C_i$. Because the $\sG^l$ are bounded by $\mu_\bullet$, our $(b^l_j)_{j=1}^k$ yield $\wt{R}_l$-points of $L_I(G_C)_{\mu_\bullet}|_{\prod_{i\in I}V_i}$. The latter is affine by Lemma \ref{ss:SLrloopgroup}, so the compatible system of $(b^l_j)_{j=1}^k$ yields an $\wt{R}\coloneqq\varprojlim_l\wt{R}_l$-point $(b_j)_{j=1}^k$ of $L_I(G_C)|_{\prod_{i\in I}V_i}$. By construction, the resulting local $G$-shtuka $\wt\sG\coloneqq((\sG_j|_{\Spec\wt{R}\lb{z}})_{j=1}^k,(b_j)_{j=1}^k)$ over $\Spec\wt{R}$ is bounded by $\mu_\bullet$. Since the $(\phi^l_j)_{\wt{R}_l\lp{z}}$ and thus $b_j^l$ are compatible with the descent data of $\sG^l_j$ from $\Spec\wt{R}_l$ to $S_l$, we see that the $b_j$ are compatible with the descent data of $\sG_j$ from $\Spec\wt{R}$ to $S$. Hence $\wt\sG$ naturally descends to a local $G$-shtuka $\sG$ over $S$ bounded by $\mu_\bullet$, as desired.
\end{proof}

\subsection{}\label{ss:mpartformalscheme}
Resume our assumption that $R$ is discrete. The following stratification of $\fLocSht^{(I_1,\dotsc,I_k)}_{G,\mu_\bullet}|_{\prod_{i\in I}\bD_i}$ has the advantage of being closed under formal completion. Write $\fLocSht^{(I_1,\dotsc,I_k)}_{G,\mu_\bullet,\wh{m}}|_{\prod_{i\in I}\bD_i}$ for the formal completion of $\fLocSht^{(I_1,\dotsc,I_k)}_{G,\mu_\bullet}|_{\prod_{i\in I}\bD_i}$ along the reduced subscheme of $\fLocSht^{(I_1,\dotsc,I_k)}_{G,\mu_\bullet,m}|_{\prod_{i\in I}\bD_i}$.
\begin{lem*}
Our $\fLocSht^{(I_1,\dotsc,I_k)}_{G,\mu_\bullet,\wh{m}}|_{\prod_{i\in I}\bD_i}$ is a formal scheme that is formally of finite type over $\prod_{i\in I}\bD_i$.
\end{lem*}
\begin{proof}
Proposition \ref{ss:formalisogenybound} and \cite[Lemma 5.4]{HV11} imply that $\fLocSht^{(I_1,\dotsc,I_k)}_{G,\mu_\bullet,\wh{m}}|_{\prod_{i\in I}\bD_i}$ equals the direct limit $\varinjlim_l\fLocSht^{(I_1,\dotsc,I_k)}_{G,\mu_\bullet,\wh{m},l}|_{\prod_{i\in I}\bD_i}$, where $\fLocSht^{(I_1,\dotsc,I_k)}_{G,\mu_\bullet,\wh{m},l}|_{\prod_{i\in I}\bD_i}$ denotes the formal completion of $\fLocSht^{(I_1,\dotsc,I_k)}_{G,\mu_\bullet,m+l}|_{\prod_{i\in I}\bD_i}$ along the reduced subscheme of $\fLocSht^{(I_1,\dotsc,I_k)}_{G,\mu_\bullet,m}|_{\prod_{i\in I}\bD_i}$. The reduced subscheme of $\fLocSht^{(I_1,\dotsc,I_k)}_{G,\mu_\bullet,m}|_{\prod_{i\in I}\bD_i}$ is quasi-compact by Proposition \ref{ss:formalisogenybound}, so it is covered by finitely many affine open subschemes $U$. Proposition \ref{ss:formalisogenybound} indicates that $\fLocSht^{(I_1,\dotsc,I_k)}_{G,\mu_\bullet,\wh{m},l}|_{\prod_{i\in I}\bD_i}$ is a noetherian formal scheme with the same reduced subscheme as $\fLocSht^{(I_1,\dotsc,I_k)}_{G,\mu_\bullet,m}|_{\prod_{i\in I}\bD_i}$, so $U$ corresponds to an affine open formal subscheme $\fU_l=\Spf A_l$ of $\fLocSht^{(I_1,\dotsc,I_k)}_{G,\mu_\bullet,\wh{m},l}|_{\prod_{i\in I}\bD_i}$.

The above shows that $\varinjlim_l\fU_l$ is an open subsheaf of $\fLocSht^{(I_1,\dotsc,I_k)}_{G,\mu_\bullet,\wh{m}}|_{\prod_{i\in I}\bD_i}$. Thus it suffices to prove that $\varinjlim_l\fU_l$ is an affine formal scheme that is formally of finite type over $\prod_{i\in I}\bD_i$. Because the inclusion morphisms
  \begin{align*}
    \fLocSht^{(I_1,\dotsc,I_k)}_{G,\mu_\bullet,\wh{m},l}|_{\prod_{i\in I}\bD_i}\ra\fLocSht^{(I_1,\dotsc,I_k)}_{G,\mu_\bullet,\wh{m},l'}|_{\prod_{i\in I}\bD_i}
  \end{align*}
  are closed embeddings, the maps $A_{l'}\ra A_l$ are surjective. Write $A\coloneqq\varprojlim_lA_l$. Write $J_0$ for the largest ideal of definition of $A_0$, and write $J$ for its preimage in $A$.

  For any positive integer $c$, we claim that $A_{l'}/J^c\ra A_l/J^c$ is an isomorphism for large enough $l$ and $l'$. Note that $A_{l'}/J^c\ra A_l/J^c$ is surjective with nilpotent kernel, and the Mittag-Leffler criterion implies that $A/J^c=\varprojlim_lA_l/J^c$. Endow $A/J^c$ with the discrete topology. Because the $\ze_i$ vanish in $A/J=A_0/J_0$, the $\ze_i$ are nilpotent in $A/J^c$, so $\bF_q\lb{\ze_i}_{i\in I}\ra A/J^c$ is continuous. Altogether we can apply Lemma \ref{ss:nilpotentalgebraization} to the local $G$-shtukas $\sG^l$ over $\Spec A_l/J^c$ obtained from the morphism
\begin{align*}
\Spec A_l/J^c\ra\Spf A_l=\fU_l\subseteq\fLocSht^{(I_1,\dotsc,I_k)}_{G,\mu_\bullet,\wh{m},l}|_{\prod_{i\in I}\bD_i}
\end{align*}
to get a local $G$-shtuka $\sG$ over $\Spec A/J^c$ bounded by $\mu_\bullet$. Next, consider the quasi-isogeny $\de^0$ obtained from $\Spec A_0/J_0\ra\fLocSht^{(I_1,\dotsc,I_k)}_{G,\mu_\bullet,\wh{m},0}|_{\prod_{i\in I}\bD_i}$. Proposition \ref{ss:rigidisogeny} uniquely lifts $\de^0$ to a quasi-isogeny $\de$ from $\sG$ to $G$, which implies that the resulting $A/J^c$-point $(\sG,\de)$ of $\fLocSht^{(I_1,\dots,I_k)}_{G,\mu_\bullet}|_{\prod_{i\in I}\bD_i}$ lies in $\fLocSht^{(I_1,\dots,I_k)}_{G,\mu_\bullet,\wh{m}}|_{\prod_{i\in I}\bD_i}$. Therefore \cite[Lemma 5.4]{HV11} indicates that $(\sG,\de)$ lies in $\fLocSht^{(I_1,\dots,I_k)}_{G,\mu_\bullet,\wh{m},l}|_{\prod_{i\in I}\bD_i}$ for large enough $l$. Pulling back to $\Spec A_0/J_0$ shows that $(\sG,\de)$ even lies in $\Spf A_l$. The uniqueness of Proposition \ref{ss:rigidisogeny} implies that the pullback of $(\sG,\de)$ to $\Spec A_{l'}/J^c$ equals $(\sG^l,\de^l)$, so $A_{l'}\ra A_l\ra A/J^c\ra A_{l'}/J^c$ equals the quotient map. Quotienting by the image of $J^c$ in $A_l$ shows that $A_{l'}/J^c\ra A_l/J^c$ is an isomorphism, which concludes our proof of the claim.

Write $\fa_l\coloneqq\ker(A\ra A_l)$. The claim indicates that the ideals $\fa_l+J^c$ of $A$ stabilize for any positive integer $c$, and because the $A_l$ are noetherian, we see that the $\im(J/J^2\ra A_l/J^2)=J/(J^2+\fa_l)$ are finite over $A$. Therefore \cite[proposition (2.5)]{RZ96} shows that $A$ with the inverse limit topology is noetherian and $J$-adic, which implies that $\varinjlim_l\fU_l=\Spf A$. Finally, the reduced subscheme of $\Spf{A}$ is of finite type over $\prod_{i\in I}v_i$ by Proposition \ref{ss:formalisogenybound}, so $\Spf{A}$ is formally of finite type over $\bD_i$.
\end{proof}

\subsection{}
We can use the quasi-isogeny to define the following distance function.
\begin{defn*}
  Let $K$ be a field over $\bF_q$, and let $x=(\sG,\de)$ and $x'=(\sG',\de')$ be $K$-points of $\fLocSht^{(I_1,\dotsc,I_k)}_G$. Write $d(x,x')$ for the smallest non-negative integer $m$ such that the quasi-isogeny $\de^{-1}\circ\de'$ of local $G$-shtukas over $\Spec K\lp{z}$ is bounded by $m$.
\end{defn*}

\begin{lem}\label{lem:formalmetric}
As $K$ runs over all fields over $\bF_q$, the maps $d$ induce a metric on the underlying set $|\fLocSht^{(I_1,\dotsc,I_k)}_G|$. For any $x$ in $|\fLocSht^{(I_1,\dotsc,I_k)}_G|$ and non-negative integer $r$, the associated closed ball $B_r(x)$ of radius $r$ centered at $x$ is closed with respect to the Zariski topology on $|\fLocSht^{(I_1,\dotsc,I_k)}_G|$.
\end{lem}
\begin{proof}
  We immediately see that $d$ is insensitive to field extensions, so $d$ induces a map $|\fLocSht^{(I_1,\dotsc,I_k)}_G|\times|\fLocSht^{(I_1,\dotsc,I_k)}_G|\ra\bZ_{\geq0}$. Since relative position bounds along the same divisor are sub-additive under composition, $d$ satisfies the triangle inequality, and because $2\rho^\vee$ is fixed by the Chevalley involution, $d$ is symmetric. Next, if $d(x,x')=0$, then $\io_*(\de_j^{-1}\circ\de_j')$ extends to an isomorphism of $\SL_r$-bundles on $\Spec K\lb{z}$ for all $1\leq j\leq k$. Since $\io$ is a monomorphism, this implies that the $\de_j^{-1}\circ\de_j'$ extend to isomorphisms of $G$-bundles on $\Spec K\lb{z}$, so $x=x'$. For the last statement, note that $B_r(x)$ equals, on the level of topological spaces, the preimage of the closed substack $[L_z^+\SL_h\!\bs\!\Gr^1_{z,\SL_h,r2\rho^\vee}]^k$ under the morphism
  \begin{align*}
    \fLocSht^{(I_1,\dotsc,I_k)}_G\ra[L_z^+\SL_h\!\bs\!\Gr^1_{z,\SL_h}]^k
  \end{align*}
given by $(\sG',\de')\mapsto(\io_*(\de^{-1}_j\circ\de'_j))_{j=1}^k$.
\end{proof}

\subsection{}\label{ss:isogenyboundingformalmoduli}
All points of $\fLocSht^{(I_1,\dotsc,I_k)}_{G,\mu_\bullet}|_{\prod_{i\in I}\bD_i}$ are close enough to one defined over a fixed finite field in the following sense.
\begin{lem*}
There exists a finite extension $\bF_{q'}$ of $\bF_q$ and a non-negative integer $D$ such that, for every $x$ in $|\fLocSht^{(I_1,\dotsc,I_k)}_{G,\mu_\bullet}|_{\prod_{i\in I}\bD_i}|$, there exists an $\bF_{q'}$-point $y$ of $\fLocSht^{(I_1,\dotsc,I_k)}_{G,\mu_\bullet}|_{\prod_{i\in I}\bD_i}$ satisfying $d(x,y)\leq D$.
\end{lem*}
\begin{proof}
  Suppose that $x$ corresponds to a $K$-point $(\sG,\de)$, where we can assume that $K$ is an algebraically closed field over $\bF_q$. Then $\sG_j$ is trivial for all $1\leq j\leq k$, and after fixing trivializations of the $\sG_j$, our $\de_j$ correspond to $g_j$ in $G(K\lp{z})$. The commutativity of the diagram
  \begin{align*}
    \xymatrix{\sG_1|_{\Spec K\lp{z}}\ar@{-->}[r]^-{\phi_1}\ar@{-->}[d]^-{\de_1} & \dotsm\ar@{-->}[r]^-{\phi_{k-1}} & \sG_k|_{\Spec K\lp{z}}\ar@{-->}[r]^-{\phi_k}\ar@{-->}[d]^-{\de_k} & \prescript\tau{}{\sG_1}|_{\Spec K\lp{z}}\ar@{-->}[d]^-{\prescript\tau{}{\de_1}}\\
    G\ar@{=}[r] &\dotsm\ar@{=}[r] & G\ar@{=}[r] & G}
  \end{align*}
  implies that $\prescript\tau{}{\de_1^{-1}}\circ\de_1=\phi_k\circ\dotsb\circ\phi_1$, so the image of $\tau(g_1)^{-1}g_1$ in $\Gr^1_{z,G}|_{v_*}$ lies in $\Gr^1_{z,G,\sum_{i\in I}\mu_i}|_{v_*}$. Now \ref{ss:affineschubertvarieties} indicates that $\Gr^1_{z,G,\sum_{i\in I}\mu_i}|_{v_*}$ is a quasi-compact scheme, so \cite[Lemma 5.4]{HV11} shows that its image under $\io_*$ lies in $\Gr^1_{z,\SL_h,m}$ for large enough $m$. Therefore \cite[2.2.1 (ii)]{Lan00} and \cite[(2.1)]{RZ99} yield a non-negative integer $D$ such that, for all such $g_1$, there exists $h_1$ in $G(\bF_q\lp{z})$ such that the image of $g_1h^{-1}_1$ in $\Gr^1_{z,\SL_h}$ lies in $\Gr^1_{z,\SL_h,D2\rho^\vee}$.

If $\sum_{i\in I}\mu_i$ is not a coroot, then $\fLocSht^{(I_1,\dotsc,I_k)}_{G,\mu_\bullet}|_{\prod_{i\in I}\bD_i}$ is empty, and the result vacuously holds. So assume that $\sum_{i\in I}\mu_i$ is a coroot, which implies that $\Gr^{(I)}_{G,\sum_{i\in I}\mu_i}|_{\prod_{i\in I}v_i}$ contains the image of $1$ in $\Gr^{(I)}_G|_{\prod_{i\in I}v_i}$. Since the convolution morphism $\Gr^{(I_1,\dotsc,I_k)}_{G,\mu_\bullet}|_{\prod_{i\in I}v_i}\ra\Gr^{(I)}_{G,\sum_{i\in I}\mu_i}|_{\prod_{i\in I}v_i}$ is of finite type by \ref{ss:affineschubertvarieties} and surjective, its fiber at $1$ has an $\bF_{q'}$-point $b$ for some finite extension $\bF_{q'}$ of $\bF_q$. Next, identify $\Gr^{(I_1,\dotsc,I_k)}_G|_{v^I}$ with $\Gr^k_{z,G}|_{v^I}$. Because the fiber of $(L_zG)^k\ra\Gr^k_G|_{\prod_{i\in I}v_i}$ at $b$ is an $(L_z^+G)^k$-bundle on $\Spec\bF_{q'}$, Lang's lemma indicates that it has an $\bF_{q'}$-point $(b_j)_{j=1}^k$. By construction, the local $G$-shtuka $\sH\coloneqq((G)_{j=1}^k,(b_j)_{j=1}^k)$ over $\Spec\bF_{q'}$ is bounded by $\mu_\bullet$, and $b_k\dotsm b_1$ equals $1$ up to right $G(\bF_{q'}\lb{z})$-translation. After replacing $b_1$ with a right $G(\bF_{q'}\lb{z})$-translate, we can assume that $b_k\dotsm b_1=1$. Combined with the fact that $h_1=\tau(h_1)$, we see that the diagram
\begin{align*}
    \xymatrix{G\ar@{-->}[r]^-{b_1}\ar@{-->}[d]^-{h_1} & \dotsm\ar@{-->}[r]^-{b_{k-1}} & G\ar@{-->}[r]^-{b_k}\ar@{-->}[d]^-{h_k} & G\ar@{-->}[d]^-{\prescript\tau{}{h_1}}\\
    G\ar@{=}[r] &\dotsm\ar@{=}[r] & G\ar@{=}[r] & G}
\end{align*}
commutes for uniquely determined $h_2,\dotsc,h_k$ in $G(\bF_{q'}\lp{z})$. Since $b_j$ and $\phi_j$ are bounded by $\sum_{i\in I_j}\mu_j$ for $1\leq j\leq k-1$, where the relative position bound is taken with respect to $z$, a quasi-compactness argument shows that, after increasing $D$ by an amount depending only on $\mu_\bullet$, the image of $g_jh^{-1}_j$ in $\Gr^1_{z,\SL_h}$ lies in $\Gr^1_{z,\SL_h,D2\rho^\vee}$. Hence the quasi-isogeny $h\coloneqq(h_j)_{j=1}^k$ from $\sH$ to $G$ yields an $\bF_{q'}$-point $y\coloneqq(\sH,h)$ of $\fLocSht^{(I_1,\dotsc,I_k)}_{G,\mu_\bullet}|_{\prod_{i\in I}\bD_i}$ with $d(x,y)\leq D$, as desired.
\end{proof}

\subsection{}\label{ss:formalmodulirepresentable}
The following theorem is the main result of this section. Write $B_r(x)_{\mu_\bullet}$ for the intersection of $B_r(x)$ and $|\fLocSht^{(I_1,\dots,I_k)}_{G,\mu_\bullet}|_{\prod_{i\in I}\bD_i}|$, and write $\mathbf1$ for the $\bF_q$-point $(G,(\id)_{j=1}^k)$ of $\fLocSht^{(I_1,\dotsc,I_k)}_{G,\mu_\bullet}$. Note that $B_m(\mathbf1)_{\mu_\bullet}$ equals $|\fLocSht^{(I_1,\dots,I_k)}_{G,\mu_\bullet,m}|_{\prod_{i\in I}\bD_i}|$.
\begin{thm*}
  Our $\fLocSht^{(I_1,\dotsc,I_k)}_{G,\mu_\bullet}|_{\prod_{i\in I}\bD_i}$ is a formal scheme that is locally formally of finite type over $\textstyle\prod_{i\in I}\bD_i$. 
\end{thm*}
\begin{proof}
Let $\bF_{q'}$ and $D$ be as in Lemma \ref{ss:isogenyboundingformalmoduli}. Write $Z^s_m$ for the union
\begin{align*}
  \bigcup_y B_D(y)_{\mu_\bullet}\cap B_m(\mathbf1)_{\mu_\bullet},
\end{align*}
where $y$ runs over $\bF_{q'}$-points of $\fLocSht^{(I_1,\dotsc,I_k)}_{G,\mu_\bullet}|_{\prod_{i\in I}\bD_i}$ satisfying $d(\mathbf1,y)\geq s$. The triangle inequality implies that it suffices to take $y$ also satisfying $d(\mathbf1,y)\leq m+D$. Because $B_{m+D}(\mathbf1)_{\mu_\bullet}$ equals $|\fLocSht^{(I_1,\dots,I_k)}_{G,\mu_\bullet,m+D}|_{\prod_{i\in I}\bD_i}|$, Proposition \ref{ss:formalisogenybound} implies that there are finitely many such $y$. Hence Lemma \ref{lem:formalmetric} indicates that $Z^s_m$ is a a finite union of Zariski closed subsets of $|\fLocSht^{(I_1,\dotsc,I_k)}_{G,\mu_\bullet}|_{\prod_{i\in I}\bD_i}|$.

Write $\fU^s_m$ for the open formal subscheme of $\fLocSht^{(I_1,\dots,I_k)}_{G,\mu_\bullet,\wh{m}}|_{\prod_{i\in I}\bD_i}$ with underlying topological space given by the complement of $Z^s_m$. By Lemma \ref{ss:mpartformalscheme}, $\fU^s_m$ is formally of finite type over $\prod_{i\in I}\bD_i$. Note that $\fLocSht^{(I_1,\dots,I_k)}_{G,\mu_\bullet,\wh{m}}|_{\prod_{i\in I}\bD_i}$ equals the formal completion of $\fLocSht^{(I_1,\dots,I_k)}_{G,\mu_\bullet,\wh{m+1}}|_{\prod_{i\in I}\bD_i}$ along the reduced subscheme of $\fLocSht^{(I_1,\dotsc,I_k)}_{G,\mu_\bullet,\wh{m}}|_{\prod_{i\in I}\bD_i}$, so $\fU^s_{m+1}$ equals the formal completion of $\fU^s_m$ along the reduced subscheme of $\fU^s_m$.

For any non-negative integer $s$, we claim that $\fU^s_m$ stabilizes. The above indicates that it suffices to check this on underlying sets, so suppose that there exists $x$ in $|\fU^s_{m+1}|\ssm|\fU^s_m|$. Lemma \ref{ss:isogenyboundingformalmoduli} yields an $\bF_{q'}$-point $y$ of $\fLocSht^{(I_1,\dotsc,I_k)}_{G,\mu_\bullet}|_{\prod_{i\in I}\bD_i}$ satisfying $d(x,y)\leq D$. As $x$ does not lie in $Z^s_{m+1}$, we have $d(\mathbf1,y)<s$, so the triangle inequality yields $m+1 = d(\mathbf1,x)<s+D$. Hence $\fU^s_m$ stabilizes for $m\geq s+D-1$, which concludes our proof of the claim.

Set $\fU^s\coloneqq\varinjlim_m\fU^s_m$. Proposition \ref{ss:formalisogenybound} implies that $\fU^s$ is an open subsheaf of
\begin{align*}
\fLocSht^{(I_1,\dotsc,I_k)}_{G,\mu_\bullet}|_{\prod_{i\in I}\bD_i}.
\end{align*}
The claim shows that $\fU^s$ equals $\fU^s_m$ for large enough $m$, so $\fU^s$ is formally of finite type over $\prod_{i\in I}\bD_i$. Now we just need to prove $\varinjlim_s\fU^s=\fLocSht^{(I_1,\dotsc,I_k)}_{G,\mu_\bullet}|_{\prod_{i\in I}\bD_i}$. It suffices to check this on underlying sets, so take $x$ in $|\fLocSht^{(I_1,\dotsc,I_k)}_{G,\mu_\bullet}|_{\prod_{i\in I}\bD_i}|$. Proposition \ref{ss:formalisogenybound} indicates that $x$ lies in
\begin{align*}
|\fLocSht^{(I_1,\dotsc,I_k)}_{G,\mu_\bullet,m}|_{\prod_{i\in I}\bD_i}|
\end{align*}
for large enough $m$, so for all $y$ in $|\fLocSht^{(I_1,\dotsc,I_k)}_G|$ such that $x$ lies in $B_D(y)_{\mu_\bullet}$, the triangle inequality yields $d(\mathbf1,y)\leq m+D$. Therefore $x$ lies in $|\fU^{m+D+1}|$. 
\end{proof}

\subsection{}\label{ss:Lgroup}
Using representations of the dual group, we can index relative position bounds as follows. Let $\wt{F}$ be the finite Galois extension of $F$ such that $\Gal(\wt{F}/F)$ equals the image of the $\Ga_F$-action on $X_*^+(T)$. Write $\wt\bD$ for $\Spd\cO_{\wt{F}}$. Let $E$ be a finite extension of $\bQ_\ell(\sqrt{q})$, write $\wh{G}$ for the dual group of $G_F$ over $\cO_E$, and write $\prescript{L}{}{G}$ for $\wh{G}\rtimes\Gal(\wt{F}/F)$.

Let $V$ be an object of $\Rep_E(\prescript{L}{}{G})^I$. Note that $\coprod_{\mu_\bullet}\fLocSht^{(I_1,\dotsc,I_k)}_{G,\mu_\bullet}|_{\wt\bD^I}$ naturally descends to a sheaf $\fLocSht^{(I_1,\dotsc,I_k)}_{G,V}$ over $\bD^I$, where $\mu_\bullet$ runs over highest weights appearing in $V_{\ov\bQ_\ell}|_{\wh{G}^I}$. Theorem \ref{ss:formalmodulirepresentable} and descent imply that $\fLocSht^{(I_1,\dotsc,I_k)}_{G,V}$ is a formal scheme that is locally formally of finite type over $\bD^I$.

\subsection{}\label{ss:formalpartialfrobenius}
Finally, we define partial Frobenii for the formal moduli of local $G$-shtukas.
\begin{defn*}
  Write $\fFr^{(I_1,\dotsc,I_k)}:\fLocSht^{(I_1,\dotsc,I_k)}_{G,V}\ra\fLocSht^{(I_2,\dotsc,I_k,I_1)}_{G,V}$ for the morphism given by sending
  \begin{align*}
    \xymatrix{\sG_1\ar@{-->}[r]^-{\phi_1}\ar@{-->}[d]^-{\de_1} & \dotsm\ar@{-->}[r]^-{\phi_{k-1}} & \sG_k\ar@{-->}[r]^-{\phi_k}\ar@{-->}[d]^-{\de_k} & \prescript\tau{}{\sG_1}\ar@{-->}[d]^-{\prescript\tau{}{\de_1}\quad\mbox{to}} & \sG_2\ar@{-->}[r]^-{\phi_2}\ar@{-->}[d]^-{\de_2} & \dotsm\ar@{-->}[r]^-{\phi_k} & \prescript\tau{}{\sG_1}\ar@{-->}[r]^-{\prescript\tau{}{\phi_1}}\ar@{-->}[d]^-{\prescript\tau{}{\de_1}} & \prescript\tau{}{\sG_2}\ar@{-->}[d]^-{\prescript\tau{}{\de_2}} \\
    G\ar@{=}[r] & \dotsm \ar@{=}[r] & G\ar@{=}[r] & G &     G\ar@{=}[r] & \dotsm \ar@{=}[r] & G\ar@{=}[r] & G.}
  \end{align*}
\end{defn*}
Note that $\fFr^{(I_1,\dotsc,I_k)}$ lies above the endomorphism of $\bD^I$ given by geometric $q$-Frobenius on the $i$-th factor for $i$ in $I_1$ and the identity on all other factors.

\section{Relative $z$-adic Hodge theory}\label{s:zadichodgetheory}
The local shtukas defined in \S\ref{s:formallocalshtukas} are (formal) algebraic, while the local shtukas used by Fargues--Scholze \cite{FS21} are (non-archimedean) analytic in nature. To compare them, we need an equicharacteristic version of Kedlaya--Liu's results \cite{KL15} on relative $p$-adic Hodge theory. Our goal in this section is to prove the necessary results on \emph{relative $z$-adic Hodge theory}, in the spirit of work of Hartl \cite{Har11}.

We begin by recalling the equicharacteristic version of Fontaine's period ring $A_{\inf}$. Using a result of Ansch\"utz \cite{An22}, we prove an algebraization theorem for $G$-bundles on $A_{\inf}$, at least pro-\'etale locally on the base. Finally, we relate $\ul{G(\cO_F)}$-local systems to $G$-bundles on the equicharacteristic version of the (relative integral) Robba ring equipped with a Frobenius automorphism.

Our arguments closely follow those of Kedlaya--Liu \cite{KL15} and Scholze--Weinstein \cite{SW20}. However, we have streamlined and simplified the presentation, both because we only prove what we need as well as because the arithmetic of formal power series is easier than that of Witt vectors.

\subsection{}
Let $S=\Spa(R,R^+)$ be an affinoid perfectoid space over $\bF_q$, and choose a pseudouniformizer $\vpi$ of $R$. Write $\cY_S$ for the complement of the vanishing locus of $\vpi$ and $z$ in $\Spa R^+\lb{z}$, and note that $\cY_S$ is the analytic locus of the pre-adic space $\Spa R^+\lb{z}$. We have a continuous map $\rad:\abs{\cY_S}\ra[0,\infty]$ given by
\begin{align*}
x\mapsto\frac{\log\abs{\vpi(\wt{x})}}{\log\abs{z(\wt{x})}},
\end{align*}
where $\wt{x}$ denotes the unique rank-$1$ generalization of $x$ in $\cY_S$. For any closed interval $\cI$ in $[0,\infty]$ with rational endpoints, write $\cY_{S,\cI}=\Spa(B_{S,\cI},B^+_{S,\cI})$ for the associated rational open subspace of $\Spa R^+\lb{z}$, which lies in $\cY_S$. More generally, for any subset $\cI$ of $[0,\infty]$, write $\cY_{S,\cI}$ for the open subspace $\bigcup_{\cI'}\cY_{S,\cI'}$ of $\cY_S$, where $\cI'$ runs over closed intervals in $\cI$ with rational endpoints. Note that $\cY_{S,\cI}\subseteq\rad^{-1}(\cI)$. We see that $\cY_{S,[0,\infty)}$ and $\cY_{S,(0,\infty)}$ are naturally isomorphic to $\bD\times S$ and $\Spa{F}\times S$, respectively.

Write $\tau:S\ra S$ for the absolute $q$-Frobenius automorphism, and by abuse of notation, write $\tau:R\lb{z}\ra R\lb{z}$ for the canonical lift of absolute $q$-Frobenius. Note that $\rad\circ\tau=q\cdot\rad$. Finally, write $X_S$ for the quotient $\cY_{S,(0,\infty)}/\tau^\bZ$.

\subsection{}\label{ss:Batinfinity}
When $\cI$ contains $\infty$, we can describe $B_{S,\cI}$ using the following lemma. For any positive $r$ in $\bZ[\textstyle\frac1p]$, write $R^+\ba{z,\textstyle\frac{\vpi^r}z}$ for the $\vpi$-adic completion of $R^+\lb{z}[\textstyle\frac{\vpi^r}z]$.
\begin{lem*}
We can identify
\begin{align*}
R^+\ba{z,\textstyle\frac{\vpi^r}z} = \left\{
  \begin{tabular}{c|c}
    \multirow{2}{*}{$\displaystyle\sum_{m=-\infty}^\infty a_mz^m$} & the $a_m$ lie in $R^+$, $\displaystyle\lim_{m\to-\infty}a_m\vpi^{rm}=0$,\\
  & and for $m\leq0$, $a_m\vpi^{rm}$ lies in $R^+$ 
  \end{tabular}\right\}.
\end{align*}
If we give $R^+\ba{z,\textstyle\frac{\vpi^r}z}$ the $(\vpi,z)$-adic topology, then $B_{S,[1/r,\infty]}$ equals $R^+\ba{z,\frac{\vpi^r}z}[\frac1z]$.
\end{lem*}
\begin{proof}
The above description of $R^+\ba{z,\textstyle\frac{\vpi^r}z}$ follows immediately from the definition. This description shows that $R^+\ba{z,\textstyle\frac{\vpi^r}z}$ is $z$-adically complete as a ring, so $R^+\ba{z,\textstyle\frac{\vpi^r}z}$ equals the $(\vpi,z)$-adic completion of $R^+\lb{z}[\textstyle\frac{\vpi^r}z]$ as rings. Since $\cY_{S,[1/r,\infty]}$ equals the rational open subspace $\{|\vpi^r|\leq|z|\neq0\}$ of $\Spa R^+\lb{z}$, this identifies $B_{S,[1/r,\infty]}$ with $R^+\ba{z,\textstyle\frac{\vpi^r}z}[\frac1z]$ if we give $R^+\ba{z,\textstyle\frac{\vpi^r}z}$ the $(\vpi,z)$-adic topology.
\end{proof}

\subsection{}\label{ss:alternativeAinf}
Sometimes, it will be convenient to ignore the topology induced from $R$ as follows. Write $A'(R^+)$ for $R^+\lb{z}$ with the $z$-adic topology.
\begin{lem*}
Our $\Spa(A'(R^+)[\textstyle\frac1z],A'(R^+))$ is a sousperfectoid adic space.
\end{lem*}
\begin{proof}
The natural map $A'(R^+)[\textstyle\frac1z]\ra R^+\lb{z^{\pm1/p^\infty}}$ is a split injection of topological $A'(R^+)[\textstyle\frac1z]$-modules, where we give $R^+\lb{z^{\pm1/p^\infty}}$ the $z$-adic topology.
\end{proof}

\begin{prop}\label{ss:Ysousperfectoid}
Our $\cY_S$ is a sousperfectoid adic space.
\end{prop}
\begin{proof}
  Note that $\cY_S$ is covered by $\cY_{S,[0,\infty)}$ and $\cY_{S,[1,\infty]}$. Now $\cY_{S,[0,\infty)}$ is a sousperfectoid adic space by \cite[Proposition II.1.1]{FS21}, so it suffices to prove that $\cY_{S,[1,\infty]}$ is a sousperfectoid adic space. By Proposition \ref{ss:Batinfinity}, $B_{S,[1,\infty]}$ equals $R^+\ba{z,\frac\vpi{z}}[\textstyle\frac1z]$, where $R^+\ba{z,\textstyle\frac\vpi{z}}$ has the $(\vpi,z)$-adic topology.

  Now $z$ divides $\vpi$ in $R^+\lb{z}[\textstyle\frac\vpi{z}]$, so the $(\vpi,z)$-adic topology on $R^+\lb{z}[\textstyle\frac\vpi{z}]$ equals the $z$-adic topology. This enables us to identify $\cY_{S,[1,\infty]}$ with the rational open subspace $\{|\vpi|\leq|z|\neq0\}$ of $\Spa(A'(R^+)[\textstyle\frac1z],A'(R^+))$. The latter is sousperfectoid by Lemma \ref{ss:alternativeAinf}, so $\cY_{S,[1,\infty]}$ is as well.
\end{proof}

\subsection{}\label{ss:vectorbundlesfullyfaithful}
Since a power of $\vpi$ divides a power of $z$ in $R^+\lb{z}[\frac{z}{\vpi^r}]$, the $(\vpi,z)$-adic topology on $R^+\lb{z}[\frac{z}{\vpi^r}]$ equals the $\vpi$-adic topology. Therefore $B_{S,[0,1/r]}$ equals the Tate algebra $R\ang{\frac{z}{\vpi^r}}$. This argument similarly lets us identify
  \begin{align*}
    B_{S,[1,1]} = \left\{\sum_{m=-\infty}^\infty a_mz^m\,\middle|\,\mbox{the }a_m\mbox{ lie in }R\mbox{ and }\lim_{m\to\pm\infty}a_m\vpi^m=0\right\}.
  \end{align*}

We will use the following result with the Tannakian description of $G$-bundles.
\begin{prop*}
Pullback yields a fully faithful functor
  \begin{align*}
    \{\mbox{vector bundles on }\Spec R^+\lb{z}\}\lhook\joinrel\longrightarrow\{\mbox{vector bundles on }\cY_S\}.
  \end{align*}
\end{prop*}
\begin{proof}
  Let $f:M\ra M'$ be a map of finite projective $R^+\lb{z}$-modules, and consider its pullback $g$ to $\cY_S$. Now Proposition \ref{ss:Ysousperfectoid} and \cite[Theorem 2.7.7]{KL15} indicate that $g|_{\cY_{S,[0,1]}}$, $g|_{\cY_{S,[1,\infty]}}$, and $g|_{\cY_{S,[1,1]}}$ correspond to maps of finite projective modules over $B_{S,[0,1]}$, $B_{S,[1,\infty]}$, and $B_{S,[1,1]}$, respectively, which are given by tensoring with $f$ over $R^+\lb{z}$. Lemma \ref{ss:Batinfinity} indicates that $B_{S,[1,\infty]}$ equals $R^+\ba{z,\textstyle\frac\vpi{z}}[\frac1z]$ as rings, so we see that $B_{S,[0,1]}$ and $B_{S,[1,\infty]}$ inject into $B_{S,[1,1]}$. Note that their intersection equals $R^+\lb{z}$. Therefore the flatness of $M$ yields a Cartesian square
  \begin{align*}
    \xymatrix{M\ar@{^{(}->}[r]\ar@{^{(}->}[d] & M\otimes_{R^+\lb{z}} B_{S,[0,1]}\ar@{^{(}->}[d] \\
    M\otimes_{R^+\lb{z}}B_{S,[1,\infty]}\ar@{^{(}->}[r] & M\otimes_{R^+\lb{z}} B_{S,[1,1]},}
  \end{align*}
  and the same holds for $M'$. In particular, we recover $f$ as the restriction of $g|_{\cY_{S,[0,1]}}$ (or of $g|_{\cY_{S,[1,\infty]}}$) to the intersection of $M\otimes_{R^+\lb{z}} B_{S,[0,1]}$ and $M\otimes_{R^+\lb{z}}B_{S,[1,\infty]}$ in $M\otimes_{R^+\lb{z}} B_{S,[1,1]}$.
\end{proof}

\subsection{}\label{ss:algebraizebundles}
We turn to the first main result of this section, which algebraizes $G$-bundles on $\cY_S$ when $S$ is a product of points as in \cite[Definition 1.2]{Gle20}.

Recall that $\Spa$ yields an anti-equivalence from the category of perfectoid Huber pairs over $\bF_q\lb{\ze_i}_{i\in I}$ to the category of affinoid perfectoid spaces over $\bD^I$. Let $S=\Spa(R,R^+)$ be an affinoid perfectoid space over $\bD^I$, and for all $i$ in $I$, write $\Ga_i$ for the graph of its $i$-th projection $S\ra\bD$, which is a closed effective Cartier divisor on $\cY_S$ \cite[Proposition VI.1.2 (i)]{FS21}.
\begin{thm*}
  Suppose that $S$ is a product of points as in \cite[Definition 1.2]{Gle20}, and let $1\leq j\leq k$ be an integer. Then pullback yields an equivalence of groupoids
  \begin{align*}
    \{G\mbox{-bundles on }\Spec R^+\lb{z}\}\lra^\sim\{G\mbox{-bundles on }\cY_S\},
  \end{align*}
where morphisms on the left-hand side are given by isomorphisms of their pullbacks to $\Spec R^+\lb{z}[\textstyle\frac1{z-\ze_i}]_{j\in I_j}$, and morphisms on the right-hand side are given by isomorphisms of their pullbacks to $\cY_S\ssm\sum_{i\in I_j}\Ga_i$ that are meromorphic along $\sum_{i\in I_j}\Ga_i$.
\end{thm*}
\begin{proof}
  First, we tackle full faithfulness. Write $\sO(\sum_{i\in I_j}\Ga_i)$ for the line bundle on $\cY_S$ associated with the closed effective Cartier divisor $\sum_{i\in I_j}\Ga_i$, and let $\sG$ and $\sG'$ be $G$-bundles on $\cY_S$. The Tannakian description of $G$-bundles implies that an isomorphism $\sG|_{\cY_S\ssm\sum_{i\in I_j}\Ga_i}\ra^\sim\sG'|_{\cY_S\ssm\sum_{i\in I_j}\Ga_i}$ that is meromorphic along $\sum_{i\in I_j}\Ga_i$ corresponds to a family of morphisms of vector bundles over $\cY_S$
  \begin{align*}
    \sG(V)\ra\sG'(V)\otimes\sO(\textstyle\sum_{i\in I_j}\Ga_i)^{\otimes n(V)}
  \end{align*} 
 that is functorial in $V$, compatible with tensor products, and compatible with duals, where $V$ runs over objects of $\Rep_{\cO_F}(G)$ and $n(V)$ is a large enough integer. Hence full faithfulness follows immediately from Proposition \ref{ss:vectorbundlesfullyfaithful}.

  As for essential surjectivity, let $\sG$ be a $G$-bundle on $\cY_S$. By \cite[Theorem 2.7.7]{KL15}, $\sG|_{\cY_{S,[0,1]}}$ and $\sG|_{\cY_{S,[1,\infty]}}$ correspond to $G$-bundles $N_0$ and $N_\infty$ on $\Spec B_{S,[0,1]}$ and $\Spec B_{S,[1,\infty]}$, respectively. Note that the $z$-adic completion of $R^+\lb{z}[\textstyle\frac{z}\vpi]$ equals $R^+\ang{\textstyle\frac{z}\vpi}$ as rings, so the global sections of the rational open subspace $\{|z|\leq|\vpi|\neq0\}$ of $\Spa(A'(R^+)[\textstyle\frac1z],A'(R^+))$ equals $R\ang{\textstyle\frac{z}\vpi}[\frac1z]=B_{S,[0,1]}[\frac1z]$ as rings. We have seen in the proof of Proposition \ref{ss:Ysousperfectoid} that the global sections of the rational open subspace $\{|\vpi|\leq|z|\neq0\}$ of $\Spa(A'(R^+)[\textstyle\frac1z],A'(R^+))$ equals $B_{S,[1,\infty]}$. Because these two rational open subspaces cover $\Spa(A'(R^+)[\textstyle\frac1z],A'(R^+))$, Lemma \ref{ss:alternativeAinf} and \cite[Theorem 2.7.7]{KL15} enable us to glue $N_0[\textstyle\frac1z]$\footnote{By abuse of notation, we apply notation for pullbacks of vector bundles to $G$-bundles.} and $N_\infty$ into a $G$-bundle $N_{\frac1z}$ on $\Spec A'(R^+)[\textstyle\frac1z]=\Spec R^+\lp{z}$.

  Note that the $z$-adic completion of $R^+\lb{z}[\textstyle\frac1\vpi]$ equals $R\lb{z}$. Since
  \begin{align*}
    N_{\frac1z}\otimes_{R^+\lp{z}}B_{S,[0,1]}[\textstyle\frac1z] = N_0[\textstyle\frac1z],
  \end{align*}
  we see that $N_{\frac1z}[\textstyle\frac1\vpi]\otimes_{R^+\lp{z}[\frac1\vpi]}R\lp{z} = N_0\otimes_{B_{S,[0,1]}}R\lp{z}$. Therefore we can apply Beauville--Laszlo to the vanishing locus of $z$ in $\Spec R^+\lb{z}[\textstyle\frac1\vpi]$ to glue $N_{\frac1z}[\textstyle\frac1\vpi]$ and $N_0\otimes_{B_{S,[0,1]}}R\lb{z}$ into a $G$-bundle $N_{\frac1\vpi}$ on $\Spec R^+\lb{z}[\textstyle\frac1\vpi]$. As $N_{\frac1\vpi}[\textstyle\frac1z] = N_{\frac1z}[\frac1\vpi]$, we can glue $N_{\frac1\vpi}$ and $N_{\frac1z}$ into a $G$-bundle $\mathring{N}$ on the complement of the vanishing locus of $\vpi$ and $z$ in $\Spec R^+\lb{z}$. Finally, because $S$ is a product of points, \cite[Proposition 11.5]{An22} uniquely extends $\mathring{N}$ to a $G$-bundle $N$ on $\Spec R^+\lb{z}$.

Let us verify that the pullback of $N$ to $\cY_S$ equals $\sG$. Because $N[\textstyle\frac1z]=\mathring{N}[\frac1z]=N_{\frac1z}$, we see that $N\otimes_{R^+\lb{z}}B_{S,[1,\infty]}=N_\infty$. Thus we just need to show $N\otimes_{R^+\lb{z}}B_{S,[0,1]} = N_0$. We have $N[\textstyle\frac1\vpi]=\mathring{N}[\frac1\vpi]=N_{\frac1\vpi}$, so
  \begin{align*}
    N\otimes_{R^+\lb{z}}B_{S,[0,1]}[\textstyle\frac1z] = N_{\frac1z}\otimes_{R^+\lp{z}}B_{S,[0,1]}[\frac1z] = N_0[\frac1z].
  \end{align*}
  Note that the $z$-adic completion of $B_{S,[0,1]}=R\ang{\frac{z}\vpi}$ equals $R\lb{z}$, and
  \begin{align*}
    N\otimes_{R^+\lb{z}}R\lb{z} = N_{\frac1\vpi}\otimes_{R^+\lb{z}[\frac1\vpi]}R\lb{z} = N_0\otimes_{B_{S,[0,1]}}R\lb{z}.
  \end{align*}
Hence the desired result follows from applying the uniqueness of Beauville--Laszlo gluing to the vanishing locus of $z$ in $\Spec B_{S,[0,1]}$.
\end{proof}

\subsection{}\label{ss:ArtinSchreierWitt}
We have the following version of non-abelian Artin--Schreier--Witt theory for $\cO_F$. Recall the terminology of $\tau$-modules as in \cite[Definition 12.3.3]{SW20}, and let $n$ be a positive integer. For any $\ul{\cO_F/z^n}$-local system $\bL$ on $S$, write $M(\bL)$ for the $\tau$-module over $\Spec R\lb{z}/z^n$ given by $\bL\otimes_{\ul{\cO_F/z^n}}(\sO_{\Spec R\lb{z}/z^n},\id)$. Conversely, for any $\tau$-module $(M,\phi)$ over $\Spec R\lb{z}/z^n$, write $\bL(M,\phi)$ for the $\ul{\sO_F/z^n}$-sheaf over $\Spec{R}$ given by $\ul{\Hom}_{\tau\text{-mod}}((\sO_{\Spec R\lb{z}/z^n},\id),(M,\phi))$.
\begin{prop*}
  Our $M(-)$ yields an exact tensor equivalence of categories
  \begin{align*}
    \{\ul{\cO_F/z^n}\mbox{-local systems on }S\}\lra^\sim\{\tau\mbox{-modules over }\Spec R\lb{z}/z^n\}.
  \end{align*}
  Consequently, $\bL\mapsto\bL\otimes_{\ul{\cO_F}}(\sO_{\Spec R\lb{z}},\id)$ is an exact tensor equivalence of categories
  \begin{align*}
    \{\ul{\cO_F}\mbox{-local systems on }S\}\lra^\sim\{\tau\mbox{-modules over }\Spec R\lb{z}\}.
  \end{align*}
\end{prop*}
\begin{proof}
  Note that $M(-)$ is left adjoint to $\bL(-)$, and the unit $\id\ra\bL(M(-))$ is an isomorphism. So we just need to prove that $M(-)$ is essentially surjective. Because $\ul{\cO_F/z^n}$-local systems are trivial after a finite \'etale cover, it suffices to prove that the same holds for $\tau$-modules over $\Spec R\lb{z}/z^n$.

  So let $(M,\phi)$ be a $\tau$-module over $\Spec R\lb{z}/z^n$ such that $M$ has rank $h$. When $n=1$, the desired result is \cite[Lemma 3.2.7]{KL15}. For $n\geq2$, by induction there exists a finite \'etale cover $\Spec R'\ra \Spec R$ such that the pullback of $(M, \phi)$ to $\Spec R'\lb{z}/z^{n-1}$ has a basis fixed by $\phi_{R'\lb{z}/z^{n-1}}$. Nakayama's lemma shows that any lift of this basis to $R'\lb{z}/z^n$ yields a basis of $M\otimes_RR'$. In these coordinates, we see that $\phi^{-1}_{R'\lb{z}/z^n}$ acts by $A\circ\tau$, where $A$ in $\GL_h(R'\lb{z}/z^n)$ satisfies $A\equiv 1\pmod{z^{n-1}}$.

  Write $\Spec\wt{R}$ for the vanishing locus in $\Spec R'[u_{ab}]_{1\leq a,b\leq h}$ of the matrix 
  \begin{align*}
   \tau(U)-U-\textstyle\frac1{z^{n-1}}(A-1),
  \end{align*}
where $U$ denotes the matrix with entries $u_{ab}$. Examining entrywise shows that $\wt{R}$ is finite over $R'$, the Jacobian criterion shows that $\wt{R}$ is \'etale over $R'$, and checking on fibers shows that $\Spec\wt{R}\ra\Spec R'$ is surjective. Finally, on $\wt{R}\lb{z}/z^n$ we have
  \begin{align*}
(1+z^{n-1}U)A\tau(1+z^{n-1}U)^{-1} = (1+z^{n-1}U)(1+A-1)(1-z^{n-1}U-(A+1)) = 1,
  \end{align*}
  so the basis of $M\otimes_R\wt{R}$ given by $1+z^{n-1}U$ is fixed by $\phi^{-1}_{\wt{R}}$. Therefore the pullback of $(M,\phi)$ to $\Spec\wt{R}\lb{z}/z^n$ is trivial, as desired.
\end{proof}

\subsection{}\label{ss:OFlocalsystemstaumodules}
We can upgrade Proposition \ref{ss:ArtinSchreierWitt} to $G$-bundles as follows. Briefly, let $X$ be a scheme or a sousperfectoid adic space over $\cO_F$, and let $\tau:X\ra X$ be an endomorphism over $\cO_F$. By a \emph{$\tau$-$G$-bundle} over $X$, we mean a $G$-bundle $\sG$ on $X$ along with an isomorphism of $G$-bundles $\phi:\sG\ra^\sim\prescript\tau{}\sG$.

Let $n$ be a positive integer or $\infty$, and define $z^\infty$ to be $0$. For any $\ul{G(\cO_F/z^n)}$-bundle $\bP$ on $S$, by abuse of notation write $M(\bP)$ for the $\tau$-$G$-bundle over $\Spec R\lb{z}/z^n$ given by $\bP\times^{\ul{G(\cO_F/z^n)}}(G,\id)$.
\begin{prop*}
  Our $M(-)$ yields an equivalence of groupoids
  \begin{align*}
    \{\ul{G(\cO_F/z^n)}\mbox{-bundles on }S\}\lra^\sim\{\tau\mbox{-}G\mbox{-bundles over }\Spec R\lb{z}/z^n\}.
  \end{align*}
\end{prop*}
\begin{proof}
The assignment $\bP\mapsto(V\mapsto\bP\times^{\ul{G(\cO_F/z^n)}}\ul{V(\cO_F/z^n)})$ yields a functor
  \begin{align*}
    \{\ul{G(\cO_F/z^n)}\mbox{-bundles on }S\}\lra\left\{\begin{tabular}{c}
    $\cO_F$-linear exact tensor functors \\
   $\Rep_{\cO_F}(G)\ra$\{$\ul{\cO_F/z^n}$-local systems on $S$\}
  \end{tabular}\right\}.
  \end{align*}
  By Proposition \ref{ss:ArtinSchreierWitt} and the Tannakian description of $G$-bundles, the right-hand side is equivalent to the groupoid of $\tau$-$G$-bundles over $\Spec R\lb{z}/z^n$. Now we just need to prove that the above functor is an equivalence of groupoids. Because $\ul{G(\cO_F/z^n)}$-bundles are trivial after a pro-\'etale cover, it suffices to prove that the same holds for objects of the right-hand side.

  So let $\rho:\Rep_{\cO_F}(G)\ra\{\ul{\cO_F/z^n}\mbox{-local systems on }S\}$ be an $\cO_F$-linear exact tensor functor, and let $\wt{S}\ra S$ be a pro-\'etale cover such that $\wt{S}$ is strictly totally disconnected. Then $\ul{\cO_F/z^n}$-local systems on $\wt{S}$ are equivalent to finite projective $\Cont(|\wt{S}|,\cO_F/z^n)$-modules, so $\rho|_{\wt{S}}$ corresponds to a $G$-bundle $\wt\sG$ on
  \begin{align*}
    \Spec\Cont(|\wt{S}|,\cO_F/z^n).
  \end{align*}
  Note that $\Cont(|\wt{S}|,\cO_F/z^n)=\Cont(\pi_0(\wt{S}),\cO_F/z^n)$. For any $s$ in $\pi_0(\wt{S})$, \cite[Lemma 2.2.3]{KL15} indicates that $\varinjlim_U\Cont(U,\cO_F/z^n)$ is Henselian with respect to the kernel of evaluation at $s$, where $U$ runs over neighborhoods of $s$ in $\pi_0(\wt{S})$. Lang's lemma shows that the pullback of $\wt\sG$ to $\Spec\Cont(s,\cO_F/z^n)=\Spec\cO_F/z^n$ is trivial, so Hensel lifting implies that the pullback of $\wt\sG$ to $\Spec\Cont(U,\cO_F/z^n)$ is trivial for some $U$. Therefore $\rho|_{\wt{U}}$ is isomorphic to the canonical fiber functor, where $\wt{U}$ denotes the preimage of $U$ in $|\wt{S}|$. As $s$ varies, this yields a pro-\'etale cover of $S$ where $\rho$ is trivial, as desired.
\end{proof}

\subsection{}\label{ss:relativeintegralRobbaring}
Let us recall the equicharacteristic version of the \emph{(relative integral) Robba ring}. Write $\norm-$ for the spectral norm on $R$, normalized such that $\norm\vpi=\textstyle\frac1q$. For any positive rational $b$, we have a map $\norm-_b:R\lb{z}\ra[0,\infty]$ given by
\begin{align*}
  \textstyle\sum_{m=0}^\infty a_mz^m\mapsto\sup_{m\geq0}\{q^{-m}\norm{a_m}^b\}.
\end{align*}
Evidently $\norm{\tau(-)}_b=\norm-_{qb}$. When $1/b$ lies in $\bZ[\textstyle\frac1p]$, \ref{ss:vectorbundlesfullyfaithful} shows that the restriction of $\norm-_b$ to $B_{S,[0,b]}\subseteq R\lb{z}$ is a norm and induces the usual topology on $B_{S,[0,b]}$. Moreover, $\textstyle\sum_{m=0}^\infty a_mz^m$ lies in $B_{S,[0,b]}$ if and only if $\norm{a_mz^m}_b\to0$.

Write $\wt\cR^{\mathrm{int}}_R$ for $\varinjlim_bB_{S,[0,b]}$, where $b$ runs over positive rationals. Note that any multiple $f$ of $z$ in $\wt\cR^{\mathrm{int}}_R$ satisfies $\norm{f}_b<1$ for small enough $b$, so the completeness of $B_{S,[0,b]}$ implies that $z$ lies in the Jacobson radical of $\wt\cR^{\mathrm{int}}_R$.

\subsection{}\label{ss:proetalelocalphibasis}
Just like $\ul{\cO_F}$-local systems, we show that $\tau$-modules over the Robba ring are trivial after a pro-finite \'etale cover.
\begin{lem*}
Let $(\wt{M},\wt\phi)$ be a $\tau$-module over $\Spec\wt\cR^{\mathrm{int}}_R$ such that $\wt{M}$ is free of rank $h$. Then there exists a pro-finite \'etale cover $\Spa(\wt{R},\wt{R}^+)\ra S$ such that the pullback of $(\wt{M},\wt\phi)$ to $\Spec\wt\cR^{\mathrm{int}}_{\wt{R}}$ is trivial.
\end{lem*}
\begin{proof}
Proposition \ref{ss:ArtinSchreierWitt} enables us to assume that the pullback of $(\wt{M},\wt\phi)$ to $\Spec R$ has a basis fixed by $\wt\phi_R$. Now \ref{ss:relativeintegralRobbaring} and Nakayama's lemma show that any lift of this basis yields a basis of $\wt{M}$, and in these coordinates, we see that $\wt\phi^{-1}$ acts by $A\circ\tau$, where $A$ in $\GL_h(\wt\cR^{\mathrm{int}}_R)$ satisfies $A\equiv 1\pmod{z}$. Proposition \ref{ss:ArtinSchreierWitt} yields a pro-finite \'etale cover $\Spa(\wt{R},\wt{R}^+)\ra\Spa(R,R^+)$ such that the pullback of $(\wt{M},\wt\phi)$ to $\Spec\wt{R}\lb{z}$ has a basis fixed by $(\wt\phi)_{\wt{R}\lb{z}}$. Since the pullback of $(\wt{M},\wt\phi)$ to $\Spec{R}$ is already trivial, we can choose this basis of $\wt{M}\otimes_{\wt\cR^{\mathrm{int}}_R}\wt{R}\lb{z}$ such that its matrix $U$ in $\GL_h(\wt{R}\lb{z})$ satisfies $U\equiv1\pmod{z}$. Now we just need to prove that $U$ lies in $\GL_h(\wt\cR^{\mathrm{int}}_{\wt{R}})$.

As $A-1$ is divisible by $z$, we have $\norm{A-1}_b<1$ for small enough positive rational $b$. Write $C\coloneqq\max\{q^{-1},\norm{A-1}_b\}<1$, write $U_n$ for the mod-$z^n$ truncation of $U$, and write $X_n$ for the $z^n$-coefficient of $U$. For any positive integer $n$, we claim that
  \begin{align*}
    \norm{z^nX_n}_{qb},\,\norm{U_n-1}_b,\mbox{ and }\norm{U_n-1}_{qb}\leq C.
  \end{align*}
  When $n=1$, the last two bounds hold because $U_1=1$. For general $n$, we have
  \begin{align*}
    &U_n+z^nX_n\equiv U\equiv A\tau(U)\equiv A(\tau(U_n)+z^n\tau(X_n))\pmod{z^{n+1}} \\
    \implies& z^n(X_n-A\tau(X_n))\equiv (A-1)\tau(U_n)+(\tau(U_n)-1)-(U_n-1)\pmod{z^{n+1}} \\
    \implies& X_n-\tau(X_n)\equiv \textstyle\frac1{z^n}\big[(A-1)\tau(U_n)+(\tau(U_n)-1)-(U_n-1)\big]\pmod{z}.
  \end{align*}
By evaluating this equation at rank-$1$ points of $S$ and considering the Newton polygon of its entries, induction on $n$ implies that
\begin{align*}
  \norm{X_n}_b&\leq\max\{1,(q^n\norm{(A-1)\tau(U_n)+\tau(U_n-1)-(U_n-1)}_b)^{1/q}\} \\
  &\leq\max\{1,(q^nC)^{1/q}\} \leq (q^nC)^{1/q}.
\end{align*}
Therefore $\norm{z^nX_n}_{qb}\leq C$, so $\norm{U_{n+1}-1}_{qb}\leq C$. Since $C\geq q^{-n}$, we also get
  \begin{align*}
    \norm{U_{n+1}-1}_b \leq\max\{\norm{z^nX_n}_b,\norm{U_n-1}_b\} \leq \max\{q^{-n}(q^nC)^{1/q},C\} \leq C,
  \end{align*}
  which concludes our proof of the claim.

By \ref{ss:relativeintegralRobbaring}, the claim implies that $U$ has coefficients in $B_{S,[0,b']}$ for any positive rational $b'<qb$ such that $1/b'$ lies in $\bZ[\textstyle\frac1p]$. After decreasing $b'$ such that $b'<b$, the claim also implies that $U$ is invertible over $B_{S,[0,b']}$. Therefore $U$ indeed lies in $\GL_h(\wt\cR^{\mathrm{int}}_{\wt{R}})$, as desired.
\end{proof}

\subsection{}\label{ss:RobbabundleslocalonR}
Vector bundles on the Robba ring are local on $S$ in the following sense. Let $(S_\al)_\al$ be a finite cover of $S$ by rational open subspaces, where $S_\al=\Spa(R_\al,R_\al^+)$. Write $S_{\al\be}=\Spa(R_{\al\be},R_{\al\be}^+)$ for their pairwise intersections, and write $S_{\al\be\ga}=\Spa(R_{\al\be\ga},R_{\al\be\ga}^+)$ for their triple intersections.
\begin{lem*}
Pullback yields an equivalence from the category of vector bundles on $\Spec\wt\cR^{\mathrm{int}}_R$ to the category of vector bundles on the $\Spec\wt\cR^{\mathrm{int}}_{R_\al}$ with transition morphisms on the $\Spec\wt\cR^{\mathrm{int}}_{R_{\al\be}}$ whose pullbacks to $\Spec\wt\cR^{\mathrm{int}}_{R_{\al\be\ga}}$ satisfy the cocycle condition. Moreover, for any vector bundle $M$ on $\Spec\wt\cR^{\mathrm{int}}_R$, there exists $(S_\al)_\al$ as above such that $M|_{\Spec\wt\cR^{\mathrm{int}}_{R_\al}}$ is trivial for all $\al$.
\end{lem*}
\begin{proof}
  Because $\wt\cR^{\mathrm{int}}_R=\textstyle\varinjlim_b B_{S,[0,b]}$, we have an equivalence of categories
  \begin{align*}
    \textstyle\varinjlim_b\{\mbox{vector bundles on }\Spec B_{S,[0,b]}\}\lra^\sim\{\mbox{vector bundles on }\Spec\wt\cR^{\mathrm{int}}_R\}.
  \end{align*}
  When $1/b$ lies in $\bZ[\textstyle\frac1p]$, the $B_{S,[0,b]}$ are Tate algebras over $R$. Hence $S\mapsto B_{S,[0,b]}$ commutes with rational localization on $S$. Applying \cite[Theorem 2.7.7]{KL15} to the resulting open cover of $\cY_{S,[0,b]}$ by $(\cY_{S_\al,[0,b]})_\al$ shows that vector bundles on $\Spec B_{S,[0,b]}$ are equivalent to vector bundles on the $\Spec B_{S_\al,[0,b]}$ with transition morphisms on the $\Spec B_{S_{\al\be},[0,b]}$ whose pullbacks to $\Spec B_{S_{\al\be\ga},[0,b]}$ satisfy the cocycle condition. Because there are finitely many $\al$, taking the directed limit over $b$ yields the first claim.

  For the second claim, \cite[Theorem 2.7.7]{KL15} shows that there exists $(S_\al)_\al$ as above such that the pullback of $M$ to $\Spec R_\al$ is trivial for all $\al$. Since $z$ lies in the Jacobson radical of $\wt\cR^{\mathrm{int}}_{R_\al}$, any trivialization lifts to $\wt\cR^{\mathrm{int}}_{R_\al}$ by Nakayama's lemma.
\end{proof}

\subsection{}\label{ss:Robbaringpowerseries}
We conclude by showing that $\tau$-modules on $R\lb{z}$ uniquely descend to the Robba ring.
\begin{thm*}
  Pullback yields an exact tensor equivalence of categories
  \begin{align*}
    \{\tau\mbox{-modules over }\Spec\wt\cR^{\mathrm{int}}_R\}\lra^\sim\{\tau\mbox{-modules over }\Spec R\lb{z}\}.
  \end{align*}
  Consequently, pullback induces an equivalence of groupoids
  \begin{align*}
    \{\tau\mbox{-}G\mbox{-bundles over }\Spec\wt\cR^{\mathrm{int}}_R\}\lra^\sim\{\tau\mbox{-}G\mbox{-bundles over }\Spec R\lb{z}\}.
  \end{align*}
\end{thm*}
\begin{proof}
  First, we tackle full faithfulness. By considering internal homs for $\tau$-modules, it suffices to prove that, for any $\tau$-module $(\wt{M},\wt\phi)$ over $\Spec\wt\cR^{\mathrm{int}}_R$, any $m$ in $\wt{M}\otimes_{\wt\cR^{\mathrm{int}}_R}R\lb{z}$ that is fixed by $\wt\phi_{R\lb{z}}$ lies in $\wt{M}$. Lemma \ref{ss:RobbabundleslocalonR} implies that it suffices to prove this after passing to an open cover of $S$, so we can assume that $\wt{M}$ is free of rank $h$. Then Lemma \ref{ss:proetalelocalphibasis} yields a pro-finite \'etale cover $\Spa(\wt{R},\wt{R}^+)\ra S$ such that the pullback of $(\wt{M},\wt\phi)$ to $\Spec\wt\cR^{\mathrm{int}}_{\wt{R}}$ has a basis fixed by $\wt\phi_{\wt\cR^{\mathrm{int}}_{\wt{R}}}$. In these coordinates, the entries of $m$ lie in $(\wt{R}^\tau)\lb{z}$, which lies in $\wt\cR^{\mathrm{int}}_{\wt{R}}$ by \ref{ss:relativeintegralRobbaring}. Note that the intersection of $R\lb{z}$ and $\wt\cR^{\mathrm{int}}_{\wt{R}}$ equals $\wt\cR^{\mathrm{int}}_R$, so the flatness of $\wt{M}$ shows that $m$ lies in $\wt{M}$.

  As for essential surjectivity, let $(M,\phi)$ be a $\tau$-module over $\Spec R\lb{z}$. By passing to a clopen cover of $S$, we can assume that $M$ has rank $h$. Proposition \ref{ss:ArtinSchreierWitt}, full faithfulness, and finite \'etale descent enable us to assume that the pullback of $(M,\phi)$ to $\Spec R$ has a basis fixed by $\phi_R$. Nakayama's lemma shows that any lift of this basis yields a basis of $M$, and in these coordinates, we see that $\phi^{-1}$ acts by $A\circ\tau$, where $A$ in $\GL_h(R\lb{z})$ satisfies $A\equiv1\pmod{z}$.

Let $n$ be a positive integer. We inductively construct certain $C_n$, $B_n$, and $U_n$ in $\GL_h(R\lb{z})$ such that $C_n-B_n$ is divisible by $z^n$. First, set $C_1\coloneqq A$ and $B_1\coloneqq1$. For general $n$, write $X_n$ for the $z^n$-coefficient of $C_n-B_n$. There exists $Y_n$ in $\Mat_h(R)$ satisfying $\norm{X_n+Y_n-\tau(Y_n)}_1<q^{n/2}$ \cite[Lemma 8.5.2]{KL15}, which we use to define
  \begin{align*}
    U_n\coloneqq1+z^nY_n,\,C_{n+1}\coloneqq U_nC_n\tau(U_n)^{-1},\mbox{ and }B_{n+1}\coloneqq B_n+z^n(X_n+Y_n-\tau(Y_n)).
  \end{align*}
By induction, we have
  \begin{align*}
    C_{n+1}&\equiv (1+z^nY_n)C_n(1-z^n\tau(Y_n))\\
    &\equiv B_n+z^n(X_n+Y_n-\tau(Y_n))\equiv B_{n+1}\pmod{z^{n+1}},
  \end{align*}
  as desired.

  We see from \ref{ss:relativeintegralRobbaring} that the $B_n$ converge to a matrix $B$ in $\GL_h(B_{S,[0,1]})$. Now the $C_n$ converge to a matrix $C$ in $\GL_h(R\lb{z})$, and because $C_n-B_n$ is divisible by $z^n$, we have $C=B$. Moreover, the infinite product $U\coloneqq U_1U_2\dotsm$ converges to a matrix $U$ in $\GL_h(R\lb{z})$, and the above shows that $UA\tau(U)^{-1}= C = B$. Thus the basis of $M$ given by $U$ descends $(M,\phi)$ to a $\tau$-module over $\Spec\wt\cR^{\mathrm{int}}_R$, as desired.

Finally, we show that pullback has an exact tensor quasi-inverse. Note that we have a commutative triangle
  \begin{align*}
    \xymatrixcolsep{-.5in}
    \xymatrix{ & \ar[dl]\{\ul{\cO_F}\mbox{-local systems on }S\}\ar[dr]^-{M(-)} & \\
    \{\tau\mbox{-modules over }\Spec\wt\cR^{\mathrm{int}}_R\}\ar[rr] & & \{\tau\mbox{-modules over }\Spec R\lb{z}\}.
    }
  \end{align*}
  Every arrow is an exact tensor functor, and $M(-)$ is an exact tensor equivalence by Proposition \ref{ss:ArtinSchreierWitt}. Hence its quasi-inverse $\bL(-)$ postcomposed with the left arrow yields an exact tensor quasi-inverse to pullback.
\end{proof}

\section{Analytic moduli of local shtukas}\label{s:analyticlocalshtukas}
In this section, we define local shtukas in the analytic setting and compare them with the formal variant from \S\ref{s:formallocalshtukas}. We start by giving an algebraic version of local shtukas over a perfectoid space, which is the equicharacteristic analogue of Breuil--Kisin--Fargues modules. This mediates between the formal variant and more analytic variants. Next, we define an analytic version of local shtukas, as well as the corresponding moduli problem. Using results from \S\ref{s:zadichodgetheory}, we show that the analytic moduli problem agrees with the formal moduli problem from \S\ref{s:formallocalshtukas}.

From here, we define the covering tower for our analytic moduli problem. We conclude by recalling the moduli of local shtukas appearing in Fargues--Scholze \cite{FS21}, which is defined purely in terms of the Fargues--Fontaine curve. While this subtly differs from our analytic moduli problem, their intersection homology complexes are naturally isomorphic, which is all we need.

\subsection{}\label{ss:algebraiclocalshtuka}
Let $S=\Spa(R,R^+)$ be an affinoid perfectoid space over $\bD^I$. For any $i$ in $I$, if $\ze_i$ is an $R^{\circ\circ}$-multiple of $\vpi^r$, then
\begin{align*}
\frac1{z-\ze_i} = \frac1z\sum_{n=0}^\infty\left(\frac{\ze_i}z\right)^n
\end{align*}
lies in $R^+\ba{z,\textstyle\frac{\vpi^r}z}[\frac1z]$. As $\ze_i$ is topologically nilpotent, this always holds for small enough $r$.

Recall the $\mu_i$ and $\bD_i$ from \ref{ss:affineschubertvarieties}, and recall Definition \ref{ss:localshtukas}. We use Definition \ref{ss:localshtukas} to define an algebraic version of local $G$-shtukas over $S$.
\begin{defn*}\hfill
  \begin{enumerate}[a)]
  \item An \emph{algebraic local $G$-shtuka} over $S$ is a local $G$-shtuka over $\Spec R^+$.
  \item Suppose that $S$ lies over $\prod_{i\in I}\bD_i$, and let $\sG$ be an algebraic local shtuka over $S$. We say that $\sG$ is \emph{bounded by $\mu_\bullet$} if the corresponding local $G$-shtuka over $\Spec R^+$ is bounded by $\mu_\bullet$.
  \item Let $\sG$ and $\sG'$ be algebraic local $G$-shtukas over $S$. A \emph{quasi-isogeny} from $\sG$ to $\sG'$ consists of, for some small enough positive $r$ in $\bZ[\textstyle\frac1p]$ and all $1\leq j\leq k$, an isomorphism of $G$-bundles
    \begin{align*}
      \de_j:\sG_j|_{\Spec R^+\ba{z,\frac{\vpi^r}z}[\frac1z]}\ra^\sim\sG_j'|_{\Spec R^+\ba{z,\frac{\vpi^r}z}[\frac1z]}
    \end{align*}
    such that the diagram
    \begin{align*}
      \xymatrixcolsep{1in}
      \xymatrix{\sG_j|_{\Spec R^+\ba{z,\frac{\vpi^r}z}[\frac1z]}\ar[r]^-{(\phi_j)_{R^+\ba{z,\frac{\vpi^r}z}[\frac1z]}}\ar[d]^-{\de_j} & \sG_{j+1}|_{\Spec R^+\ba{z,\frac{\vpi^r}z}[\frac1z]}\ar[d]^-{\de_{j+1}}\\
      \sG_j'|_{\Spec R^+\ba{z,\frac{\vpi^r}z}[\frac1z]}\ar[r]^-{(\phi_j')_{R^+\ba{z,\frac{\vpi^r}z}[\frac1z]}} & \sG'_{j+1}|_{\Spec R^+\ba{z,\frac{\vpi^r}z}[\frac1z]}}
    \end{align*}
    commutes, where $\de_{k+1}$ denotes the isomorphism $\prescript\tau{}{\de_1}$.
  \end{enumerate}
\end{defn*}

\subsection{}\label{ss:quasiisogenylimit}
Let $n$ be a non-negative integer, and note that $R^+/\vpi^n$ is a discrete $\bF_q\lb{\ze_i}_{i\in I}$-algebra. For any algebraic local shtuka $\sG$ over $S$, write $\sG^n$ for the local shtuka over $S_n\coloneqq\Spec R^+/\vpi^n$ given by pullback. Since $R^+\ba{z,\textstyle\frac{\vpi^r}z}[\frac1z]/\vpi^n$ equals $(R^+/\vpi^n)\lp{z}$, quasi-isogenies of algebraic local $G$-shtukas over $S$ pull back to quasi-isogenies of local $G$-shtukas over $S_n$.

Lemma \ref{ss:nilpotentalgebraization} shows that bounded algebraic local $G$-shtukas are all captured by this limit process. The following lemma shows that quasi-isogenies between them are also all captured by this limit process.
\begin{lem*}
  Suppose that $S$ lies over $\prod_{i\in I}\bD_i$, and let $\sG$ and $\sG'$ be algebraic local $G$-shtukas over $S$ bounded by $\mu_\bullet$. Then pullback yields a bijection
\begin{align*}
\{\mbox{quasi-isogenies from }\sG\mbox{ to }\sG'\}\ra^\sim\textstyle\varprojlim_n\{\mbox{quasi-isogenies from }\sG^n\mbox{ to }\sG'^n\}.
\end{align*}
\end{lem*}
\begin{proof}
  Let $(\de^n)_{n\geq0}$ be a compatible system of quasi-isogenies from $\sG^n$ to $\sG'^n$. Because $\varprojlim_n(R^+/\vpi^n)\lp{z}$ equals $R^+\ba{z,\textstyle\frac1z}$, we see that $\de_j\coloneqq\varprojlim_n\de^n_j$ yields an isomorphism of $G$-bundles $\sG_j|_{\Spec R^+\ba{z,\frac1z}}\ra^\sim\sG'_j|_{\Spec R^+\ba{z,\frac1z}}$ for all $1\leq j\leq k$. Now $\de^0$ is bounded by $m$ for some non-negative integer $m$ as in Definition \ref{ss:quasiisogenies}.b), so Proposition \ref{ss:rigidisogeny} yields a non-negative integer $B$ such that $\de^n$ is bounded by $m+B\lceil\log_qn\rceil$. From here, the Tannakian description of $G$-bundles implies that $\de_j$ naturally descends to an isomorphism of $G$-bundles
  \begin{align*}
    \sG_j|_{\Spec R^+\ba{z,\frac{\vpi^r}z}[\frac1z]}\ra^\sim\sG_j'|_{\Spec R^+\ba{z,\frac{\vpi^r}z}[\frac1z]}
  \end{align*}
for any positive $r$ in $\bZ[\textstyle\frac1p]$. By taking $r$ small enough such that $\textstyle\frac1{z-\ze_i}$ lies in $R^+\ba{z,\frac{\vpi^r}z}[\frac1z]$ for all $i$ in $I$, the commutativity of the square in Definition \ref{ss:algebraiclocalshtuka}.c) follows from the commutativity of the analogous square in Definition \ref{ss:quasiisogenies}.a).
\end{proof}

\subsection{}\label{ss:analyticBDgrassmannian}
Before introducing the analytic version of local $G$-shtukas, we need some notation on the $B_{\dR}$-affine Grassmannian. Write $B^+_{\dR}(S)$ for the ring of global sections of the completion of $\sO_{\cY_S}$ along $\sum_{i\in I}\Ga_i$, and write $B_{\dR}^j(S)$ for the version that is punctured along $\sum_{i\in I_j}\Ga_i$.
\begin{defn*}\hfill
  \begin{enumerate}[a)]
  \item Write $\cL^n_IG$ and $\cL^+_IG$ for the small v-sheaves over $(\bD^I)^\Diamond$ given by sending $S$ to $G(\sO_{n\sum_{i\in I}\Ga_i})$ and $G(B^+_{\dR}(S))$, respectively.
  \item   Write $\cGr^{(I_1,\dotsc,I_k)}_G$ for the small v-sheaf over $(\bD^I)^\Diamond$ whose $S$-points parametrize data consisting of
  \begin{enumerate}[i)]
  \item for all $1\leq j\leq k$, a $G$-bundle $\sG_j$ on $\Spec B^+_{\dR}(S)$,
  \item for all $1\leq j\leq k$, an isomorphism of $G$-bundles
    \begin{align*}
      \phi_j:\sG_j|_{\Spec B^j_{\dR}(S)}\ra^\sim\sG_{j+1}|_{\Spec B^j_{\dR}(S)},
    \end{align*}
    where $\sG_{k+1}$ denotes the trivial $G$-bundle.
  \end{enumerate}
  \end{enumerate}
\end{defn*}

\subsection{}\label{ss:analytificationfunctorofpoints}
In certain cases, we can describe the functor of points of (generalized) analytifications without analytically sheafifying. Briefly, let $A$ be a noetherian ring, and let $X$ be a scheme locally of finite type over $Z\coloneqq\Spec{A}$. Let $J\subseteq A$ be an ideal, write $\wh{A}$ for the completion of $A$ with respect to $J$, and write $\wh{Z}$ for the adic space $\Spa\wh{A}$. Write $X_{\wh{Z}}$ for the fiber product as in \cite[(3.8)]{Hub94}.
\begin{lem*}
Suppose that $X$ is quasi-projective over $Z$. For any analytic affinoid adic space $S=\Spa(R,R^+)$, the $S$-points of $X_{\wh{Z}}$ consist of the $R$-points of $X$ such that the resulting ring homomorphism $A\ra R$ is continuous for the $J$-adic topology on $A$.
\end{lem*}
\begin{proof}
The universal property of $X_{\wh{Z}}$ \cite[(3.8)]{Hub94} indicates that an $S$-point of $X_{\wh{Z}}$ is equivalent to a morphism $S\ra\wh{Z}$ of adic spaces along with a morphism $S\ra X$ of locally ringed spaces such that, in the category of locally ringed spaces, the square
  \begin{align*}
    \xymatrix{S\ar[r]\ar[d] & X\ar[d] \\
    \wh{Z}\ar[r] & Z}
  \end{align*}
commutes. The $\Spec$-global sections adjunction shows that $S\ra X\ra Z$ yields a ring homomorphism $A\ra R$, and note that the commutativity of this square is equivalent to $A\ra R$ being continuous for the $J$-adic topology on $A$.

Now assume that $X=\bP^N_Z$. Since $Z$ is affine, the $\Spec$-global sections adjunction implies that $S\ra X$ is equivalent to the data of a line bundle $\sL$ on $S$ along with sections $s_0,\dotsc,s_N$ that generate $\sL$. By \cite[Theorem 1.4.2]{Ked19}, this is equivalent to a finite projective $R$-module $M$ of rank $1$ along with elements $r_0,\dotsc,r_N$ that generate $M$, which is precisely the data of an $R$-point of $X$.

In general, $X$ is a locally closed subscheme of $\bP^N_Z$. Because $Z$ is noetherian, there exist finitely many homogeneous polynomials $f_1,\dotsc,f_l$ and $g_1,\dotsc,g_m$ in $A[t_0,\dotsc,t_N]$ such that $X\subseteq\bP^N_Z$ is the locus where $f_a(s_0,\dotsc,s_N)$ vanishes for all $1\leq a\leq l$ and $g_b(s_0,\dotsc,s_N)$ does not vanish for all $1\leq b\leq m$. These properties are preserved by \cite[Theorem 1.4.2]{Ked19}, so we see that $S\ra X$ is equivalent to an $R$-point of $X$.
\end{proof}

\subsection{}\label{ss:BDgrassmanniancomparison}
We check that the $B_{\dR}$-affine Grassmannian and its affine Schubert varieties are the analytifications of their algebraic counterparts. Write $S^{\alg}$ for the $R$-point of $C^I$ given by $\Spec{R}\ra\Spec\bF_q\lb{\ze_i}_{i\in I}\ra C^I$, and write $\Ga_i^{\alg}$ for the resulting relative effective Cartier divisor on $C\times S$ as in \ref{ss:globalBDgrassmannian}. Recall the $F_i$ from \ref{ss:affineschubertvarieties}.
\begin{lem*}
  We have a natural isomorphism of rings $\sO_{n\sum_{i\in I}\Ga_i^{\alg}}\cong\sO_{n\sum_{i\in I}\Ga_i}$. Consequently, we obtain natural isomorphisms from $(L^n_I(G_C))_{\bD^I}^\Diamond$ and $(L^+_I(G_C))_{\bD^I}^\Diamond$ to $\cL^n_I(G)$ and $\cL^+_I(G)$, respectively, and we may view $(\wh\Gr^{(I_1,\dotsc,I_k)}_{G,\mu_\bullet}|_{\prod_{i\in I}\bD_i})^\Diamond$ as a closed subsheaf
  \begin{align*}
    \cGr^{(I_1,\dotsc,I_k)}_{G,\mu_\bullet}|_{\prod_{i\in I}\bD_i^\Diamond}\subseteq\cGr^{(I_1,\dotsc,I_k)}_G|_{\prod_{i\in I}\bD_i^\Diamond}.
  \end{align*}
  Finally, the $S$-points of $\cGr^{(I_1,\dotsc,I_k)}_{G,\mu_\bullet}|_{\prod_{i\in I}\Spd F_i}$ consist of the $((\sG_j)_{j=1}^k,(\phi_j)_{j=1}^k)$ such that, for all geometric points $\ov{s}$ of $S$ and $1\leq j\leq k$, the relative position of $\phi_{j,\ov{s}}$ at $\Ga_{i,\ov{s}}$ is bounded by $\textstyle\sum_{i'}\mu_{i'}$, where $i'$ runs over elements of $I$ satisfying $\Ga_{i',\ov{s}}=\Ga_{i,\ov{s}}$.
\end{lem*}
\begin{proof}
  The first claim is immediate, which identifies $(L^n_I(G_C))_{\bD^I}^\Diamond$ with $\cL^n_I(G)$. The first claim also induces isomorphisms $\wh\cO_C(S^{\alg})\cong B^+_{\dR}(S)$ and $\wh\cO^{j,\circ}_C(S^{\alg})\cong B_{\dR}(S)$, which identifies $(L^+_I(G_C))_{\bD^I}^\Diamond$ with $\cL^+_I(G)$. This also shows that, for any presentation of $\Gr^{(I_1,\dotsc,I_k)}_{G_C}$ as a directed limit $\varinjlim_lX_l$ of projective schemes $X_l$ over $C^I$, we have
  \begin{align*}
    \cGr^{(I_1,\dotsc,I_k)}_G(S) = \Gr^{(I_1,\dotsc,I_k)}_{G_C}(S^{\alg}) = (\textstyle\varinjlim_lX_l)(S^{\alg}) = \varinjlim_lX_l(S^{\alg}) = \varinjlim_l(X_l)^\Diamond_{\bD^I}(S),
  \end{align*}
  where the last two equalities follow from \cite[Lemma 5.4]{HV11} and Lemma \ref{ss:analytificationfunctorofpoints}, respectively. Now \ref{ss:affineschubertvarieties} indicates that $\Gr^{(I_1,\dots,I_k)}_{G_C,\mu_\bullet}|_{\prod_{i\in I}C_i}$ is a closed subscheme of $X_l|_{\prod_{i\in I}C_i}$ for large enough $l$. Since $\Gr^{(I_1,\dots,I_k)}_{G_C,\mu_\bullet}|_{\prod_{i\in I}C_i}$ is projective over $\prod_{i\in I}C_i$, the natural morphism of adic spaces $\wh\Gr^{(I_1,\dotsc,I_k)}_{G,\mu_\bullet}|_{\prod_{i\in I}\bD_i}\ra(\Gr^{(I_1,\dots,I_k)}_{G_C,\mu_\bullet})_{\prod_{i\in I}\bD_i}$ is an isomorphism \cite[(4.6.iv.d)]{Hub94}. Hence taking $(-)^\Diamond$ yields the desired closed subsheaf
  \begin{align*}
        \cGr^{(I_1,\dotsc,I_k)}_{G,\mu_\bullet}|_{\prod_{i\in I}\bD_i^\Diamond}\subseteq\cGr^{(I_1,\dotsc,I_k)}_G|_{\prod_{i\in I}\bD_i^\Diamond}.
  \end{align*}
Finally, the description of $\cGr^{(I_1,\dotsc,I_k)}_{G,\mu_\bullet}|_{\prod_{i\in I}\Spd F_i}$ follows from \ref{ss:affineschubertvarieties}.
\end{proof}

\subsection{}\label{ss:analyticlocalshtukas}
Now, we can define an analytic version of local $G$-shtukas over $S$. Let $a$ in $\bZ[\textstyle\frac1p]$ be non-negative. For any $i$ in $I$, if $\ze_i^a$ is an $R^{\circ\circ}$-multiple of $\vpi$, then $\rad(\Ga_i)$ lie in $[0,a)$. As $\ze_i$ is topologically nilpotent, this always holds for large enough $a$.
\begin{defn*}\hfill
  \begin{enumerate}[a)]
  \item An \emph{analytic local $G$-shtuka} over $S$ consists of
    \begin{enumerate}[i)]
    \item for all $1\leq j\leq k$, a $G$-bundle $\sG_j$ on $\cY_{S,[0,\infty)}$,
    \item for all $1\leq j\leq k$, an isomorphism of $G$-bundles
      \begin{align*}
        \phi_j:\sG_j|_{\cY_{S,[0,\infty)}\ssm\sum_{i\in I_j}\Ga_i}\ra^\sim\sG_{j+1}|_{\cY_{S,[0,\infty)}\ssm\sum_{i\in I_j}\Ga_i},
      \end{align*}
      that is meromorphic along $\sum_{i\in I_j}\Ga_i$, where $\sG_{k+1}$ denotes the $G$-bundle $\prescript\tau{}{\sG_1}$.
    \end{enumerate}
  \item Suppose that $S$ lies over $\prod_{i\in I}\bD_i$, and let $\sG$ be an analytic local $G$-shtuka over $S$. We say that $\sG$ is \emph{bounded by $\mu_\bullet$} if, for any affinoid perfectoid \'etale cover $\Spa(\wt{R},\wt{R}^+)\ra S$ where $\prescript\tau{}{\sG_1}|_{\cY_{\Spa(\wt{R},\wt{R}^+),[0,\infty)}}$ is trivial and any trivialization $t:\prescript\tau{}{\sG_1}|_{\cY_{\Spa(\wt{R},\wt{R}^+),[0,\infty)}}\ra^\sim G$, the $\Spa(\wt{R},\wt{R}^+)$-point of $\cGr^{(I_1,\dotsc,I_k)}_G|_{\prod_{i\in I}\bD_i^\Diamond}$ given by
  \begin{align*}
    \xymatrixcolsep{0.75in}
    \xymatrix{\sG_1|_{\Spec B^+_{\dR}(\wt{R})}\ar@{-->}[r]^-{(\phi_1)_{B^1_{\dR}(\wt{R})}} & \dotsm\ar@{-->}[r]^-{(\phi_{k-1})_{B^{k-1}_{\dR}(\wt{R})}} & \sG_k|_{B^+_{\dR}(\wt{R})}\ar@{-->}[r]^-{(t\circ\phi_k)_{B^k_{\dR}(\wt{R})}} & G}
  \end{align*}
lies in $\cGr^{(I_1,\dotsc,I_k)}_{G,\mu_\bullet}|_{\prod_{i\in I}\bD_i^\Diamond}$.
    
  \item Let $\sG$ and $\sG'$ be analytic local $G$-shtukas over $S$. A \emph{quasi-isogeny} from $\sG$ to $\sG'$ consists of, for some large enough rational $a$ and all $1\leq j\leq k$, an isomorphism of $G$-bundles
    \begin{align*}
      \de_j:\sG_j|_{\cY_{S,[a,\infty)}}\ra^\sim\sG_j'|_{\cY_{S,[a,\infty)}}
    \end{align*}
    such that the diagram
    \begin{align*}
      \xymatrixcolsep{0.75in}
            \xymatrix{\sG_j|_{\cY_{S,[a,\infty)}}\ar[r]^-{(\phi_j)_{\cY_{S,[a,\infty)}}}\ar[d]^-{\de_j} & \sG_{j+1}|_{\cY_{S,[a,\infty)}}\ar[d]^-{\de_{j+1}}\\
      \sG_j'|_{\cY_{S,[a,\infty)}}\ar[r]^-{(\phi_j')_{\cY_{S,[a,\infty)}}} & \sG'_{j+1}|_{\cY_{S,[a,\infty)}}}
    \end{align*}
    commutes, where $\de_{k+1}$ denotes the isomorphism $\prescript\tau{}{\de_1}$.
  \end{enumerate}
\end{defn*}
It suffices to check Definition \ref{ss:analyticlocalshtukas}.b) for a single $\Spa(\wt{R},\wt{R}^+)\ra S$ and $t$.

\subsection{}
We now define the analytic moduli of local $G$-shtukas.
\begin{defn*}
  Write $\cLocSht^{(I_1,\dotsc,I_k)}_{G,\mu_\bullet}|_{\prod_{i\in I}\bD_i^\Diamond}$ for the small v-sheaf over $\textstyle\prod_{i\in I}\bD_i^\Diamond$ whose $S$-points parametrize data consisting of
  \begin{enumerate}[i)]
  \item an analytic local $G$-shtuka over $S$ bounded by $\mu_\bullet$,
  \item a quasi-isogeny $\de$ from $\sG$ to the trivial analytic local $G$-shtuka $G$.
  \end{enumerate}
Write $f^\cL:\cLocSht^{(I_1,\dotsc,I_k)}_{G,\mu_\bullet}|_{\prod_{i\in I}\bD_i^\Diamond}\ra\textstyle\prod_{i\in I}\bD_i^\Diamond$ for the structure morphism.
\end{defn*}

\subsection{}
Let us compare the formal and analytic moduli of local $G$-shtukas. Recall $\fLocSht^{(I_1,\dotsc,I_k)}_{G,\mu_\bullet}|_{\prod_{i\in I}\bD_i}$ from Definition \ref{ss:formalmoduli}.
\begin{prop*}
  Our $(\fLocSht^{(I_1,\dotsc,I_k)}_{G,\mu_\bullet}|_{\prod_{i\in I}\bD_i})^\Diamond$ is the analytific sheafification of the presheaf over $\textstyle\prod_{i\in I}\bD_i^\Diamond$ whose $S$-points parametrize data consisting of
  \begin{enumerate}[i)]
  \item an algebraic local $G$-shtuka $\sG$ over $S$ bounded by $\mu_\bullet$,
  \item a quasi-isogeny $\de$ from $\sG$ to the trivial algebraic local $G$-shtuka $G$.
  \end{enumerate}
  In particular, we have a canonical morphism of v-sheaves over $\textstyle\prod_{i\in I}\bD_i^\Diamond$
  \begin{align*}
    \ul{\an}:(\fLocSht^{(I_1,\dotsc,I_k)}_{G,\mu_\bullet}|_{\prod_{i\in I}\bD_i})^\Diamond\ra\cLocSht^{(I_1,\dotsc,I_k)}_{G,\mu_\bullet}|_{\prod_{i\in I}\bD_i^\Diamond}
  \end{align*}
given by pulling back $(\sG,\de)$ to $\cY_{S,[0,\infty)}$.
\end{prop*}
\begin{proof}
Theorem \ref{ss:formalmodulirepresentable} shows that $\fLocSht^{(I_1,\dotsc,I_k)}_{G,\mu_\bullet}|_{\prod_{i\in I}\bD_i}$ is a locally noetherian formal scheme, so as an adic space it is the analytic sheafification of the presheaf
\begin{align*}
\Spa(A,A^+)\mapsto\Hom(\Spa(A^+,A^+), \fLocSht^{(I_1,\dotsc,I_k)}_{G,\mu_\bullet}|_{\prod_{i\in I}\bD_i}).
\end{align*}
Because $R^+$ is adic with ideal of definition generated by $\vpi$, we have
\begin{align*}
  &\Hom(S, \fLocSht^{(I_1,\dotsc,I_k)}_{G,\mu_\bullet}|_{\prod_{i\in I}\bD_i}) \\
  = \,&\Hom(\Spf R^+, \fLocSht^{(I_1,\dotsc,I_k)}_{G,\mu_\bullet}|_{\prod_{i\in I}\bD_i})\\
  = \,&\textstyle\varprojlim_n\Hom(\Spec R^+/\vpi^n, \fLocSht^{(I_1,\dotsc,I_k)}_{G,\mu_\bullet}|_{\prod_{i\in I}\bD_i}).
\end{align*}
From here, Lemma \ref{ss:nilpotentalgebraization} and Lemma \ref{ss:quasiisogenylimit} yield the first claim. The second claim follows from the fact that $\cLocSht^{(I_1,\dotsc,I_k)}_{G,\mu_\bullet}|_{\prod_{i\in I}\bD_i^\Diamond}$ is already a sheaf in the analytic topology, so pulling back $(\sG,\de)$ induces a morphism $\ul{\an}$ as desired.
\end{proof}

\begin{thm}\label{thm:analytificationisomorphism}
Our $\ul{\an}$ is an isomorphism. Consequently, $\cLocSht^{(I_1,\dotsc,I_k)}_{G,\mu_\bullet}|_{\prod_{i\in I}\Spd F_i}$ is a locally spatial diamond.
\end{thm}
\begin{proof}
First, we prove that $\ul{\an}$ is an isomorphism. Because products of points as in \cite[Definition 1.2]{Gle20} form a basis for the v-topology \cite[Example 1.1]{Gle20}\footnote{However, in \cite[Example 1.1]{Gle20} one must replace the $k(x)$ with its completed algebraic closure $C(x)$ and $k(x)^+$ with its integral closure in $C(x)$.} and both $(\fLocSht^{(I_1,\dotsc,I_k)}_{G,\mu_\bullet}|_{\prod_{i\in I}\bD_i})^\Diamond$ and $\cLocSht^{(I_1,\dotsc,I_k)}_{G,\mu_\bullet}|_{\prod_{i\in I}\bD_i^\Diamond}$ are v-sheaves, it suffices to check this on $S$-points when $S$ is a product of points. Products of points are totally disconnected \cite[Proposition 1.6]{Gle20}, so we do not need to analytically sheafify when evaluating $\fLocSht^{(I_1,\dotsc,I_k)}_{G,\mu_\bullet}|_{\prod_{i\in I}\bD_i}$ on them.

So assume that $S$ is a product of points, and let $(\sG,\de)$ be an $S$-point of
\begin{align*}
  \cLocSht^{(I_1,\dotsc,I_k)}_{G,\mu_\bullet}|_{\prod_{i\in I}\bD_i^\Diamond}.
\end{align*}
For large enough rational $a$ and all $1\leq j\leq k$, we can use $\de_j|_{\cY_{S,[a,a]}}$ to glue $\sG_j|_{\cY_{S,[0,a]}}$ and $G|_{\cY_{S,[a,\infty]}}$ into a $G$-bundle $\ov\sG_j$ on $\cY_S$. The commutativity of the square in Definition \ref{ss:analyticlocalshtukas}.c) imply that $\phi_j$ and $\id$ glue into an isomorphism of $G$-bundles
  \begin{align*}
    \ov\phi_j:\ov\sG_j|_{\cY_S\ssm\sum_{i\in I_j}\Ga_i}\ra^\sim\ov\sG_{j+1}|_{\cY_S\ssm\sum_{i\in I_j}\Ga_i},
  \end{align*}
  where $\ov\sG_{k+1}$ denotes the $G$-bundle $\prescript\tau{}{\ov\sG_1}$. Then Theorem \ref{ss:algebraizebundles} indicates that $\ov\sG_j$ and $\ov\phi_j$ are uniquely pulled back from a $G$-bundle $\sG_j^{\alg}$ on $\Spec R^+\lb{z}$ and an isomorphism of $G$-bundles $\phi_j^{\alg}:\sG_j^{\alg}|_{\Spec R^+\lb{z}[\frac1{z-\ze_i}]_{i\in I_j}}\ra^\sim\sG^{\alg}_{j+1}|_{\Spec R^+\lb{z}[\frac1{z-\ze_i}]_{i\in I_j}}$, where $\sG^{\alg}_{k+1}$ denotes the $G$-bundle $\prescript\tau{}\sG_1^{\alg}$.

  Altogether $\sG^{\alg}\coloneqq((\sG^{\alg}_j)_{j=1}^k,(\phi^{\alg}_j)_{j=1}^k)$ is an algebraic local $G$-shtuka over $S$. Since $\sG$ is bounded by $\mu_\bullet$, Lemma \ref{ss:BDgrassmanniancomparison} shows that $\sG^{\alg}$ is too. Finally, take $a$ for which $r\coloneqq 1/a$ lies in $\bZ[\textstyle\frac1p]$. Applying Lemma \ref{ss:Batinfinity}, Proposition \ref{ss:Ysousperfectoid}, and \cite[Theorem 2.7.7]{KL15} to the canonical isomorphism $\sG_j|_{\cY_{S,[a,\infty]}}\ra^\sim G$ yields an isomorphism of $G$-bundles $\de_j^{\alg}:\sG^{\alg}_j|_{\Spec R^+\ba{z,\frac{\vpi^r}z}[\frac1z]}\ra^\sim G$, and we see that $\de^{\alg}\coloneqq(\de^{\alg}_j)_{j=1}^k$ is a quasi-isogeny from $\sG^{\alg}$ to $G$. The uniqueness of Theorem \ref{ss:algebraizebundles} and \cite[Theorem 2.7.7]{KL15} imply that $(\sG,\de)$ is uniquely the image of $(\sG^{\alg},\de^{\alg})$ under $\ul{\an}$. Hence $\ul{\an}$ is bijective on $S$-points, as desired. Finally, the last statement follows from $\fLocSht^{(I_1,\dotsc,I_k)}_{G,\mu_\bullet}|_{\prod_{i\in I}\Spa F_i}$ being an analytic adic space and \cite[Lemma 15.6]{Sch17}.
\end{proof}

\subsection{}\label{ss:localshtukalevelstructuredefinition}
Next, we turn to level structures. Let $n$ be a non-negative integer.
\begin{defn*}
  Suppose that $S$ lies over $(\Spa F)^I$, and let $\sG$ be an analytic local $G$-shtuka over $S$. A \emph{level-$n$ structure} on $\sG$ consists of, for all $1\leq j\leq k$, an isomorphism of $G$-bundles
  \begin{align*}
    \psi_j:\sG_j|_{\Spec R\lb{z}/z^n}\ra^\sim G
  \end{align*}
  such that the diagram
  \begin{align*}
    \xymatrixcolsep{0.75in}
    \xymatrix{\sG_j|_{\Spec R\lb{z}/z^n}\ar[r]^-{(\phi_j)_{R\lb{z}/z^n}}\ar[d]^-{\psi_j} & \sG_{j+1}|_{\Spec R\lb{z}/z^n}\ar[d]^-{\psi_{j+1}}\\
    G\ar@{=}[r]&  G
    }
  \end{align*}
commutes, where $\sG_{k+1}$ denotes $\prescript\tau{}{\sG_1}$, and $\psi_{k+1}$ denotes $\prescript\tau{}{\psi_1}$.
\end{defn*}
Since $S$ lies over $(\Spa F)^I$, the $(\phi_j)_{R\lb{z}/z^n}$ are isomorphisms. Therefore $\psi_1$ uniquely determines $\psi_j$ for $2\leq j\leq k$.

\subsection{}
We now define the covering tower of the generic fiber of $\cLocSht^{(I_1,\dotsc,I_k)}_{G,\mu_\bullet}|_{\prod_{i\in I}\bD^\Diamond_i}$.
\begin{defn*}
  Write $\cLocSht^{(I_1,\dotsc,I_k)}_{G,\mu_\bullet,nv}|_{\prod_{i\in I}\Spd F_i}$ for the small v-sheaf over $\textstyle\prod_{i\in I}\Spd F_i$ whose $S$-points parametrize data consisting of
  \begin{enumerate}[i)]
  \item an analytic local $G$-shtuka $\sG$ over $S$ bounded by $\mu_\bullet$,
  \item a quasi-isogeny $\de$ from $\sG$ to the trivial analytic local $G$-shtuka,
  \item a level-$n$ structure $\psi=(\psi_j)_{j=1}^k$ on $\sG$.
  \end{enumerate}
Write $f^\cL:\cLocSht^{(I_1,\dotsc,I_k)}_{G,\mu_\bullet,nv}|_{\prod_{i\in I}\Spd F_i}\ra\textstyle\prod_{i\in I}\Spd F_i$ for the structure morphism.
\end{defn*}

\subsection{}\label{ss:localshtukalevelstructure}
For $n'\geq n$, we have morphisms
\begin{align*}
  \cLocSht^{(I_1,\dotsc,I_k)}_{G,\mu_\bullet,n'v}|_{\prod_{i\in I}\Spd F_i}\ra\cLocSht^{(I_1,\dotsc,I_k)}_{G,\mu_\bullet,nv}|_{\prod_{i\in I}\Spd F_i}
\end{align*}
given by pulling back $\psi_j$ to $\Spec R\lb{z}/z^n$ for all $1\leq j\leq k$. Write $K_{n',n}$ for the kernel of $G(\cO_F/z^{n'})\ra G(\cO_F/z^n)$, and note that $K_{n',n}$ acts on $\cLocSht^{(I_1,\dotsc,I_k)}_{G,\mu_\bullet,n'v}|_{\prod_{i\in I}\Spd F_i}$ over $\cLocSht^{(I_1,\dotsc,I_k)}_{G,\mu_\bullet,nv}|_{\prod_{i\in I}\Spd F_i}$ via postcomposition with $\psi_j$ for all $1\leq j\leq k$.
\begin{prop*}
  The morphism
  \begin{align*}
    \cLocSht^{(I_1,\dotsc,I_k)}_{G,\mu_\bullet,n'v}|_{\prod_{i\in I}\Spd F_i}\ra\cLocSht^{(I_1,\dotsc,I_k)}_{G,\mu_\bullet,nv}|_{\prod_{i\in I}\Spd F_i}
  \end{align*}
  is finite Galois, where the Galois action is given by that of $K_{n',n}$. Consequently, $\cLocSht^{(I_1,\dotsc,I_k)}_{G,\mu_\bullet,nv}|_{\prod_{i\in I}\Spd F_i}$ is a locally spatial diamond.
\end{prop*}
\begin{proof}
  First, take $n=0$, so that
  \begin{align*}
   \cLocSht^{(I_1,\dotsc,I_k)}_{G,\mu_\bullet,nv}|_{\prod_{i\in I}\Spd F_i}=\cLocSht^{(I_1,\dotsc,I_k)}_{G,\mu_\bullet}|_{\prod_{i\in I}\Spd F_i}.
  \end{align*}
  For any $S$-point $(\sG,\de)$ of $\cLocSht^{(I_1,\dotsc,I_k)}_{G,\mu_\bullet}|_{\prod_{i\in I}\Spd F_i}$, form the Cartesian square
  \begin{align*}
    \xymatrix{S'\ar[r]\ar[d] & \cLocSht^{(I_1,\dotsc,I_k)}_{G,\mu_\bullet,n'v}|_{\prod_{i\in I}\Spd F_i}\ar[d] \\
    S\ar[r] & \cLocSht^{(I_1,\dotsc,I_k)}_{G,\mu_\bullet}|_{\prod_{i\in I}\Spd F_i}.}
  \end{align*}
  Then $S'$ parametrizes level-$n'$ structures $\psi$ on $\sG$. Because $\psi_1$ uniquely determines $\psi_j$ for $2\leq j\leq k$, we see that level-$n'$ structures on $\sG$ are equivalent to trivializations of the $\tau$-$G$-bundle $(\sG_1|_{\Spec R\lb{z}/z^{n'}},(\phi_k\circ\dotsb\circ\phi_1)_{R\lb{z}/z^{n'}})$ over $\Spec R\lb{z}/z^{n'}$. Thus Proposition \ref{ss:OFlocalsystemstaumodules} and \cite[Proposition 9.7]{Sch17} imply that $S'\ra S$ is finite Galois with the desired Galois action.

  For general $n$, the result follows from the commutative triangle
  \begin{align*}
    \xymatrixcolsep{-.5in}
    \xymatrix{\cLocSht^{(I_1,\dotsc,I_k)}_{G,\mu_\bullet,n'v}|_{\prod_{i\in I}\Spd F_i}\ar[rr]\ar[rd] & & \cLocSht^{(I_1,\dotsc,I_k)}_{G,\mu_\bullet,nv}|_{\prod_{i\in I}\Spd F_i}\ar[ld] \\
    & \cLocSht^{(I_1,\dotsc,I_k)}_{G,\mu_\bullet}|_{\prod_{i\in I}\Spd F_i} &}
  \end{align*}
and compatibility of the $K_{n',n}$-action with changing $n'$ and $n$. Finally, the last statement follows from Theorem \ref{thm:analytificationisomorphism} and \cite[Lemma 11.21]{Sch17}.
\end{proof}

\subsection{}\label{ss:localheckecorrespondences}
The covering tower enjoys the following Hecke correspondences. Write
\begin{align*}
 \cLocSht^{(I_1,\dotsc,I_k)}_{G,\mu_\bullet,\infty v}\coloneqq\textstyle\varprojlim_n\cLocSht^{(I_1,\dotsc,I_k)}_{G,\mu_\bullet,nv}|_{\prod_{i\in I}\Spd F_v},
\end{align*}
and write $K_n$ for the kernel of $G(\cO_F)\ra G(\cO_F/z^n)$.
\begin{prop*}
  We have a canonical $G(F)$-action on $\cLocSht^{(I_1,\dotsc,I_k)}_{G,\mu_\bullet,\infty v}$ over $\textstyle\prod_{i\in I}\Spd F_i$ that extends the $G(\cO_F)$-action from \ref{ss:localshtukalevelstructure}. Consequently, for any $g$ in $G(F)$, we have a canonical finite \'etale correspondence $\mathbf{1}_{K_ngK_n}$ from $\cLocSht^{(I_1,\dotsc,I_k)}_{G,\mu_\bullet,nv}|_{\prod_{i\in I}\Spd F_i}$ to itself.
\end{prop*}
\begin{proof}
  Let $(\sG,\de)$ be an $S$-point of $\cLocSht^{(I_1,\dotsc,I_k)}_{G,\mu_\bullet}|_{\prod_{i\in I}\Spd F_i}$, and let $(\psi^n)_{n\geq0}$ be a compatible system of level-$n$ structures $\psi^n$ on $\sG$. For all $1\leq j\leq k$, we see that $\psi_j\coloneqq\textstyle\varprojlim_n\psi^n_j$ yields an isomorphism of $G$-bundles $\sG_j|_{\Spec R\lb{z}}\ra^\sim G$. For any $g$ in $G(F)$, we get an isomorphism of $G$-bundles $g\circ(\psi_j)_{R\lp{z}}:\sG_j|_{\Spec R\lp{z}}\ra^\sim G$, which we use with Beauville--Laszlo to glue $G|_{\Spec R\lb{z}}$ and $\sG_j|_{\cY_{S,(0,\infty)}}$ into a $G$-bundle $g\cdot\sG_j$ on $\cY_{S,[0,\infty)}$.

  Since $(g\cdot\sG_j)|_{\cY_{S,(0,\infty)}\ssm\sum_{i\in I_j}\Ga_i}$ is canonically isomorphic to $\sG_j|_{\cY_{S,(0,\infty)}\ssm\sum_{i\in I_j}\Ga_i}$, the commutativity of the square in Definition \ref{ss:localshtukalevelstructuredefinition} and Beauville--Laszlo let us glue $\id$ and $(\phi_j)_{\cY_{S,(0,\infty)}\ssm\sum_{i\in I_j}\Ga_i}$ into an isomorphism of $G$-bundles
  \begin{align*}
    g\cdot\phi_j:(g\cdot\sG_j)|_{\cY_{S,[0,\infty)}\ssm\sum_{i\in I_j}\Ga_i}\ra^\sim(g\cdot\sG_{j+1})|_{\cY_{S,[0,\infty)}\ssm\sum_{i\in I_j}\Ga_i},
  \end{align*}
  where $g\cdot\sG_{k+1}$ denotes $\prescript\tau{}{(g\cdot\sG_1)}$. As $\sG$ is bounded by $\mu_\bullet$, the analytic local $G$-shtuka $g\cdot\sG\coloneqq((g\cdot\sG_j)_{j=1}^k,(g\cdot\phi_j)_{j=1}^k)$ is too. Because $(g\cdot\sG_j)|_{\cY_{S,[a,\infty)}}$ is canonically isomorphic to $\sG_j|_{\cY_{S,[a,\infty)}}$, our $\de$ induces a quasi-isogeny from $g\cdot\sG$ to $G$. Since $(g\cdot\sG_j)|_{\Spec R\lb{z}}$ is canonically trivial, we have the trivial level-$n$ structure $\id=(\id)_{j=1}^k$ on $g\cdot\sG$.

  Altogether, we define the image of $(\sG,\de,(\psi^n)_{n\geq0})$ under $g$ to be $(g\cdot\sG,\de,(\id)_{n\geq0})$. When $g$ lies in $G(\cO_F)$, our $g\circ(\psi_j)_{R\lp{z}}$ above extends to an isomorphism of $G$-bundles $g\circ\psi_j:\sG_j|_{\Spec R\lb{z}}\ra^\sim G$, and tracing through our identifications shows that this indeed recovers the action from \ref{ss:localshtukalevelstructure}. Finally, $\mathbf{1}_{K_ngK_n}$ is given by
  \begin{align*}
    \xymatrix{\cLocSht^{(I_1,\dotsc,I_k)}_{G,\mu_\bullet,\infty v}/(K_n\cap g^{-1}K_ng)\ar[r]^-g\ar[d] & \cLocSht^{(I_1,\dotsc,I_k)}_{G,\mu_\bullet,\infty v}/(gK_ng^{-1}\cap K_n)\ar[d] \\
    \cLocSht^{(I_1,\dotsc,I_k)}_{G,\mu_\bullet,\infty v}/K_n & \cLocSht^{(I_1,\dotsc,I_k)}_{G,\mu_\bullet,\infty v}/K_n
    }
  \end{align*}
  and identifying $\cLocSht^{(I_1,\dotsc,I_k)}_{G,\mu_\bullet,\infty v}/K_n$ with $\cLocSht^{(I_1,\dotsc,I_k)}_{G,\mu_\bullet,nv}|_{\prod_{i\in I}\Spd F_i}$.
\end{proof}

\subsection{}\label{ss:FSshtukas}
Recall the following variant of the moduli of local shtukas, which is defined purely in terms of the Fargues--Fontaine curve. Let $K$ be a compact open subgroup of $G(F)$.
\begin{defn*}
  Write $\cM^I_{G,\mu_\bullet,K}|_{\prod_{i\in I}\Spd F_i}$ for the small v-sheaf over $\textstyle\prod_{i\in I}\Spd F_i$ whose $S$-points parametrize data consisting of
  \begin{enumerate}[i)]
  \item a $G$-bundle $\sE$ on $X_S$ such that, for all geometric points $\ov{s}$ of $S$, its pullback $\sE_{\ov{s}}$ to $X_{\ov{s}}$ is trivial,
  \item an isomorphism of $G$-bundles
    \begin{align*}
\al:\sE|_{X_S\ssm\sum_{i\in I}\Ga_i}\ra^\sim G
    \end{align*}
    that is meromorphic along $\textstyle\sum_{i\in I}\Ga_i$ such that, for all geometric points $\ov{s}$ of $S$, the relative position of $\al_{\ov{s}}$ at $\Ga_{i,\ov{s}}$ is bounded by $\sum_{i'}\mu_{i'}$, where $i'$ runs over elements of $I$ satisfying $\Ga_{i',\ov{s}}=\Ga_{i,\ov{s}}$,
  \item a $\ul{K}$-bundle $\bP$ on $S$ whose pushforward along $\ul{K}\ra\ul{G(F)}$ equals the $\ul{G(F)}$-bundle on $S$ corresponding to $\sE$ via \cite[Theorem III.2.4]{FS21}.
  \end{enumerate}
Write $f^\cM:\cM^I_{G,\mu_\bullet,K}|_{\prod_{i\in I}\Spd F_i}\ra\textstyle\prod_{i\in I}\Spd F_i$ for the structure morphism.
\end{defn*}
Recall that $\cM^I_{G,\mu_\bullet,K}|_{\prod_{i\in I}\Spd F_i}$ is a locally spatial diamond.

\subsection{}\label{ss:FSshtukascompare}
The analytic moduli of local $G$-shtukas is related to the Fargues--Fontaine variant as follows.
\begin{prop*}
  We have a canonical morphism
  \begin{align*}
    c:\cLocSht^{(I)}_{G,\mu_\bullet,nv}|_{\prod_{i\in I}\Spd F_i}\ra\cM^I_{G,\mu_\bullet,K_n}|_{\prod_{i\in I}\Spd F_i}
  \end{align*} of locally spatial diamonds over $\textstyle\prod_{i\in I}\Spd F_i$. 
\end{prop*}
\begin{proof}
  Let $(\sG,\de,\psi)$ be an $S$-point of $\cLocSht^{(I)}_{G,\mu_\bullet,nv}$. Theorem \ref{ss:Robbaringpowerseries} and Proposition \ref{ss:OFlocalsystemstaumodules} show that $(\sG_1|_{\Spec\wt\cR^{\mathrm{int}}_R},(\phi_1)_{\wt\cR^{\mathrm{int}}_R})$ corresponds to a $\ul{G(\cO_F)}$-bundle on $S$, and Proposition \ref{ss:OFlocalsystemstaumodules} implies that $\psi_1$ corresponds to a reduction $\bP$ of this $\ul{G(\cO_F)}$-bundle to a $\ul{K_n}$-bundle. Via continuation by Frobenius, $(\sG_1|_{\Spec\wt\cR^{\mathrm{int}}_R},(\phi_1)_{\wt\cR^{\mathrm{int}}_R})$ also induces a $\tau$-$G$-bundle $(\sF,\ups)$ over $\cY_{S,[0,\infty)}$ such that $(\sF|_{\cY_{S,(0,\infty)}},(\ups)_{\cY_{S,(0,\infty)}})$ corresponds to the pushforward of $\bP$ along $\ul{K_n}\ra\ul{G(F)}$. Therefore the pullback of the $G$-bundle $\sE\coloneqq(\sF|_{\cY_{S,(0,\infty)}})/(\ups)_{\cY_{S,(0,\infty)}}^\bZ$ from $X_S$ to $X_{\ov{s}}$ is trivial for all geometric points $\ov{s}$ of $S$, and the corresponding $\ul{G(F)}$-bundle on $S$ via \cite[Theorem III.2.4]{FS21} equals the pushforward of $\bP$ along $\ul{K_n}\ra\ul{G(F)}$. Finally, continuation by Frobenius and Lemma \ref{ss:BDgrassmanniancomparison} indicate that $\de_1$ induces an isomorphism of $G$-bundles $\al:\sE|_{X_S\ssm\sum_{i\in I}\Ga_i}\ra^\sim G$ with the desired relative position bound, so altogether $(\sE,\al,\bP)$ yields an $S$-point of $\cM^I_{G,\mu_\bullet,K_n}$.
\end{proof}

\subsection{}\label{ss:FSnaivecompare}
We will need the following results of Fargues--Scholze \cite{FS21} on the intersection homology of the moduli of local shtukas. Recall the notation of \ref{ss:Lgroup}, and let $V$ be an object of $\Rep_E(\prescript{L}{}{G})^I$. Note that
\begin{align*}
  \textstyle\coprod_{\mu_\bullet}\cLocSht^{(I_1,\dotsc,I_k)}_{G,\mu_\bullet,nv}|_{(\Spd\wt{F})^I}\mbox{ and }\coprod_{\mu_\bullet}\cM^I_{G,\mu_\bullet,K}|_{(\Spd\wt{F})^I}
\end{align*}
naturally descend to small v-sheaves $\cLocSht^{(I_1,\dotsc,I_k)}_{G,V,nv}$ and $\cM^I_{G,V,K}$ over $(\Spd F)^I$, respectively, where $\mu_\bullet$ runs over highest weights appearing in $V_{\ov\bQ_\ell}|_{\wh{G}^I}$. Proposition \ref{ss:localshtukalevelstructure} and \cite[Proposition 13.4 (iv)]{Sch17} imply that $\cLocSht^{(I_1,\dotsc,I_k)}_{G,V,nv}$ is a locally spatial diamond, and we see that $\cM^I_{G,V,K}$ is also a locally spatial diamond.

Let $\La$ be $\cO_E$ or $E$, and now let $V$ be an object of $\Rep_{\cO_E}(\prescript{L}{}G)^I$. If $\La=\cO_E$, then by abuse of notation write $V$ for $V_E$. Write $(\Spd\breve{F})^I$ for the $I$-th power of $\Spd\breve{F}$ over $\ov\bF_q$, and write $\prescript\prime{}\cF_{V,K,\La}^I$ for the object of $D_{\solid}(\cM^I_{G,V,K}|_{(\Spd\breve{F})^I},\La)$ obtained from \cite[Theorem VI.11.1]{FS21} and $V$ by first applying the double-dual embedding as in \cite[p.~264]{FS21} and then pulling back to $\cM^I_{G,V,K}|_{(\Spd\breve{F})^I}$. Write $\prescript\prime{}\cF_{V,nv,\La}^{(I_1,\dotsc,I_k)}$ for the pullback of $\prescript\prime{}\cF_{V,K_n,\La}^I$ under the composition
\begin{align*}
\cLocSht^{(I_1,\dotsc,I_k)}_{G,V,nv}|_{(\Spd\breve{F})^I}\ra\cLocSht^{(I)}_{G,V,nv}|_{(\Spd\breve{F})^I}\ra^c\cM^I_{G,V,K_n}|_{(\Spd\breve{F})^I}.
\end{align*}

Write $W_F$ for the absolute Weil group of $F$.
\begin{thm*}
Our $c$ induces an isomorphism $f_\natural^\cL(\prescript\prime{}\cF_{V,nv,\La}^{(I)})\ra^\sim f_\natural^\cM(\prescript\prime{}\cF^I_{V,K,\La})$. Consequently, the object $f_\natural^\cL(\prescript\prime{}\cF_{V,nv,\La}^{(I)})$ of $D_{\solid}((\Spd\breve{F})^I,\La)$ naturally arises via pullback from $D(W_F^I,\La)$.
\end{thm*}
\begin{proof}
Using Theorem \ref{ss:Robbaringpowerseries} and Proposition \ref{ss:OFlocalsystemstaumodules}, the argument in the proof of \cite[Proposition IX.3.2]{FS21} yields the first claim. For the second claim, \cite[Proposition VII.3.1 (iii)]{FS21} enables us to identify $f_\natural^\cM(\prescript\prime{}\cF^I_{V,K,\La})$ with $i_1^*T_V(i_{1!}(\cInd_{K_n}^{G(F)}\La))$ as objects of $D(\La)$, where $i_1:[*/\ul{G(F)}]\ra\Bun_G$ is the canonical open embedding, and $T_V$ is the geometric Hecke operator associated with $V$. Therefore \cite[Corollary IX.2.3]{FS21} yields the desired result.
\end{proof}

\subsection{}\label{ss:FSpartialfrobeniuscomparison}
Finally, we define partial Frobenii for the analytic moduli of local $G$-shtukas and relate them to partial Frobenii on the Fargues--Fontaine variant as follows. Write $\cFr^{(I_1,\dotsc,I_k)}:\cLocSht^{(I_1,\dotsc,I_k)}_{G,V,nv}\ra\cLocSht^{(I_2,\dotsc,I_k,I_1)}_{G,V,nv}$ for the morphism that sends
  \begin{align*}
    \xymatrix{\sG_1\ar@{-->}[r]^-{\phi_1}\ar@{-->}[d]^-{\de_1} & \dotsm\ar@{-->}[r]^-{\phi_{k-1}} & \sG_k\ar@{-->}[r]^-{\phi_k}\ar@{-->}[d]^-{\de_k} & \prescript\tau{}{\sG_1}\ar@{-->}[d]^-{\prescript\tau{}{\de_1}\quad\mbox{to}} & \sG_2\ar@{-->}[r]^-{\phi_2}\ar@{-->}[d]^-{\de_2} & \dotsm\ar@{-->}[r]^-{\phi_k} & \prescript\tau{}{\sG_1}\ar@{-->}[r]^-{\prescript\tau{}{\phi_1}}\ar@{-->}[d]^-{\prescript\tau{}{\de_1}} & \prescript\tau{}{\sG_2}\ar@{-->}[d]^-{\prescript\tau{}{\de_2}} \\
    G\ar@{=}[r] & \dotsm \ar@{=}[r] & G\ar@{=}[r] & G &     G\ar@{=}[r] & \dotsm \ar@{=}[r] & G\ar@{=}[r] & G.}
  \end{align*}
Note that $\cM^I_{G,V,K}$ naturally descends to a v-sheaf over $(\Div_F^1)^I$, where $\Div^1_F$ denotes the small v-sheaf over $\Spd\bF_q$ whose $S$-points parametrize degree-$1$ relative effective Cartier divisors of $X_S$. Write $\vp_{I_1}:\cM^I_{G,V,K}\ra\cM^I_{G,V,K}$ for the resulting endomorphism given by geometric $q$-Frobenius on the $i$-th factor of $(\Spd F)^I$ for $i$ in $I_1$ and the identity on all other factors.
\begin{lem*}
  We have a commutative diagram
  \begin{align*}
    \xymatrix{\cLocSht^{(I_1,\dotsc,I_k)}_{G,V,nv}\ar[r]\ar[d]^-{\cFr^{(I_1,\dotsc,I_k)}} & \cLocSht^{(I)}_{G,V,nv} \ar[r]^-c & \cM^I_{G,V,K_n}\ar[d]^-{\vp_{I_1}} \\
    \cLocSht^{(I_2,\dotsc,I_k,I_1)}_{G,V,nv}\ar[r] & \cLocSht^{(I)}_{G,V,nv} \ar[r]^-c & \cM^I_{G,V,K_n}. }
  \end{align*}
\end{lem*}
\begin{proof}
This follows immediately from the proof of Proposition \ref{ss:FSshtukascompare}.
\end{proof}

\section{Uniformizing the moduli spaces of global shtukas}\label{s:uniformizing}
At this point, we shift focus from local to global considerations. Our goal in this section is to define the uniformization morphism, which is essential for our main results. First, we recall some facts about global shtukas and their moduli spaces. We then take formal completions at a fixed place and define the uniformization morphism on the level of formal stacks. By restricting to a Harder--Narasimhan truncation on the global moduli and using results from \S\ref{s:formallocalshtukas} on the local moduli, we can pass from formal stacks to formal schemes that are locally formally of finite type over $\bD^I$. This lets us avoid questions about analytifying stacks, as well as upgrade the formal \'etaleness of our uniformization morphism to \'etaleness (after passing to generic fibers). Finally, we extend the uniformization morphism to the covering tower on generic fibers.

\subsection{}
We start by switching our notation to a global context. Let $C$ be a geometrically connected smooth proper curve over a finite field $\bF_q$, and write $F$ for $\bF_q(C)$. Fix a separable closure $\ov{F}$ of $F$, and write $\Ga_F$ for $\Gal(\ov{F}/F)$. Write $\bA$ for the adele ring of $C$, and write $\bO$ for its subring of integral adeles.

Let $G$ be a parahoric group scheme over $C$ as in \cite[Definition 2.18]{Ric16}, and write $Z$ for the center of $G$. By \cite[Proposition 2.2(b)]{AH13}, there exists an $\SL_h$-bundle $\sV$ on $C$ and a closed embedding $\io:G^{\ad}\ra\ul\Aut(\sV)$ of group schemes over $C$ such that $\ul\Aut(\sV)/G^{\ad}$ satisfies \cite[(2.1)]{AH13}.

Let $T$ be a maximal subtorus of $G_F$, and write $X^+_*(T)$ for the set of dominant cocharacters of $T_{\ov{F}}$ with respect to a fixed Borel subgroup $B\subseteq G_{\ov{F}}$ containing $T_{\ov{F}}$. Identify $X_*^+(T)$ with the set of conjugacy classes of cocharacters of $G_{\ov{F}}$. Let $\mu_\bullet=(\mu_i)_{i\in I}$ be in $X_*^+(T)$, and identify the field of definition of $\mu_i$ with $\bF_q(C_i)$ for some finite generically \'etale cover $C_i\ra C$. Write $\Gr^{(I_1,\dotsc,I_k)}_{G,\mu_\bullet}|_{\prod_{i\in I}C_i}$ for the closed affine Schubert variety as in \ref{ss:affineschubertvarieties}.

\subsection{}
Let us recall the definition of global $G$-shtukas. Let $S$ be an affine scheme over $C^I$, and adopt the notation of \ref{ss:globalBDgrassmannian}. Write $\tau:S\ra S$ for the absolute $q$-Frobenius endomorphism, and by abuse of notation, write $\tau:C\times S\ra C\times S$ for the identity times $\tau$.
\begin{defn*}\hfill
  \begin{enumerate}[a)]
  \item A \emph{global $G$-shtuka} over $S$ consists of
      \begin{enumerate}[i)]
  \item for all $1\leq j\leq k$, a $G$-bundle $\sG_j$ on $C\times S$,
  \item for all $1\leq j\leq k$, an isomorphism of $G$-bundles
    \begin{align*}
      \phi_j:\sG_j|_{C\times S\ssm\sum_{i\in I_j}\Ga_i}\ra^\sim\sG_{j+1}|_{C\times S\ssm\sum_{i\in I_j}\Ga_i},
    \end{align*}
    where $\sG_{k+1}$ denotes the $G$-bundle $\prescript\tau{}{\sG_1}$.
  \end{enumerate}
\item Suppose that $S$ lies over $\textstyle\prod_{i\in I}C_i$, and let $\sG=((\sG_j)_{j=1}^k,(\phi_j)_{j=1}^k)$ be a global $G$-shtuka over $S$. We say that $\sG$ is \emph{bounded by $\mu_\bullet$} if the $S$-point of
  \begin{align*}
    [L_I^+(G)\bs\!\Gr^{(I_1,\dotsc,I_k)}_G|_{\prod_{i\in I}C_i}]
  \end{align*}
  given by $((\sG_j|_{\Spec\wh\cO_C(S)})_{j=1}^k,((\phi_j)_{\wh\cO^{j,\circ}_C(S)})_{j=1}^k)$ lies in $[L_I^+(G)\bs\!\Gr^{(I_1,\dotsc,I_k)}_{G,\mu_\bullet}|_{\prod_{i\in I}C_i}]$.

\item Let $\sG$ be a global $G$-shtuka over $S$. We say that $\sG$ has \emph{Harder--Narasimhan polygon bounded by $s$} if the $\SL_r$-bundle $\io_*(\sG_1^{\ad})$ has Harder--Narasimhan polygon bounded by $s2\rho^\vee$, where $2\rho^\vee$ denotes the sum of positive coroots in $\SL_h$.
\end{enumerate}
\end{defn*}

\subsection{}
Next, we turn to level structures. Let $N$ be a finite closed subscheme of $C$.
\begin{defn*}
  Suppose that $S$ lies over $(C\ssm N)^I$, and let $\sG$ be a global $G$-shtuka over $S$. A \emph{level-$N$ structure} on $\sG$ consists of, for all $1\leq j\leq k$, an isomorphism of $G$-bundles
    \begin{align*}
      \psi_j:\sG_j|_{N\times S}\ra^\sim G
    \end{align*}
      such that the diagram
    \begin{align*}
      \xymatrix{\sG_j|_{N\times S}\ar[r]^-{(\phi_j)_N}\ar[d]^-{\psi_j} & \sG_{j+1}|_{N\times S}\ar[d]^-{\psi_{j+1}} \\
G\ar@{=}[r] & G}
    \end{align*}
    commutes, where $\sG_{k+1}$ denotes $\prescript\tau{}{\sG_1}$, and $\psi_{k+1}$ denotes $\prescript\tau{}{\psi_1}$.
  \end{defn*}
  Since $S$ lies over $(C\ssm N)^I$, the $(\phi_j)_N$ are isomorphisms. Therefore $\psi_1$ uniquely determines $\psi_j$ for $2\leq j\leq k$.

  \subsection{}\label{ss:globalshtukamoduli}
We now recall the moduli of global $G$-shtukas and its associated structures. Write $N_i$ for the preimage of $N$ in $C_i$.
\begin{defn*}
  Write $\Sht^{(I_1,\dotsc,I_k)}_{G,\mu_\bullet,N}|_{\prod_{i\in I}C_i\ssm N_i}$ for the stack over $\textstyle\prod_{i\in I}C_i\ssm N_i$ whose $S$-points parametrize data consisting of
  \begin{enumerate}[i)]
  \item a global $G$-shtuka $\sG$ over $S$ bounded by $\mu_\bullet$,
  \item a level-$N$ structure $\psi=(\psi_j)_{j=1}^k$ on $\sG$.
  \end{enumerate}
  Write $\Sht^{(I_1,\dotsc,I_k),\leq s}_{G,\mu_\bullet,N}|_{\prod_{i\in I}C_i\ssm N_i}$ for the open substack of $\Sht^{(I_1,\dotsc,I_k)}_{G,\mu_\bullet,N}|_{\prod_{i\in I}C_i\ssm N_i}$ whose $S$-points consist of the $(\sG,\psi)$ such that $\sG$ has Harder--Narasimhan polygon bounded by $s$.

  Write $f^{\mathrm{S}}:\Sht^{(I_1,\dotsc,I_k)}_{G,\mu_\bullet,N}|_{\prod_{i\in I}C_i\ssm N_i}\ra\textstyle\prod_{i\in I}C_i\ssm N_i$ for the structure morphism. 
\end{defn*}
Our $\Sht^{(I_1,\dotsc,I_k)}_{G,\mu_\bullet,N}|_{\prod_{i\in I}C_i\ssm N_i}$ has an action of $Z(F)\bs Z(\bA)$ by twisting. Since the image of $Z$ in $\ul{\Aut}(\sV)$ is trivial, $\Sht^{(I_1,\dotsc,I_k),\leq s}_{G,\mu_\bullet,N}|_{\prod_{i\in I}C_i\ssm N_i}$ is preserved by the $Z(F)\bs Z(\bA)$-action. Finally, note that $\Sht^{(I_1,\dotsc,I_k)}_{G,\mu_\bullet,N}|_{\prod_{i\in I}C_i\ssm N_i}$ is the increasing union of the $\Sht^{(I_1,\dotsc,I_k),\leq s}_{G,\mu_\bullet,N}|_{\prod_{i\in I}C_i\ssm N_i}$.

\subsection{}\label{ss:globalshtukalevelstructure}
For finite closed subschemes $N'\supseteq N$ of $C$, we have morphisms
\begin{align*}
  \Sht^{(I_1,\dotsc,I_k)}_{G,\mu_\bullet,N'}|_{\prod_{i\in I}C_i\ssm N_i'}\ra\Sht^{(I_1,\dotsc,I_k)}_{G,\mu_\bullet,N}|_{\prod_{i\in I}C_i\ssm N_i'}
\end{align*}
given by pulling back $\psi_j$ to $N\times S$ for all $1\leq j\leq k$. Write $K_{N',N}$ for the kernel of $G(\sO_{N'})\ra G(\sO_N)$, and note that $K_{N',N}$ acts on $\Sht^{(I_1,\dotsc,I_k)}_{G,\mu_\bullet,N'}|_{\prod_{i\in I}C_i\ssm N_i'}$ over $\Sht^{(I_1,\dotsc,I_k)}_{G,\mu_\bullet,N}|_{\prod_{i\in I}C_i\ssm N_i'}$ via postcomposition with $\psi_j$ for all $1\leq j\leq k$.
\begin{prop*}
The morphism $\Sht^{(I_1,\dotsc,I_k)}_{G,\mu_\bullet,N'}|_{\prod_{i\in I}C_i\ssm N_i'}\ra\Sht^{(I_1,\dotsc,I_k)}_{G,\mu_\bullet,N}|_{\prod_{i\in I}C_i\ssm N_i'}$ is finite Galois, where the Galois action is given by that of $K_{N',N}$.
\end{prop*}
\begin{proof}
  When $N=\varnothing$, the result follows from the proof of \cite[Proposition 2.16 b)]{Var04}. For general $N$, the result follows from the commutative triangle
    \begin{align*}
    \xymatrixcolsep{-.5in}
    \xymatrix{\Sht^{(I_1,\dotsc,I_k)}_{G,\mu_\bullet,N'}|_{\prod_{i\in I}C_i\ssm N'_i}\ar[rr]\ar[rd] & & \Sht^{(I_1,\dotsc,I_k)}_{G,\mu_\bullet,N}|_{\prod_{i\in I}C_i\ssm N'_i}\ar[ld] \\
    & \Sht^{(I_1,\dotsc,I_k)}_{G,\mu_\bullet,\varnothing}|_{\prod_{i\in I}C_i\ssm N'_i} &}
  \end{align*}
and compatibility of the $K_{N',N}$-action with changing $N'$ and $N$.
\end{proof}

\begin{prop}\label{ss:globalshtukarepresentable}
Our $\Sht^{(I_1,\dotsc,I_k)}_{G,\mu_\bullet,N}|_{\prod_{i\in I}C_i\ssm N_i}$ is a Deligne--Mumford stack that is separated and locally of finite type over $\textstyle\prod_{i\in I}C_i\ssm N_i$. Moreover, for large enough $\deg{N}$, our $\Sht^{(I_1,\dotsc,I_k),\leq s}_{G,\mu_\bullet,N}|_{\prod_{i\in I}C_i\ssm N_i}$ is a scheme that is separated and locally of finite type over $\textstyle\prod_{i\in I}C_i\ssm N_i$.
\end{prop}
\begin{proof}
The second claim follows from the proof of \cite[Lemme 12.19]{Laf16}. Using Proposition \ref{ss:globalshtukalevelstructure}, the first claim follows from the argument in \cite[\S5.1.5]{YZ17}.
\end{proof}

\subsection{}
Let $\wt{F}$ be the finite Galois extension of $F$ such that $\Gal(\wt{F}/F)$ equals the image of the $\Ga_F$-action on $X_*^+(T)$, and identify $\wt{F}$ with $\bF_q(\wt{C})$ for some finite generically \'etale cover $\wt{C}\ra C$. Write $\wt{N}$ for the preimage of $N$ in $\wt{C}$. Write $\wh{G}$ for the dual group of $G_F$ over $\cO_E$, and write $\prescript{L}{}G$ for $\wh{G}\rtimes\Gal(\wt{F}/F)$.

Let $V$ be an object of $\Rep_E(\prescript{L}{}G)^I$. Note that $\coprod_{\mu_\bullet}\Sht^{(I_1,\dotsc,I_k)}_{G,\mu_\bullet,N}|_{(\wt{C}\ssm\wt{N})^I}$ and $\coprod_{\mu_\bullet}\Sht^{(I_1,\dotsc,I_k),\leq s}_{G,\mu_\bullet,N}|_{(\wt{C}\ssm\wt{N})^I}$ naturally descend to stacks
\begin{align*}
  \Sht^{(I_1,\dotsc,I_k)}_{G,V,N}\mbox{ and }\Sht^{(I_1,\dotsc,I_k),\leq s}_{G,V,N}
\end{align*}
over $(C\ssm N)^I$, respectively, where $\mu_\bullet$ runs over highest weights appearing in $V_{\ov\bQ_\ell}|_{\wh{G}^I}$. Proposition \ref{ss:globalshtukarepresentable} and descent imply that $\Sht^{(I_1,\dotsc,I_k)}_{G,V,N}$ is a Deligne--Mumford stack that is separated and locally of finite type over $(C\ssm N)^I$, and for large enough $\deg{N}$, our $\Sht^{(I_1,\dotsc,I_k),\leq s}_{G,V,N}$ is a scheme that is separated and locally of finite type over $(C\ssm N)^I$.

\subsection{}\label{ss:heckecorrespondences}
Write $K_N$ for the kernel of $G(\bO)\ra G(\sO_N)$. For any $g$ in $G(\bA)$, recall that we have a canonical finite \'etale correspondence $\mathbf{1}_{K_NgK_N}$ from $\Sht^{(I_1,\dotsc,I_k)}_{G,\mu_\bullet,N}|_{\prod_{i\in I}\bF_q(C_i)}$ to itself \cite[Construction 2.20]{Laf16}\footnote{Although \cite[Construction 2.20]{Laf16} only addresses the case when $G$ is split, it extends to the general case. Indeed, this is already implicitly used in \cite[(12.16)]{Laf16}.}. Note that $\mathbf{1}_{K_NgK_N}$ commutes with the $Z(F)\bs Z(\bA)$-action.

\begin{defn}\label{ss:globalshtukapartialfrobenius}
  Write $\Fr^{(I_1,\dotsc,I_k)}:\Sht^{(I_1,\dotsc,I_k)}_{G,V,N}\ra\Sht^{(I_2,\dotsc,I_k,I_1)}_{G,V,N}$ for the morphism given by
  \begin{align*}
    \xymatrix{(\sG_1\ar@{-->}[r]^-{\phi_1} & \dotsm\ar@{-->}[r]^-{\phi_{k-1}} & \sG_k\ar@{-->}[r]^-{\phi_k} & \prescript\tau{}{\sG_1})\ar@{|->}[r] & (\sG_2\ar@{-->}[r]^-{\phi_2} & \dotsm\ar@{-->}[r]^-{\phi_k} & \prescript\tau{}{\sG_1}\ar@{-->}[r]^-{\prescript\tau{}{\phi_1}} & \prescript\tau{}{\sG_2}).}
  \end{align*}
\end{defn}
Note that $\Fr^{(I_1,\dotsc,I_k)}$ lies above the endomorphism of $(C\ssm N)^I$ given by geometric $q$-Frobenius on the $i$-th factor for $i$ in $I_1$ and the identity on all other factors.

By \cite[Lemme 3.1]{Laf16}\footnote{While \cite[Lemme 3.1]{Laf16} only treats the case of split $G$, it extends to the general case. Indeed, this is already implicitly used in \cite[(12.15)]{Laf16}.}, there exists a non-negative integer $\ka(V)$ such that
\begin{align*}
  (\Fr^{(I_1,\dotsc,I_k)})^{-1}(\Sht^{(I_2,\dotsc,I_k,I_1),\leq s}_{G,V,N})&\subseteq\Sht^{(I_1,\dotsc,I_k),\leq s+\ka(V)}_{G,V,N}\mbox{ and } \\
  \Fr^{(I_1,\dotsc,I_k)}(\Sht^{(I_1,\dotsc,I_k),\leq s}_{G,V,N})&\subseteq\Sht^{(I_2,\dotsc,I_k,I_1),\leq s+\ka(V)}_{G,V,N}.
\end{align*}

\subsection{}
At this point, we fix a place of $F$ and begin exploring the interplay between the local and global situations. Let $v$ be a closed point of $C$, write $r$ for the degree of $v$, and write $\cO_v$ for $\wh\sO_{C,v}$. Choose a uniformizer $z$ of $\cO_v$, which yields an identification $\cO_v=\bF_{q^r}\lb{z}$. Write $F_v$ for the fraction field of $\cO_v$, and write $\bD$ for the formal scheme $\Spf\cO_v$.

Fix a separable closure $\ov{F}_v$ of $F_v$, and fix an embedding $\ov{F}\ra\ov{F}_v$. By abuse of notation, write $G$ for the pullback of $G$ to $\cO_v$. Using $T_{F_v}$ for our maximal subtorus of $G_{F_v}$ and $B_{\ov{F}_v}$ for our Borel subgroup of $G_{\ov{F}_v}$, we can identify $F_i$ from \ref{ss:affineschubertvarieties} with the closure of $\bF_q(C_i)$ in $\ov{F}_v$ as well as identify $\bD_i$ from \ref{ss:affineschubertvarieties} with the formal completion of $C_i$ at the closed point $v_i$ of $C_i$ above $v$ induced by $\ov{F}\ra\ov{F}_v$.

\subsection{}\label{ss:Frobeniusfiberproducts}
The following two lemmas explain how to resolve the clash between our local and global base fields. Write $\bD^I$ for the $I$-th power of $\bD$ over $\bF_{q^r}$. Adopt the notation of \ref{ss:formaldisk}, and let $S=\Spec{R}$ be an affine scheme over $\bD^I$.
\begin{lem*}
  We have a natural isomorphism of affine formal schemes
  \begin{align*}
    \coprod_d\bD\times_{\bF_{q^r},d} S\ra^\sim\bD\times S,
  \end{align*}
where $\bD\times_{\bF_{q^r},d}S$ denotes the product of $S\ra\Spec\bF_{q^r}\ra^{\tau^d}\Spec\bF_{q^r}$ and $\bD$ over $\bF_{q^r}$, and $d$ runs over $\bZ/r$. Under this identification, $\tau:\bD\times S\ra\bD\times S$ on the right-hand side corresponds to the disjoint union of $\tau:\bD\times_{\bF_{q^r,d}}S\ra\bD\times_{\bF_{q^r,d-1}}S$ on the left-hand side.
\end{lem*}
\begin{proof}
Take $\bD\times_{\bF_{q^r,d}}S\ra\bD\times S$ to be the natural morphism. Since $\bF_{q^r}$ is finite Galois over $\bF_q$ with $\Gal(\bF_{q^r}/\bF_q)=\tau^{\bZ/r}$, the induced morphism above is an isomorphism. The last statement follows immediately.
\end{proof}

\begin{lem}\label{ss:changeFrobenius}
A local $G$-shtuka over $S$ is equivalent to data consisting of
  \begin{enumerate}[i)]
  \item for all $1\leq j\leq k$, a $G$-bundle $\sH_j$ on $\bD\times S$,
  \item for all $1\leq j\leq k$, an isomorphism of $G$-bundles
    \begin{align*}
      \chi_j:\sH_j|_{\bD\times S\ssm\sum_{i\in I_j}\Ga_i}\ra^\sim\sH_{j+1}|_{\bD\times S\ssm\sum_{i\in I_j}\Ga_i},
    \end{align*}
    where $\sH_{k+1}$ denotes the $G$-bundle $\prescript\tau{}{\sH_k}$.
  \end{enumerate}
\end{lem}
\begin{proof}
  Let $\sG$ be a local $G$-shtuka over $S$, and for all $1\leq j\leq k$, view $\sG_j$ as a $G$-bundle on $\bD\times_{\bF_{q^r}}S$. Using Lemma \ref{ss:Frobeniusfiberproducts}, we can form $\sH_j$ by taking $\prescript{\tau^d}{}{\sG_1}$ on $\bD\times_{\bF_{q^r},d}S$ for $1\leq d\leq r-1$ and $\sG_j$ on $\bD\times_{\bF_{q^r}}S$. Note that $\prescript\tau{}{\sH_1}$ is given by $\prescript{\tau^d}{}{\sG_1}$ on $\bD\times_{\bF_{q^r},d}S$ for all $1\leq d\leq r$. Therefore we can form $\chi_j$ by taking $\id$ on $\bD\times_{\bF_{q^r},d}S$ for $1\leq d\leq r-1$ and $\phi_j$ on $\bD\times_{\bF_{q^r}}S$.

  Conversely, let $\sH\coloneqq((\sH_j)_{j=1}^k,(\chi_j)_{j=1}^k)$ be as above. Write $(-)|_d$ for restrictions to $\bD\times_{\bF_{q^r},d}S$. Since $\Ga_i$ lies in $\bD\times_{\bF_{q^r}}S$ for all $i$ in $I$, our $\chi_j|_d$ is an isomorphism for all $1\leq j\leq k$ and $1\leq d\leq r-1$. By repeatedly using Lemma \ref{ss:Frobeniusfiberproducts}, this identifies $\sH_j|_d$ with $\prescript{\tau^d}{}{\sH_1}|_r$. Hence this also identifies $\sH_{k+1}|_r$ with $\prescript{\tau^r}{}{\sH_1}|_r$, so altogether we see that $\sH|_r$ yields a local $G$-shtuka over $S$.
\end{proof}

\subsection{}\label{ss:uniformizationformallyetale}
In our study of the uniformization morphism, we start by defining it on the level of formal stacks. Write $\textstyle\prod_{i\in I}\bD_i$ for the product of the $\bD_i$ over $\bF_{q^r}$, and write $\textstyle\prod_{i\in I}v_i$ for the product of the $v_i$ over $\bF_{q^r}$. Assume that $N$ and $v$ are disjoint, and write $\wh\Sht^{(I_1,\dotsc,I_k)}_{G,\mu_\bullet,N}|_{\prod_{i\in I}\bD_i}$ for the formal completion of $\Sht^{(I_1,\dotsc,I_k)}_{G,\mu_\bullet,N}|_{\prod_{i\in I}C_i\ssm N_i}$ along $\textstyle\prod_{i\in I}v_i$ in $\textstyle\prod_{i\in I}C_i\ssm N_i$.
\begin{prop*}
We have a canonical morphism
  \begin{align*}
\wh\Te: \fLocSht^{(I_1,\dotsc,I_k)}_{G,\mu_\bullet}|_{\prod_{i\in I}\bD_i}\ra\wh\Sht^{(I_1,\dotsc,I_k)}_{G,\mu_\bullet,N}|_{\prod_{i\in I}\bD_i}
  \end{align*}
of stacks over $\textstyle\prod_{i\in I}\bD_i$ that is formally \'etale.
\end{prop*}
This result generalizes cases of \cite[Theorem 5.3]{AH13}.
\begin{proof}
First, we define $\wh\Te$.  Let $(\sG,\de)$ be an $S$-point of $\fLocSht^{(I_1,\dotsc,I_k)}_{G,\mu_\bullet}|_{\prod_{i\in I}\bD_i}$, and let $((\sH_j)_{j=1}^k,(\chi_j)_{j=1}^k)$ be the data corresponding to $\sG$ as in Lemma \ref{ss:changeFrobenius}. For all $1\leq j\leq k$, Lemma \ref{ss:Frobeniusfiberproducts} shows that taking $\de_j$ on $\bD\times_{\bF_{q^r}}S$ and $\prescript{\tau^d}{}{\de_1}$ on $\bD\times_{\bF_{q^r},d}S$ for $1\leq d\leq r-1$ yields an isomorphism of $G$-bundles
  \begin{align*}
    \eps_j:\sH_j|_{\bD\times S\ssm v\times S}\ra^\sim G.
  \end{align*}
Beauville--Laszlo lets us use $\eps_j$ to glue $\sH_j$ and $G|_{C\times S\ssm v\times S}$ into a $G$-bundle $\sG_j^\Te$ on $C\times S$. Because the square in Definition \ref{ss:quasiisogenies}.b) commutes, Beauville--Laszlo also lets us glue $\chi_j$ and $\id$ into an isomorphism of $G$-bundles
  \begin{align*}
    \phi_j^\Te:\sG_j^\Te|_{C\times S\ssm\sum_{i\in I_j}\Ga_i}\ra^\sim\sG_{j+1}^\Te|_{C\times S\ssm\sum_{i\in I_j}\Ga_i},
  \end{align*}
  where we use Lemma \ref{ss:formaldisk} to identify $R\lb{z}$ with $\wh\cO_C(S)$, and $\sG^\Te_{k+1}$ denotes the $G$-bundle $\prescript\tau{}{\sG_1^\Te}$. As $\sG$ is bounded by $\mu_\bullet$, the global $G$-shtuka $\sG^\Te\coloneqq((\sG_j^\Te)_{j=1}^k,(\phi_j^\Te)_{j=1}^k)$ is too. Because $N$ and $v$ are disjoint, $\sG_j^\Te|_{N\times S}$ and $\phi_j^\Te|_{N\times S}$ are canonically trivial, so we have the trivial level-$N$ structure $\id=(\id)_{j=1}^k$ on $\sG^\Te$. Altogether, we define $\wh\Te(\sG,\de)$ to be the $S$-point $(\sG^\Te,\id)$ of $\wh\Sht^{(I_1,\dotsc,I_k)}_{G,\mu_\bullet,N}|_{\prod_{i\in I}\bD_i}$.

  To see that $\wh\Te$ is formally \'etale, let $J$ be an ideal of $R$ satisfying $J^n=0$, and write $\ov{S}\ra S$ for the associated closed embedding. For any commutative square
    \begin{align*}
    \xymatrix{\ov{S}\ar[r]\ar[d] & \fLocSht^{(I_1,\dotsc,I_k)}_{G,\mu_\bullet}|_{\prod_{i\in I}\bD_i}\ar[d]^-{\wh\Te} \\
S\ar[r]\ar@{.>}[ru] & \wh\Sht^{(I_1,\dotsc,I_k)}_{G,\mu_\bullet,N}|_{\prod_{i\in I}\bD_i},
    }
  \end{align*}
  write $(\ov\sG,\ov\de)$ for the $\ov{S}$-point of $\fLocSht^{(I_1,\dotsc,I_k)}_{G,\mu_\bullet}|_{\prod_{i\in I}\bD_i}$, and write $(\sF,\psi)$ for the $S$-point of $\wh\Sht^{(I_1,\dotsc,I_k)}_{G,\mu_\bullet,N}|_{\prod_{i\in I}\bD_i}$. The restriction of $\sF$ to $\bD\times S$ yields data as in Lemma \ref{ss:changeFrobenius}, which corresponds to a local $G$-shtuka $\sG$ over $S$. As $\sF$ is bounded by $\mu_\bullet$, our $\sG$ is too. Because the pullback of $\sF$ to $\ov{S}$ is $\wh\Te(\ov\sG,\ov\de)$, we see that the pullback of $\sG$ to $\ov{S}$ is $\ov\sG$. Therefore Proposition \ref{ss:rigidisogeny} yields a unique quasi-isogeny $\de$ from $\sG$ to $G$ whose pullback to $\ov{S}$ is $\ov\de$.

  Consider the $S$-point of $\fLocSht^{(I_1,\dotsc,I_k)}_{G,\mu_\bullet}|_{\prod_{i\in I}\bD_i}$ given by $(\sG,\de)$. The top triangle commutes by construction, and the bottom triangle commutes by the uniqueness of Beauville--Laszlo gluing. Finally, the uniqueness of Proposition \ref{ss:rigidisogeny} and Beauville--Laszlo gluing also imply that $(\sG,\de)$ is the unique such morphism, as desired.
\end{proof}

\subsection{}\label{ss:truncatedetale}
By restricting to a Harder--Narasimhan truncation and letting the (tame) level be large enough, we can pass from formal stacks to formal schemes. Maintain the assumptions of \ref{ss:uniformizationformallyetale}, Write $\wh\Sht^{(I_1,\dotsc,I_k),\leq s}_{G,\mu_\bullet,N}|_{\prod_{i\in I}\bD_i}$ for the formal completion of
\begin{align*}
 \Sht^{(I_1,\dotsc,I_k),\leq s}_{G,\mu_\bullet,N}|_{\prod_{i\in I}C_i\ssm N_i}
\end{align*}
along $\textstyle\prod_{i\in I}v_i$ in $\textstyle\prod_{i\in I}C_i\ssm N_i$, and write $\fLocSht^{(I_1,\dotsc,I_k),\leq s}_{G,\mu_\bullet}|_{\prod_{i\in I}\bD_i}$ for the preimage of $\wh\Sht^{(I_1,\dotsc,I_k),\leq s}_{G,\mu_\bullet,N}|_{\prod_{i\in I}\bD_i}$ under $\wh\Te$.

\begin{prop*}
  For large enough $\deg{N}$, the restriction
  \begin{align*}
    \wh\Te:\fLocSht^{(I_1,\dotsc,I_k),\leq s}_{G,\mu_\bullet}|_{\prod_{i\in I}\bD_i}\ra\wh\Sht^{(I_1,\dotsc,I_k),\leq s}_{G,\mu_\bullet,N}|_{\prod_{i\in I}\bD_i}
  \end{align*}
  is a morphism of formal schemes that is formally \'etale and locally formally of finite type.
\end{prop*}
\begin{proof}
Proposition \ref{ss:uniformizationformallyetale} shows that the restriction
  \begin{align*}
    \wh\Te:\fLocSht^{(I_1,\dotsc,I_k),\leq s}_{G,\mu_\bullet}|_{\prod_{i\in I}\bD_i}\ra\wh\Sht^{(I_1,\dotsc,I_k),\leq s}_{G,\mu_\bullet,N}|_{\prod_{i\in I}\bD_i}
  \end{align*}
  is formally \'etale. Because $\wh\Sht^{(I_1,\dotsc,I_k),\leq s}_{G,\mu_\bullet,N}|_{\prod_{i\in I}\bD_i}$ is an open substack of
  \begin{align*}
    \wh\Sht^{(I_1,\dotsc,I_k)}_{G,\mu_\bullet,N}|_{\prod_{i\in I}\bD_i},
  \end{align*}
we see that $\fLocSht^{(I_1,\dotsc,I_k),\leq s}_{G,\mu_\bullet}|_{\prod_{i\in I}\bD_i}$ is an open subsheaf of $\fLocSht^{(I_1,\dotsc,I_k)}_{G,\mu_\bullet}|_{\prod_{i\in I}\bD_i}$, so Theorem \ref{ss:formalmodulirepresentable} implies that $\fLocSht^{(I_1,\dotsc,I_k),\leq s}_{G,\mu_\bullet}|_{\prod_{i\in I}\bD_i}$ is a formal scheme that is locally formally of finite type over $\textstyle\prod_{i\in I}\bD_i$. For large enough $\deg{N}$, Proposition \ref{ss:globalshtukarepresentable} implies that $\wh\Sht^{(I_1,\dotsc,I_k),\leq s}_{G,\mu_\bullet,N}|_{\prod_{i\in I}\bD_i}$ is a formal scheme that is formally of finite type over $\textstyle\prod_{i\in I}\bD_i$. Hence $\wh\Te$ is locally formally of finite type, as desired.
\end{proof}

\subsection{}\label{ss:towerpreparations}
To add level at $v$, we need to pass to generic fibers as follows. Maintain the assumptions of \ref{ss:truncatedetale}, and assume that $\deg{N}$ is large enough as in Proposition \ref{ss:truncatedetale}. Proposition \ref{ss:globalshtukarepresentable} shows that $\Sht^{(I_1,\dotsc,I_k),\leq s}_{G,\mu_\bullet,N}|_{\prod_{i\in I}C_i\ssm N_i}$ is separated over $\textstyle\prod_{i\in I}C_i\ssm N_i$, so the natural morphism of adic spaces
\begin{align*}
  \wh\Sht^{(I_1,\dotsc,I_k),\leq s}_{G,\mu_\bullet,N}|_{\prod_{i\in I}\bD_i}\ra(\Sht^{(I_1,\dotsc,I_k),\leq s}_{G,\mu_\bullet,N})_{\prod_{i\in I}\bD_i}
\end{align*}
is an open embedding \cite[(4.6.iv.c)]{Hub94}. Write $\textstyle\prod_{i\in I}\Spa F_i$ for the product of the $\Spa F_i$ over $\bF_{q^r}$. For any non-negative integer $n$, write $\wh\Sht^{(I_1,\dotsc,I_k),\leq s}_{G,\mu_\bullet,nv+N}|_{\prod_{i\in I}\Spa F_i}$ for the preimage of $\wh\Sht^{(I_1,\dotsc,I_k),\leq s}_{G,\mu_\bullet,N}|_{\prod_{i\in I}\Spa F_i}$ in $(\Sht^{(I_1,\dotsc,I_k),\leq s}_{G,\mu_\bullet,nv+N})_{\prod_{i\in I}\Spa F_i}$.

Write $\prod_{i\in I}\Spd F_i$ for the product of the $\Spd F_i$ over $\bF_{q^r}$. Write
\begin{align*}
  \cLocSht^{(I_1,\dotsc,I_k),\leq s}_{G,\mu_\bullet,nv}|_{\prod_{i\in I}\Spd F_i}
\end{align*}
for the preimage of $(\fLocSht^{(I_1,\dotsc,I_k),\leq s}_{G,\mu_\bullet}|_{\prod_{i\in I}\Spa F_i})^\Diamond$ in $\cLocSht^{(I_1,\dotsc,I_k)}_{G,\mu_\bullet,nv}|_{\prod_{i\in I}\Spd F_i}$, where we use Theorem \ref{thm:analytificationisomorphism} to identify $(\fLocSht^{(I_1,\dotsc,I_k)}_{G,\mu_\bullet}|_{\prod_{i\in I}\bD_i})^\Diamond$ with
\begin{align*}
  \cLocSht^{(I_1,\dotsc,I_k)}_{G,\mu_\bullet}|_{\prod_{i\in I}\bD_i^\Diamond}.
\end{align*}

\subsection{}\label{ss:uniformizationtower}
We can now define the uniformization morphism on generic fibers. Maintain the assumptions of \ref{ss:towerpreparations}, and let $S=\Spa(R,R^+)$ be an affinoid perfectoid space over $\textstyle\prod_{i\in I}\Spa F_i$.
\begin{thm*}
  We have a canonical morphism
  \begin{align*}
    \Te_n:\cLocSht^{(I_1,\dotsc,I_k),\leq s}_{G,\mu_\bullet,nv}|_{\prod_{i\in I}\Spd F_i}\ra(\wh\Sht^{(I_1,\dotsc,I_k),\leq s}_{G,\mu_\bullet,nv+N}|_{\prod_{i\in I}\Spa F_i})^\Diamond
  \end{align*}
of locally spatial diamonds over $\textstyle\prod_{i\in I}\Spd F_i$ that is \'etale.
\end{thm*}
\begin{proof}
  First, we define $\Te_n$. By Theorem \ref{thm:analytificationisomorphism}, an $S$-point of
  \begin{align*}
    \cLocSht^{(I_1,\dotsc,I_k),\leq s}_{G,\mu_\bullet,nv}|_{\prod_{i\in I}\Spd F_i}
  \end{align*}
  corresponds to a cover $(S_\al)_\al$ of $S$ by rational open subspaces $S_\al=\Spa(R_\al,R^+_\al)$ with pairwise intersections $S_{\al\be}=\Spa(R_{\al\be},R^+_{\al\be})$, a family $(\sG^\al,\de^\al)$ of $\Spf R^+_\al$-points of $\fLocSht^{(I_1,\dotsc,I_k),\leq s}_{G,\mu_\bullet}|_{\prod_{i\in I}\bD_i}$ that agree on $\Spf R^+_{\al\be}$, and a level-$n$ structure $\psi$ on the analytic local $G$-shtuka over $S$ obtained from gluing the $(\sG^\al)^{\an}$.

  Proposition \ref{ss:globalshtukarepresentable} indicates that $\Sht^{(I_1,\dotsc,I_k),\leq s}_{G,\mu_\bullet,N}|_{\prod_{i\in I}C_i\ssm N_i}$ is locally of finite type over $\textstyle\prod_{i\in I}C_i\ssm N_i$, so for all $\al$, our $\Te(\sG^\al,\de^\al)$ yields an $R^+_\al$-point of
  \begin{align*}
    \Sht^{(I_1,\dotsc,I_k),\leq s}_{G,\mu_\bullet,N}|_{\prod_{i\in I}C_i\ssm N_i}.
  \end{align*}
  Write $\sG^{\al,\Te}$ for the resulting global $G$-shtuka over $\Spec R_\al$, which is bounded by $\mu_\bullet$ and has Harder--Narasimhan polygon bounded by $m$. Note that the pullback $\psi^\al$ of $\psi$ to $S_\al$ is precisely a level-$nv$ structure on $\sG^{\al,\Te}$, so we can form a level-$(nv+N)$ structure $\psi^{\al,\Te}$ on $\sG^{\al,\Te}$ by taking $\psi^\al$ on $nv$ and $\id$ on $N$. Then $(\sG^{\al,\Te},\psi^{\al,\Te})$ induces an $S_\al$-point of $\wh\Sht^{(I_1,\dotsc,I_k),\leq s}_{G,\mu_\bullet,nv+N}|_{\prod_{i\in I}\Spa F_i}$, and because the $\sG^{\al,\Te}$ and $\psi^{\al,\Te}$ agree on $\Spec{R_{\al\be}}$, the resulting family glues into an $S$-point. We define this $S$-point to be the value of $\Te_n$.

  To see that $\Te_n$ is \'etale, note that we have a commutative square
  \begin{align*}
    \xymatrix{\cLocSht^{(I_1,\dotsc,I_k),\leq s}_{G,\mu_\bullet,nv}|_{\prod_{i\in I}\Spd F_i}\ar[r]\ar[d]^-{\Te_n} & (\fLocSht^{(I_1,\dotsc,I_k),\leq s}_{G,\mu}|_{\prod_{i\in I}\Spa F_i})^\Diamond\ar[d]^-{\wh\Te^\Diamond}\\
    (\wh\Sht^{(I_1,\dotsc,I_k),\leq s}_{G,\mu_\bullet,nv+N}|_{\prod_{i\in I}\Spa F_i})^\Diamond\ar[r] & (\wh\Sht^{(I_1,\dotsc,I_k),\leq s}_{G,\mu_\bullet,N}|_{\prod_{i\in I}\Spa F_i})^\Diamond.
    }
  \end{align*}
  Theorem \ref{thm:analytificationisomorphism} and Proposition \ref{ss:localshtukalevelstructure} show that the top arrow is \'etale, and Proposition \ref{ss:globalshtukalevelstructure} and \cite[Lemma 15.6]{Sch17} imply that the bottom arrow is \'etale. Proposition \ref{ss:truncatedetale} and \cite[Lemma 15.6]{Sch17} imply that $\wh\Te^\Diamond$ is \'etale, so the 2-out-of-3 property \cite[Proposition 11.30]{Sch17} concludes that $\Te_n$ is \'etale.
\end{proof}

\subsection{}
As before, we reindex everything in terms of representations of the dual group. Maintain the assumptions of \ref{ss:towerpreparations}. Let $\wt{F}_v$ be the extension of $F_v$ as in \ref{ss:Lgroup}, and identify $\wt{F}_v$ with the completion of $\wt{F}$ at the place $\wt{v}$ of $\wt{F}$ above $v$ induced by $\ov{F}\ra\ov{F}_v$. Identify $\wh{G}$ with the dual group of $G_{F_v}$ over $\cO_E$, and write $\prescript{L}{}G_v$ for $\wh{G}\rtimes\Gal(\wt{F}_v/F_v)$. Note that we have a natural inclusion $\prescript{L}{}G_v\ra\prescript{L}{}G$.

Let $V$ be an object of $\Rep_E(\prescript{L}{}G_v)^I$. Write $\wh\Sht^{(I_1,\dotsc,I_k)}_{G,V,N}$ and $\wh\Sht^{(I_1,\dotsc,I_k),\leq s}_{G,V,N}$ for the formal completions of $\Sht^{(I_1,\dotsc,I_k)}_{G,V,N}$ and $\Sht^{(I_1,\dotsc,I_k),\leq s}_{G,V,N}$, respectively, along $v^I$ in $(C\ssm N)^I$. Proposition \ref{ss:uniformizationformallyetale} and descent yield a canonical morphism
\begin{align*}
\wh\Te:\fLocSht^{(I_1,\dotsc,I_k)}_{G,V}\ra\wh\Sht^{(I_1,\dotsc,I_k)}_{G,V,N}
\end{align*}
that is formally \'etale. Write $\fLocSht^{(I_1,\dotsc,I_k),\leq s}_{G,V}$ for the preimage of $\wh\Sht^{(I_1,\dotsc,I_k),\leq s}_{G,V,N}$ under $\Te$.

Write $\wh\Sht^{(I_1,\dotsc,I_k),\leq s}_{G,V,nv+N}$ for the preimage of $\wh\Sht^{(I_1,\dotsc,I_k),\leq s}_{G,V,N}$ in $(\Sht^{(I_1,\dotsc,I_k)}_{G,V,nv+N})_{(\Spa F_v)^I}$, and write $\cLocSht^{(I_1,\dotsc,I_k),\leq s}_{G,V,nv}$ for the preimage of $(\fLocSht^{(I_1,\dotsc,I_k),\leq s}_{G,V})_{(\Spa F_v)^I}^\Diamond$ in $\cLocSht^{(I_1,\dotsc,I_k)}_{G,V,nv}$, where we use Theorem \ref{thm:analytificationisomorphism} to identify $(\fLocSht^{(I_1,\dotsc,I_k),\leq s}_{G,V})_{(\Spa F_v)^I}^\Diamond$ with $\cLocSht^{(I_1,\dotsc,I_k)}_{G,V,0v}$. Theorem \ref{ss:uniformizationtower} and Galois descent yield a canonical morphism
\begin{align*}
\Te_n:\cLocSht^{(I_1,\dotsc,I_k),\leq s}_{G,V,nv}\ra(\wh\Sht^{(I_1,\dotsc,I_k),\leq s}_{G,V,nv+N})^\Diamond
\end{align*}
of locally spatial diamonds over $(\Spd F_v)^I$ that is \'etale.

\subsection{}\label{ss:localglobalpartialfrobenius}
We conclude by showing that the uniformization morphism is compatible with partial Frobenii. Maintain the assumptions of \ref{ss:towerpreparations}.

\begin{lem*}
  Our $\cFr^{(I_1,\dotsc,I_k)}$ restricts to a morphism
  \begin{align*}
    \cFr^{(I_1,\dotsc,I_k)}:\cLocSht^{(I_1,\dotsc,I_k),\leq s}_{G,V,nv}\ra\cLocSht^{(I_1,\dotsc,I_k),\leq s+r\ka(V)}_{G,V,nv}.
  \end{align*}
After enlarging $\deg{N}$, we can also form the $r$-fold composition
  \begin{align*}
    (\Fr^{(I_1,\dotsc,I_k)})_{\prescript{\tau^{r-1}}{}{(\Spa F_v)}^{I_1}\times(\Spa F_v)^{I\ssm I_1}}\circ\dotsb\circ(\Fr^{(I_1,\dotsc,I_k)})_{(\Spa F_v)^I},
  \end{align*}
  which yields a morphism
  \begin{align*}
  (\Fr^{(I_1,\dotsc,I_k)})^r_{(\Spa F_v)^I}:(\Sht^{(I_1,\dotsc,I_k),\leq s}_{G,V,nv+N})_{(\Spa F_v)^I}\ra(\Sht^{(I_1,\dotsc,I_k),\leq s+r\ka(V)}_{G,V,nv+N})_{(\Spa F_v)^I}.
  \end{align*}
Finally, we have $\Te_n\circ\cFr^{(I_1,\dotsc,I_k)}=(\Fr^{(I_1,\dotsc,I_k)})^{r,\Diamond}_{(\Spa F_v)^I}\circ\Te_n$.
\end{lem*}
\begin{proof}
  Write $\wh\Sht^{(I_1,\dotsc,I_k)}_{G,V,N}|_{\prescript\tau{}\bD^{I_1}\times\bD^{I\ssm I_1}}$ for the formal completion of $\Sht^{(I_1,\dotsc,I_k)}_{G,V,N}$ along $\tau(v)^{I_1}\times v^{I\ssm I_1}$ in $(C\ssm N)^I$. We see from \ref{ss:globalshtukapartialfrobenius} that $\Fr^{(I_1,\dotsc,I_k)}$ induces a morphism
  \begin{align*}
    \wh\Fr^{(I_1,\dotsc,I_k)}:\wh\Sht^{(I_1,\dotsc,I_k)}_{G,V,N}\ra\wh\Sht^{(I_2,\dotsc,I_k,I_1)}_{G,V,N}|_{\prescript\tau{}\bD^{I_1}\times\bD^{I\ssm I_1}}.
  \end{align*}
  If $r=1$, then stop here. Otherwise, the relative effective Cartier divisors on $C\times S$ corresponding to $S$-points of $\prescript\tau{}\bD$ and $\bD$ are disjoint, so the right-hand side is naturally isomorphic to $\wh\Sht^{(I_1,\dotsc,I_k)}_{G,V,N}|_{\prescript\tau{}\bD^{I_1}\times\bD^{I\ssm I_1}}$. By forming $\wh\Fr^{(I_1,\dotsc,I_k)}|_{\prescript\tau{}\bD^{I_1}\times\bD^{I\ssm I_1}}$ and repeating this $r-1$ more times, we obtain a morphism
  \begin{align*}
    \wh\Fr^{(I_1,\dotsc,I_k)}|_{\prescript{\tau^{r-1}}{}\bD^{I_1}\times\bD^{I\ssm I_1}}\circ\dotsb\circ\wh\Fr^{(I_1,\dotsc,I_k)}:\wh\Sht^{(I_1,\dotsc,I_k)}_{G,V,N}\ra\wh\Sht^{(I_2,\dotsc,I_k,I_1)}_{G,V,N}.
  \end{align*}
  Tracing through our identifications shows that
  \begin{align*}
    \wh\Te\circ\fFr^{(I_1,\dotsc,I_k)}=\wh\Fr^{(I_1,\dotsc,I_k)}|_{\prescript{\tau^{r-1}}{}\bD^{I_1}\times\bD^{I\ssm I_1}}\circ\dotsb\circ\wh\Fr^{(I_1,\dotsc,I_k)}\circ\wh\Te,
  \end{align*}
so \ref{ss:globalshtukapartialfrobenius} implies that $\fFr^{(I_1,\dotsc,I_k)}$ restricts to a morphism
  \begin{align*}
    \fLocSht^{(I_1,\dotsc,I_k),\leq s}_{G,V}\ra\fLocSht^{(I_1,\dotsc,I_k),\leq s+r\ka(V)}_{G,V}.
  \end{align*}
Pulling back to $\Spa F_v$ and using Theorem \ref{thm:analytificationisomorphism} yields the desired result.
\end{proof}

\section{Local-global compatibility}\label{s:localglobalcompatibility}
Our goal in this section is to prove Theorem A. First, we recall the coefficient sheaves used for the cohomology of the global and local moduli problems. We show that they are compatible under the uniformization morphism from \S\ref{s:uniformizing}. Next, we recall smoothness theorems for our cohomology sheaves, which are due to Xue \cite{Xue20b} in the global case and Fargues--Scholze \cite{FS21} in the local case.

These smoothness theorems yield global and local excursion operators. Using the uniformization morphism, we prove that the global and local excursion operators are compatible. From this, we deduce that the Bernstein center elements constructed by Genestier--Lafforgue \cite{GL17} agree with those constructed by Fargues--Scholze \cite{FS21}, and we also deduce Theorem A.

\subsection{}\label{ss:globalsatake}
For the cohomology of the moduli of global $G$-shtukas, we use the following sheaves obtained via geometric Satake. For large enough $e$, recall from \ref{ss:affineschubertvarieties} that the natural $L^+_I(G)$-action on $\Gr^{(I_1,\dotsc,I_k)}_{G,\mu_\bullet}|_{\prod_{i\in I}C_i}$ factors through $L^e_I(G)$. Write $\mathrm{A}^{(I_1,\dotsc,I_k)}_{G,\mu_\bullet,N}$ for the $L^e_I(G)$-bundle on $\Sht^{(I_1,\dotsc,I_k)}_{G,\mu_\bullet,N}|_{\prod_{i\in I}C_i\ssm N_i}$ whose fiber over $(\sG,\psi)$ parametrizes trivializations of the $G$-bundle $\prescript\tau{}{\sG_1}|_{e\sum_{i\in I}\Ga_i}$. Note that we have a natural $L^e_I(G)$-equivariant morphism $\mathrm{A}^{(I_1,\dotsc,I_k)}_{G,\mu_\bullet,N}\ra\Gr^{(I_1,\dotsc,I_k)}_{G,\mu_\bullet}|_{\prod_{i\in I}C_i}$, which is smooth by \cite[p.~867]{Laf16}. Write $\mathrm{A}^{(I_1,\dotsc,I_k),\leq s}_{G,\mu_\bullet,N}$ for the restriction of $\mathrm{A}^{(I_1,\dotsc,I_k)}_{G,\mu_\bullet,N}$ to
\begin{align*}
  \Sht^{(I_1,\dotsc,I_k),\leq s}_{G,\mu_\bullet,N}|_{\prod_{i\in I}C_i\ssm N_i}.
\end{align*}

Write $V_{\mu_\bullet}$ for the highest weight representation of $\wh{G}^I$ corresponding to $\mu_\bullet$, and write $\cS^{(I_1,\dotsc,I_k)}_{\mu_\bullet,E}$ for the corresponding object of $D(\Gr^{(I_1,\dotsc,I_k)}_{G,\mu_\bullet}|_{\prod_{i\in I}C_i},E)$ under geometric Satake. Write $\cF_{\mu_\bullet,N,E}^{(I_1,\dotsc,I_k)}$ for the object of $D(\Sht^{(I_1,\dotsc,I_k)}_{G,\mu_\bullet,N}|_{\prod_{i\in I}C_i\ssm N_i},E)$ obtained by first pulling back $\cS^{(I_1,\dotsc,I_k)}_{\mu_\bullet,E}$ to $\mathrm{A}^{(I_1,\dotsc,I_k)}_{G,\mu_\bullet,N}$ and then using $L^e_I(G)$-equivariance to descend along $\mathrm{A}^{(I_1,\dotsc,I_k)}_{G,\mu_\bullet,N}\ra\Sht^{(I_1,\dotsc,I_k)}_{G,\mu_\bullet,N}|_{\prod_{i\in I}C_i\ssm N_i}$. Finally, write $\cF^{(I_1,\dotsc,I_k),\leq s}_{\mu_\bullet,N,E}$ for the restriction of $\cF^{(I_1,\dotsc,I_k)}_{\mu_\bullet,N,E}$ to $\Sht^{(I_1,\dotsc,I_k),\leq s}_{G,\mu_\bullet,N}|_{\prod_{i\in I}C_i\ssm N_i}$.

\subsection{}\label{ss:globalsatakemodxi}
We will also take cohomology after quotienting by a lattice $\Xi$ of $Z(F)\bs Z(\bA)$, where a \emph{lattice} means a discrete torsionfree cocompact subgroup. We proceed as follows. Note that $L^+_I(Z)$ acts trivially on $\Gr^{(I_1,\dotsc,I_k)}_{G,\mu_\bullet}|_{\prod_{i\in I}C_i}$, so the natural $L^+_I(G)$-action factors through $L^+_I(G^{\ad})$. For large enough $e$, \ref{ss:affineschubertvarieties} indicates that this factors through $L^e_I(G^{\ad})$. Now $L^e_I(Z)$ acts trivially on the objects of
\begin{align*}
D(\Gr^{(I_1,\dotsc,I_k)}_{G,\mu_\bullet}|_{\prod_{i\in I}C_i},E)
\end{align*}
obtained from geometric Satake \cite[Th\'eor\`eme 12.16]{Laf16}, so these objects are $L^e_I(G^{\ad})$-equivariant. Adapting the construction in \ref{ss:globalsatake} yields an object $\cF^{(I_1,\dotsc,I_k)}_{\Xi,\mu_\bullet,N,E}$ of
\begin{align*}
 D(\Sht^{(I_1,\dotsc,I_k)}_{G,\mu_\bullet,N}\!\!/\Xi\,|_{\prod_{i\in I}C_i\ssm N_i},E),
\end{align*}
and we see that the pullback of $\cF^{(I_1,\dotsc,I_k)}_{\Xi,\mu_\bullet,N,E}$ to $\Sht^{(I_1,\dotsc,I_k)}_{G,\mu_\bullet,N}|_{\prod_{i\in I}C_i\ssm N_i}$ equals $\cF^{(I_1,\dotsc,I_k)}_{\Xi,\mu_\bullet,N,E}$.

\subsection{}
Next, we describe the sheaves used for the homology of the moduli of local $G$-shtukas. Recall $\cL^e_I(G)$ and $\cL^+_I(G)$ from Definition \ref{ss:analyticBDgrassmannian}. For large enough $e$, \ref{ss:affineschubertvarieties} and Lemma \ref{ss:BDgrassmanniancomparison} indicate that the natural $\cL^+_I(G)$-action on $\cGr^{(I_1,\dotsc,I_k)}_{G,\mu_\bullet}|_{\prod_{i\in I}\bD^\Diamond_i}$ factors through $\cL^e_I(G)$. Write $\cA^{(I_1,\dotsc,I_k)}_{G,\mu_\bullet,nv}$ for the $\cL^e_I(G)$-bundle on $\cLocSht^{(I_1,\dotsc,I_k)}_{G,\mu_\bullet,nv}|_{\prod_{i\in I}\Spd F_i}$ whose fiber over $(\sG,\de,\psi)$ parametrizes trivializations of the $G$-bundle $\prescript{\tau^r}{}{\sG_1}|_{e\sum_{i\in I}\Ga_i}$. Note that we have a natural $\cL^e_I(G)$-equivariant morphism
\begin{align*}
  \cA^{(I_1,\dotsc,I_k)}_{G,\mu_\bullet,nv}\ra\cGr^{(I_1,\dotsc,I_k)}_{G,\mu_\bullet}|_{\prod_{i\in I}\Spd F_i}.
\end{align*}

Recall $\La$ from \ref{ss:FSnaivecompare}, and write $\textstyle\prod_{i\in I}\Spd\breve{F}_i$ for the product of the $\Spd\breve{F}_i$ over $\ov\bF_q$. Write $\prescript\prime{}\cF^{(I_1,\dotsc,I_k)}_{\mu_\bullet,nv,\La}$ for the object of $D_{\solid}(\cLocSht^{(I_1,\dotsc,I_k)}_{G,\mu_\bullet,nv}|_{\prod_{i\in I}\Spd\breve{F}_i},\La)$ obtained from \cite[Theorem VI.11.1]{FS21} and $V_{\mu_\bullet}$ by first applying the double-dual embedding as in \cite[p.~264]{FS21}, then pulling back to $\cA^{(I_1,\dotsc,I_k)}_{G,\mu_\bullet,nv}|_{\prod_{i\in I}\Spd\breve{F}_i}$, and finally using $\cL^e_I(G)$-equivariance and \cite[Proposition 17.3]{Sch17} to descend along
\begin{align*}
\cA^{(I_1,\dotsc,I_k)}_{G,\mu_\bullet,nv}\ra\cLocSht^{(I_1,\dotsc,I_k)}_{G,\mu_\bullet,nv}|_{\prod_{i\in I}\Spd\breve{F}_i}.
\end{align*}

Write $\cA^{(I_1,\dotsc,I_k),\leq s}_{G,\mu_\bullet,nv}$ for the restriction of $\cA^{(I_1,\dotsc,I_k)}_{G,\mu_\bullet,nv}$ to $\cLocSht^{(I_1,\dotsc,I_k),\leq s}_{G,\mu_\bullet,nv}|_{\prod_{i\in I}\Spd F_i}$, and write $\prescript\prime{}\cF^{(I_1,\dotsc,I_k),\leq s}_{\mu_\bullet,nv,\La}$ for the restriction of $\prescript\prime{}\cF^{(I_1,\dotsc,I_k)}_{\mu_\bullet,nv,\La}$ to
\begin{align*}
  \cLocSht^{(I_1,\dotsc,I_k),\leq s}_{G,\mu_\bullet,nv}|_{\prod_{i\in I}\Spd\breve{F}_i}.
\end{align*}

\subsection{}\label{ss:localglobalsatakecomparison}
Our local and global coefficient sheaves are compatible under $\Te_n$ in the following sense. Adapt the assumptions of \ref{ss:towerpreparations}, and write $\textstyle\prod_{i\in I}\Spa\breve{F}_i$ for the product of the $\Spa\breve{F}_i$ over $\ov\bF_q$. Write $(\cF^{(I_1,\dotsc,I_k),\leq s}_{\mu_\bullet,nv+N,E})_{\prod_{i\in I}\Spa\breve{F}_i}$ for the object of
\begin{align*}
 D((\Sht^{(I_1,\dotsc,I_k),\leq s}_{G,\mu_\bullet,nv+N})_{\prod_{i\in I}\Spa\breve{F}_i},E)
\end{align*}
obtained by analytifying $\cF^{(I_1,\dotsc,I_k),\leq s}_{\mu_\bullet,nv+N,E}$ as in \cite[(3.2.8)]{Hub96}. Because
\begin{align*}
(\Sht^{(I_1,\dotsc,I_k),\leq s}_{G,\mu_\bullet,nv+N})_{\prod_{i\in I}\Spa\breve{F}_i}
\end{align*}
is an analytic adic space, \cite[Lemma 15.6]{Sch17} and \cite[Remark 14.14]{Sch17} indicate that $(\cF^{(I_1,\dotsc,I_k),\leq s}_{\mu_\bullet,nv+N,E})_{\prod_{i\in I}\Spa\breve{F}_i}$ yields an object $(\cF^{(I_1,\dotsc,I_k),\leq s}_{\mu_\bullet,nv+N,E})_{\prod_{i\in I}\Spa\breve{F}_i}^\Diamond$ of
\begin{align*}
D_{\et}((\Sht^{(I_1,\dotsc,I_k),\leq s}_{G,\mu_\bullet,nv+N})_{\prod_{i\in I}\Spa\breve{F}_i}^\Diamond,E).
\end{align*}
\begin{lem*}
$(\cF^{(I_1,\dotsc,I_k),\leq s}_{\mu_\bullet,nv+N,E})_{\prod_{i\in I}\Spa\breve{F}_i}^\Diamond$ is universally locally acyclic over $\textstyle\prod_{i\in I}\Spd\breve{F}_i$. Moreover, its image $\prescript\prime{}{(}\cF^{(I_1,\dotsc,I_k),\leq s}_{\mu_\bullet,nv+N,E})_{\prod_{i\in I}\Spa\breve{F}_i}^\Diamond$ in $D_{\solid}((\Sht^{(I_1,\dotsc,I_k),\leq s}_{G,\mu_\bullet,nv+N})_{\prod_{i\in I}\Spa\breve{F}_i}^\Diamond,E)$ under the double-dual embedding as in \cite[p.~260]{FS21} satisfies
  \begin{align*}
   \Te^*_n\big[\prescript\prime{}{(}\cF^{(I_1,\dotsc,I_k),\leq s}_{\mu_\bullet,nv+N,E})_{\prod_{i\in I}\Spa\breve{F}_i}^\Diamond\big] = \prescript\prime{}\cF^{(I_1,\dotsc,I_k),\leq s}_{\mu_\bullet,nv,E}.
  \end{align*}
\end{lem*}
\begin{proof}
  We start by rewriting $(\cF^{(I_1,\dotsc,I_k),\leq s}_{\mu_\bullet,nv+N,E})_{\prod_{i\in I}\Spa\breve{F}_i}^\Diamond$ as follows. Since
  \begin{align*}
    (\Gr^{(I_1,\dotsc,I_k)}_{G,\mu_\bullet})_{\prod_{i\in I}\Spa\breve{F}_i}
  \end{align*}
  is an analytic adic space, \cite[Lemma 15.6]{Sch17} and \cite[Remark 14.14]{Sch17} indicate that $(\cS^{(I_1,\dotsc,I_k)}_{\mu_\bullet,E})^\Diamond_{\prod_{i\in I}\Spa\breve{F}_i}$ yields an object of $D_{\et}((\Gr^{(I_1,\dotsc,I_k)}_{G,\mu_\bullet})_{\prod_{i\in I}\Spa\breve{F}_i}^\Diamond,E)$. By first pulling back $(\cS^{(I_1,\dotsc,I_k)}_{\mu_\bullet,E})^\Diamond_{\prod_{i\in I}\Spa\breve{F}_i}$ to $(\mathrm{A}^{(I_1,\dotsc,I_k),\leq s}_{G,\mu_\bullet,nv+N})_{\prod_{i\in I}\Spa\breve{F}_i}^\Diamond$ and then using $\cL^e_I(G)$-equivariance and \cite[Proposition 17.3]{Sch17} to descend along
  \begin{align*}
   (\mathrm{A}^{(I_1,\dotsc,I_k),\leq s}_{G,\mu_\bullet,nv+N})^\Diamond_{\prod_{i\in I}\Spa\breve{F}_i}\ra(\Sht^{(I_1,\dotsc,I_k),\leq s}_{G,\mu_\bullet,nv+N})_{\prod_{i\in I}\Spa\breve{F}_i}^\Diamond,
  \end{align*}
  where we use Lemma \ref{ss:BDgrassmanniancomparison} to identify $(L^e_I(G))^\Diamond_{\bD^I}$ with $\cL^e_I(G)$, we see that the resulting object of $D_{\et}((\Sht^{(I_1,\dotsc,I_k),\leq s}_{G,\mu_\bullet,nv+N})_{\prod_{i\in I}\Spa\breve{F}_i}^\Diamond,E)$ equals $(\cF^{(I_1,\dotsc,I_k),\leq s}_{\mu_\bullet,nv+N,E})_{\prod_{i\in I}\Spa\breve{F}_i}^\Diamond$.

Let us prove the first claim. By using the explicit description in \cite[Proposition VI.7.9]{FS21} and the fiberwise criterion for perversity \cite[Corollary VI.7.6]{FS21}, we see that $(\cS^{(I_1,\dotsc,I_k)}_{\mu_\bullet,E})^\Diamond$ equals the object obtained from \cite[Theorem VI.11.1]{FS21} and $V_{\mu_\bullet}$, where we use Lemma \ref{ss:BDgrassmanniancomparison} to identify $(\Gr^{(I_1,\dotsc,I_k)}_{G,\mu_\bullet})_{\prod_{i\in I}\Spa\breve{F}_i}^\Diamond$ with $\cGr^{(I_1,\dotsc,I_k)}_{G,\mu_\bullet}|_{\prod_{i\in I}\Spd\breve{F}_i}$. Hence $(\cS^{(I_1,\dotsc,I_k)}_{\mu_\bullet,E})^\Diamond$ is universally locally acyclic over $\textstyle\prod_{i\in I}\Spd\breve{F}_i$. Now \ref{ss:globalsatake} and \cite[Proposition 24.4]{Sch17} show that
  \begin{align*}
    (\mathrm{A}^{(I_1,\dotsc,I_k),\leq s}_{G,\mu_\bullet,nv+N})_{\prod_{i\in I}\Spa\breve{F}_i}^\Diamond\ra(\Gr^{(I_1,\dotsc,I_k)}_{G,\mu_\bullet})_{\prod_{i\in I}\Spa\breve{F}_i}^\Diamond
  \end{align*}
  is $\ell$-cohomologically smooth, so \cite[Proposition IV.2.13 (i)]{FS21} implies that the pullback of $(\cS^{(I_1,\dotsc,I_k)}_{\mu_\bullet,E})^\Diamond$ to $(\mathrm{A}^{(I_1,\dotsc,I_k),\leq s}_{G,\mu_\bullet,nv+N})_{\prod_{i\in I}\Spa\breve{F}_i}^\Diamond$ remains universally locally acyclic over $\textstyle\prod_{i\in I}\Spd\breve{F}_i$. Applying \cite[Proposition 24.4]{Sch17} again shows that
  \begin{align*}
    (\mathrm{A}^{(I_1,\dotsc,I_k),\leq s}_{G,\mu_\bullet,nv+N})_{\prod_{i\in I}\Spa\breve{F}_i}^\Diamond\ra(\Sht^{(I_1,\dotsc,I_k),\leq s}_{G,\mu_\bullet,nv+N})_{\prod_{i\in I}\Spa\breve{F}_i}^\Diamond
  \end{align*}
  is $\ell$-cohomologically smooth, so \cite[Proposition IV.2.13 (ii)]{FS21} implies that
  \begin{align*}
    (\cF^{(I_1,\dotsc,I_k),\leq s}_{\mu_\bullet,nv+N,E})_{\prod_{i\in I}\Spa\breve{F}_i}^\Diamond
  \end{align*}
  is universally locally acyclic over $\textstyle\prod_{i\in I}\Spd\breve{F}_i$, as desired.

  For the second claim, note that $\Te_n$ naturally induces a morphism
  \begin{align*}
    \cA^{(I_1,\dotsc,I_k),\leq s}_{G,\mu_\bullet,nv}\ra(\mathrm{A}^{(I_1,\dotsc,I_k),\leq s}_{G,\mu_\bullet,nv+N})_{\prod_{i\in I}\Spa F_i}^\Diamond
  \end{align*}
  such that the diagram
  \begin{align*}
    \xymatrix{\cGr^{(I_1,\dotsc,I_k)}_{G,\mu_\bullet}|_{\prod_{i\in I}\Spd\breve{F}_i}\ar@{=}[r] & (\Gr^{(I_1,\dotsc,I_k)}_{G,\mu_\bullet})_{\prod_{i\in I}\Spa\breve{F}_i}^\Diamond\\
    \cA^{(I_1,\dotsc,I_k),\leq s}_{G,\mu_\bullet,nv}|_{\prod_{i\in I}\Spd\breve{F}_i}\ar[r]\ar[u]\ar[d] &(\mathrm{A}^{(I_1,\dotsc,I_k),\leq s}_{G,\mu_\bullet,nv+N})_{\prod_{i\in I}\Spa\breve{F}_i}^\Diamond\ar[u]\ar[d]\\
    \cLocSht^{(I_1,\dotsc,I_k),\leq s}_{G,\mu_\bullet,nv}|_{\prod_{i\in I}\Spd\breve{F}_i}\ar[r]^-{\Te_n} & (\Sht^{(I_1,\dotsc,I_k),\leq s}_{G,\mu_\bullet,nv+N})^\Diamond_{\prod_{i\in I}\Spa\breve{F}_i}
    }
  \end{align*}
commutes. Therefore the above discussion yields the desired result.
\end{proof}

\subsection{}\label{ss:globalcohomologysheaves}
We now consider the cohomology of the moduli of global $G$-shtukas. Let $V$ be an object of $\Rep_E(\prescript{L}{}G)^I$. Note that the $\cF^{(I_1,\dotsc,I_k)}_{\mu_\bullet,N,E}$ and $\cF^{(I_1,\dotsc,I_k),\leq s}_{\mu_\bullet,N,E}$ naturally descend to objects $\cF^{(I_1,\dotsc,I_k)}_{V,N,E}$ and $\cF^{(I_1,\dotsc,I_k),\leq s}_{V,N,E}$ of $D(\Sht^{(I_1,\dotsc,I_k)}_{G,V,N},E)$ and $D(\Sht^{(I_1,\dotsc,I_k),\leq s}_{G,V,N},E)$, respectively, where $\mu_\bullet$ runs over highest weights appearing in $V_{\ov\bQ_\ell}|_{\wh{G}^I}$ with multiplicity.

Recall that $f^{\mathrm{S}}_!\cF^{(I_1,\dotsc,I_k),\leq s}_{V,N,E}$ is independent of the ordered partition $I_1,\dotsc,I_k$ \cite[p.~868]{Laf16}, so we write it as $\cH^{I,\leq s}_{V,N,E}$. The same holds for $f^{\mathrm{S}}_!\cF^{(I_1,\dotsc,I_k)}_{V,N,E}$, so we write it as $\cH^I_{V,N,E}$. Because $\Sht^{(I_1,\dotsc,I_k)}_{G,V,N}$ is the increasing union of the $\Sht^{(I_1,\dotsc,I_k),\leq s}_{G,V,N}$, we have $\cH^I_{V,N,E}=\varinjlim_s\cH^{I,\leq s}_{V,N,m,E}$. Note that \ref{ss:heckecorrespondences} yields an action of $C_c(K_N\bs G(\bA)/K_N,E)$ on $\cH^I_{V,N,E}$.

\subsection{}\label{ss:Xue}
Recall the following smoothness result of Xue \cite{Xue20b}. Write $\ov\eta$ for $\Spec\ov{F}$, and write $\De$ for diagonal morphisms. Write $W_F$ for the absolute Weil group of $F$, and write $\val_F:W_F\ra\bZ$ for the homomorphism that sends geometric $q$-Frobenii to $1$. Write $U\subseteq C$ for the largest open subspace where $G_U$ is reductive.
\begin{thm*}
The cohomology sheaves of $\cH^I_{V,N,E}|_{(U\ssm N)^I}$ are ind-smooth, and the cohomology sheaves of $\cH^I_{V,N,E}|_{\De(\ov\eta)}$ have a natural action of $W_F^I$. For any $\ga_\bullet=(\ga_i)_{i\in I}$ in $W_F^I$, the $\ga_\bullet$-action sends the image of the cohomology groups of $\cH^{I,\leq s}_{V,N,E}|_{\De(\ov\eta)}$ to the image of the cohomology groups of $\cH^{I,\leq s'}_{V,N,E}|_{\De(\ov\eta)}$ for $s'\geq s+\textstyle\sum_{i\in I}\max\{0,\val_F(\ga_i)\}$.
\end{thm*}
\begin{proof}
The first claim follows from the proof of \cite[Theorem 6.0.12]{Xue20b}, and the $W_F^I$-action follows from the proof of \cite[Proposition 6.0.10]{Xue20b}. The last claim follows from \ref{ss:globalshtukapartialfrobenius}.
\end{proof}

\subsection{}\label{ss:globalcohomologysheavesmodxi}
Let us record the analogous results after quotienting by $\Xi$. Let $V$ be an object of $\Rep_E(\prescript{L}{}G)^I$, and note that the $\cF^{(I_1,\dotsc,I_k)}_{\Xi,\mu_\bullet,N,E}$ naturally descend to an object $\cF^{(I_1,\dotsc,I_k)}_{\Xi,V,N,E}$ of $D(\Sht^{(I_1,\dotsc,I_k)}_{G,V,N}\!\!/\Xi,E)$, where $\mu_\bullet$ runs over highest weights appearing in $V_{\ov\bQ_\ell}|_{\wh{G}^I}$ with multiplicity.

Recall that $f^{\mathrm{S}}_!\cF^{(I_1,\dotsc,I_k)}_{\Xi,V,N,E}$ is independent of the ordered partition $I_1,\dotsc,I_k$ \cite[p.~868]{Laf16}, so we write it as $\cH^I_{\Xi,V,N,E}$. Note that \ref{ss:heckecorrespondences} yields an action of
\begin{align*}
  C_c(K_N\bs G(\bA)/K_N,E)
\end{align*}
on $\cH^I_{\Xi,V,N,E}$. Recall that the cohomology sheaves of $\cH^I_{\Xi,V,N,E}|_{(U\ssm N)^I}$ are ind-smooth \cite[Theorem 6.0.12]{Xue20b}, and the cohomology sheaves of $\cH^I_{\Xi,V,N,E}|_{\De(\ov\eta)}$ have a natural action of $W^I_F$ \cite[Proposition 6.0.10]{Xue20b}.

\subsection{}\label{ss:localcohomologysheaves}
Next, we consider the homology of the moduli of local $G$-shtukas. Let $V$ be an object of $\Rep_{\cO_E}(\prescript{L}{}G_v)^I$. Note that the $\prescript\prime{}\cF^{(I_1,\dotsc,I_k)}_{\mu_\bullet,nv,\La}$ and $\prescript\prime{}\cF^{(I_1,\dotsc,I_k),\leq s}_{\mu_\bullet,nv,\La}$ naturally descend to objects $\prescript\prime{}\cF^{(I_1,\dotsc,I_k)}_{V,nv,\La}$ and $\prescript\prime{}\cF^{(I_1,\dotsc,I_k),\leq s}_{V,nv,\La}$ of $D_{\solid}(\cLocSht^{(I_1,\dotsc,I_k)}_{G,V,nv}|_{(\Spd\breve{F}_v)^I},\La)$ and $D_{\solid}(\cLocSht^{(I_1,\dotsc,I_k),\leq s}_{G,V,nv}|_{(\Spd\breve{F}_v)^I},\La)$, respectively, where $\mu_\bullet$ runs over highest weights appearing in $V_{\ov\bQ_\ell}|_{\wh{G}^I}$ with multiplicity.

Recall the notation of \ref{ss:FSnaivecompare}. Since the square
\begin{align*}
  \xymatrix{\cLocSht^{(I)}_{G,V,nv}|_{(\Spd\breve{F}_v)^I}\ar[r]^-c & \cM^I_{G,V,K_n}|_{(\Spd\breve{F}_v)^I}\ar[d] \\
  \cLocSht^{(I_1,\dotsc,I_k)}_{G,V,nv}|_{(\Spd\breve{F}_v)^I}\ar[u]\ar[r] & [\cL^e_I(G)\bs\cGr^{(I)}_{G,V}|_{(\Spd\breve{F}_v)^I}]
  }
\end{align*}
commutes, where $\cGr^{(I)}_{G,V}$ denotes the natural descent of $\coprod_{\mu_\bullet}\cGr^{(I)}_{G,\mu_\bullet}|_{(\wt\bD^I)^\Diamond}$ to $(\bD^I)^\Diamond$, the $\prescript\prime{}\cF^{(I_1,\dotsc,I_k)}_{V,nv,\La}$ defined in \ref{ss:FSnaivecompare} agrees with the $\prescript\prime{}\cF^{(I_1,\dotsc,I_k)}_{V,nv,\La}$ defined here.

The smallness of convolution implies that $f_\natural^\cM(\prescript\prime{}\cF^{(I_1,\dotsc,I_k),\leq s}_{V,nv,\La})$ is independent of the ordered partition $I_1,\dotsc,I_k$, so we write it as $\cH^{\loc,I,\leq s}_{V,nv,\La}$. The same holds for $f_\natural^\cM(\prescript\prime{}\cF^{(I_1,\dotsc,I_k)}_{V,nv,\La})$, so we write it as $\cH^{\loc,I}_{V,nv,\La}$. Because $\cLocSht^{(I_1,\dotsc,I_k)}_{G,V,nv}$ is the increasing union of the $\cLocSht^{(I_1,\dotsc,I_k),\leq s}_{G,V,nv}$, we have $\cH^{\loc,I}_{V,nv,\La}=\textstyle\varinjlim_s\cH^{\loc,I,\leq s}_{V,nv,\La}$. Note that Proposition \ref{ss:localheckecorrespondences} yields an action of $C_c(K_n\bs G(F_v)/ K_n,E)$ on $\cH^{\loc,I}_{V,nv,\La}$.

Write $\bC_v$ for the completion of $\ov{F}_v$, and write $\ov\eta_v$ for $\Spd\bC_v$. Theorem \ref{ss:FSnaivecompare} yields a natural action of $W_{F_v}^I$ on the cohomology groups of $\cH^{\loc,I}_{V,nv,\La}|_{\De(\ov\eta_v)}$. For any $\ga_\bullet$ in $W_{F_v}$, Lemma \ref{ss:localglobalpartialfrobenius} and Lemma \ref{ss:FSpartialfrobeniuscomparison} imply that the $\ga_\bullet$-action sends the image of the cohomology groups of $\cH^{\loc,I,\leq s}_{V,nv,\La}$ to the image of the cohomology groups of $\cH^{\loc,I,\leq s'}_{V,nv,\La}$ for $s'\geq s+\textstyle\sum_{i\in I}\max\{0,\val_F(\ga_i)\}$.

\subsection{}\label{ss:excursionalgebra}
Let us recall some facts about excursion algebras. For any abstract group $W$, finite group $Q$ with a pinned action on $\wh{G}$, and group homomorphism $W\ra Q$, write $\Exc(W,\wh{G})$ for the excursion algebra over $\cO_E$ as in \cite[Definition VIII.3.4]{FS21}. Recall that $\Exc(W,\wh{G})$ is flat over $\cO_E$ and has canonical generators $S_{I,V,x,\xi,\ga_\bullet}$, where $I$ runs over finite sets, $V$ runs over objects of $\Rep_{\cO_E}((\wh{G}\rtimes Q)^I)$, $x$ runs over morphisms $\mathbf{1}\ra V|_{\De(\wh{G})}$, $\xi$ runs over morphisms $V|_{\De(\wh{G})}\ra\mathbf{1}$, and $\ga_\bullet$ runs through $W^I$.
\begin{prop*}
Let $L$ be an algebraically closed field over $\cO_E$. We have a unique bijection
\begin{align*}
  \left\{
  {\begin{tabular}{c}
     $\cO_E$-algebra homomorphisms\\
     $\chi:\Exc(W,\wh{G})\ra L$
  \end{tabular}}
\right\}\ra^\sim\left\{
  {\begin{tabular}{c}
     semisimple homomorphisms \\
     $\rho:W\ra\wh{G}(L)\rtimes Q$ over $Q$
  \end{tabular}}
\right\}\!\Big/\wh{G}(L)\mbox{-conj.}
\end{align*}
such that $\chi(S_{I,V,x,\xi,\ga_\bullet})$ equals the composition
\begin{align*}
\xymatrixcolsep{.5in}
\xymatrix{L\ar[r]^-x & V(L)\ar[r]^-{(\rho(\ga_i))_{i\in I}} & V(L)\ar[r]^-\xi & L.}
\end{align*}
\end{prop*}
\begin{proof}
This follows immediately from \cite[Corollary VII.4.3]{FS21}.
\end{proof}

\subsection{}\label{ss:globalexcursion}
The following theorem summarizes the work of V. Lafforgue \cite{Laf16} and Xue \cite{Xue20b} on global excursion operators. Write $\Bun_{G,N}(\bF_q)$ for the groupoid of $G$-bundles on $C$ equipped with a trivialization along $N$.
\begin{thm*}
There exists a unique $E$-algebra homomorphism
  \begin{align*}
    \Exc(W_F,\wh{G})_E\ra\End_{C_c(K_N\bs G(\bA)/K_N,E)}(C_c(\Bun_{G,N}(\bF_q),E))
  \end{align*}
  that sends $S_{I,V,x,\xi,\ga_\bullet}$ to the composition
    \begin{align*}
    \xymatrix{C_c(\Bun_{G,N}(\bF_q),E)\ar@{=}[r] & \cH^{*,0}_{\mathbf{1},N,E}|_{\ov\eta}\ar[r]^-x & \cH^{*,0}_{V|_{\De(\wh{G})},N,E}|_{\ov\eta} \ar@{=}[r] &\cH^{I,0}_{V,N,E}|_{\De(\ov\eta)} \ar[d]^-{\ga_\bullet} \\
    C_c(\Bun_{G,N}(\bF_q),E) & \ar@{=}[l] \cH^{*,0}_{\mathbf{1},N,E}|_{\ov\eta} & \ar[l]_-\xi \cH^{*,0}_{V|_{\De(\wh{G})},N,E}|_{\ov\eta} &\ar@{=}[l] \cH^{I,0}_{V,N,E}|_{\De(\ov\eta)}.}
    \end{align*}
    Moreover, the image of $\Exc(W_F,\wh{G})_E$ in $\End_{C_c(K_N\bs G(\bA)/K_N,E)}(C_c(\Bun_{G,N}(\bF_q),E))$ preserves the kernel of the surjective $C_c(K_N\bs G(\bA)/K_N,E)$-equivariant map
    \begin{align*}
      C_c(\Bun_{G,N}(\bF_q),E)\ra C_c(\Bun_{G,N}(\bF_q)/\Xi,E),
    \end{align*}
    so we obtain an $E$-algebra homomorphism
    \begin{align*}
      \Exc(W_F,\wh{G})_E\ra\End_{C_c(K_N\bs G(\bA)/K_N,E)}(C_c(\Bun_{G,N}(\bF_q)/\Xi,E)).
    \end{align*}
\end{thm*}
\begin{proof}
  Arguing as in \cite[p.~870]{Laf16} shows that the images of the $S_{I,V,x,\xi,\ga_\bullet}$ satisfy the necessary relations, so we get the desired $E$-algebra homomorphism
  \begin{align*}
    \Exc(W_F,\wh{G})_E\ra\End_{C_c(K_N\bs G(\bA)/K_N,E)}(C_c(\Bun_{G,N}(\bF_q),E)).
  \end{align*}
Next, because $\Sht^{(I_1,\dotsc,I_k)}_{G,V,N}\ra\Sht^{(I_1,\dotsc,I_k)}_{G,V,N}\!\!/\Xi$ is \'etale, \ref{ss:globalsatakemodxi} yields a natural $!$-pushforward morphism $\cH^I_{V,N,E}\ra\cH^I_{\Xi,V,N,E}$, which induces a morphism from the composition diagram above to the analogous composition diagram for $\cH^I_{\Xi,V,N,E}$. Note that, when $I=*$ and $V=\mathbf{1}$, the natural $!$-pushforward morphism recovers
  \begin{align*}
    C_c(\Bun_{G,N}(\bF_q),E)\ra C_c(\Bun_{G,N}(\bF_q)/\Xi,E)
  \end{align*}
on fibers. Thus the image of $S_{I,V,x,\xi,\ga_\bullet}$ in $\End_{C_c(K_N\bs G(\bA)/K_N,E)}(C_c(\Bun_{G,N}(\bF_q),E))$ satisfies the desired property.
\end{proof}

\subsection{}\label{ss:globalexcursionjustG}
We now elaborate on variants of Theorem \ref{ss:globalexcursion}. Recall that 
\begin{align*}
\Bun_{G,N}(\bF_q)\cong\coprod_\al G_\al(F)\bs G_\al(\bA)/K_N
\end{align*}
as groupoids \cite[Remarque 12.2]{Laf16}, where $\al$ runs over $G$-bundles on $\Spec{F}$ whose pullback to $\Spec{F_c}$ is trivial for all closed points $c$ of $C$, and $G_\al$ denotes the inner twist of $G_F$ over $F$ associated with $\al$. Hence $C_c(G(F)\bs G(\bA)/K_N,E)$ and $C_c(G(F)\Xi\bs G(\bA)/K_N,E)$ are $C_c(K_N\bs G(\bA)/K_N,E)$-stable direct summands of
\begin{align*}
  C_c(\Bun_{G,N}(\bF_q),E)\mbox{ and }C_c(\Bun_{G,N}(\bF_q)/\Xi,E),
\end{align*}
respectively, so Theorem \ref{ss:globalexcursion} induces $E$-algebra homomorphisms
\begin{align*}
  \Exc(W_F,\wh{G})_E&\ra\End_{C_c(K_N\bs G(\bA)/K_N,E)}(C_c(G(F)\bs G(\bA)/K_N,E)),\\
  \Exc(W_F,\wh{G})_E&\ra\End_{C_c(K_N\bs G(\bA)/K_N,E)}(C_c(G(F)\Xi\bs G(\bA)/K_N,E)).
\end{align*}

\subsection{}\label{ss:localexcursion}
For us, the most convenient interpretation of Fargues--Scholze \cite{FS21} is the following theorem. Write $\fz_{K_n}(G(F_v),\La)$ for the center of $C_c(K_n\bs G(F_v)/K_n,\La)$.
\begin{thm*}
There exists a unique $\La$-algebra homomorphism
\begin{align*}
\Exc(W_{F_v},\wh{G})_\La\ra\fz_{K_n}(G(F_v),\La)
\end{align*}
that sends $S_{I,V,x,\xi,\ga_\bullet}$ to the composition
    \begin{align*}
    \xymatrix{C_c(G(F_v)/K_n,\La)\ar@{=}[r] & \cH^{\loc,*,0}_{\mathbf{1},nv,\La}|_{\ov\eta_v}\ar[r]^-x & \cH^{\loc,*,0}_{V|_{\De(\wh{G})},nv,\La}|_{\ov\eta_v}\ar@{=}[r] &\cH^{\loc,I,0}_{V,nv,\La}|_{\De(\ov\eta_v)}\ar[d]^-{\ga_\bullet} \\
    C_c(G(F_v)/K_n,\La) & \ar@{=}[l] \cH^{\loc,*,0}_{\mathbf{1},nv,\La}|_{\ov\eta_v} & \ar[l]_-\xi \cH^{\loc,*,0}_{V|_{\De(\wh{G})},nv,\La}|_{\ov\eta_v} &\ar@{=}[l] \cH^{\loc,I,0}_{V,nv,\La}|_{\De(\ov\eta_v)}.}
  \end{align*}
\end{thm*}
\begin{proof}
This follows from \cite[Corollary IX.2.4]{FS21} and \cite[Theorem VIII.4.1]{FS21}.
\end{proof}

\subsection{}\label{ss:localglobalcompatibilityalgebras}
We now prove local-global compatibility on the level of algebras over $E$. Write $\bA^v$ for the away-from-$v$ adeles, write $K_N^v$ for $\bA^v\cap K_N$, and let $n$ be the multiplicity of $v$ in $N$. So $K_N=K_nK^v_N$.
\begin{thm*}
  The square
  \begin{align*}
    \xymatrix{\Exc(W_{F_v},\wh{G})_E\ar[r]\ar[d] & \fz_{K_n}(G(F_v),E)\ar[d] \\
\Exc(W_F,\wh{G})_E\ar[r] & \End_{C_c(K_N\bs G(\bA)/K_N,E)}(C_c(G(F)\bs G(\bA)/K_N,E))
    }
  \end{align*}
  commutes.
\end{thm*}
\begin{proof}
It suffices to check commutativity on the canonical generators $S_{I,V,x,\xi,\ga_\bullet}$ of $\Exc(W_{F_v},\wh{G})_E$, where $I$ is a finite set, $V$ is an object of $\Rep_E((\wh{G}\rtimes\Gal(\wt{F}/F))^I)$, $x$ is a morphism $\mathbf{1}\ra V|_{\De(\wh{G})}$, $\xi$ is a morphism $V|_{\De(\wh{G})}\ra\mathbf{1}$, and $\ga_\bullet$ is in $W_{F_v}^I$. This amounts to computing certain actions on $C_c(G(F)\bs G(\bA)/K_N,E)$, which we check on the basis given by $\mathbf{1}_{G(F)gK_N}$ for $g$ in $G(\bA)$. Since the $C_c(K_n\bs G(F_v)/K_n,E)$-action commutes with the $C_c(K_N^v\bs G(\bA^v)/K_N^v,E)$-action, we can assume that the away-from-$v$ components of $g$ equal $1$.

  Then $\mathbf{1}_{G(F)gK_N}$ equals the image of $\mathbf{1}_{g_vK_n}$ under the natural pushforward map
  \begin{align*}
    C_c(G(F_v)/K_n,E)\ra C_c(G(F)\bs G(\bA)/K_N,E).
  \end{align*}
Because this map commutes with the $C_c(K_n\bs G(F_v)/K_n,E)$-action, it also commutes with the action of the image of $S_{I,V,x,\xi,\ga_\bullet}$ in $\fz_{K_n}(G(F_v),E)$. Hence we can compute the latter for $\mathbf{1}_{G(F)gK_N}$ by computing it for $\mathbf{1}_{g_vK_n}$.

Fix $s$ such that $\mathbf{1}_{g_vK_n}$ lies in the image of $\cH^{\loc,*,\leq s,0}_{\mathbf{1},nv,E}|_{\ov\eta_v}$ in
\begin{align*}
 \cH^{\loc,*,0}_{\mathbf{1},nv,E}|_{\ov\eta_v}=C_c(G(F_v)/K_n,E).
\end{align*}
By Theorem \ref{ss:localexcursion} and \ref{ss:localcohomologysheaves}, the image of $S_{I,V,x,\xi,\ga_\bullet}$ in $\fz_{K_n}(G(F_v),E)$ acts on $\mathbf{1}_{gK_n}$ via the composition
\begin{align*}\label{eqn:local}
  \cH^{\loc,*,\leq s,0}_{\mathbf{1},nv,E}|_{\ov\eta_v}&\lra^x \cH^{\loc,*,\leq s,0}_{V|_{\De(\wh{G})},nv,E}|_{\ov\eta_v} = \cH^{\loc,I,\leq s,0}_{V,nv,E}|_{\De(\ov\eta_v)}\\
  &\lra^{\ga_\bullet}\cH^{\loc,I,\leq s',0}_{V,nv,E}|_{\De(\ov\eta_v)}=\cH^{\loc,*,\leq s',0}_{V|_{\De(\wh{G})},nv,E}|_{\ov\eta_v}\ra^\xi\cH^{\loc,*,\leq s',0}_{\mathbf{1},nv,E}|_{\ov\eta_v}\tag{$\star$}
\end{align*}
for large enough $s'$. By enlarging the away-from-$v$ part of $N$ and using the action of $C_c(K_N^v\bs G(\bA^v)/K_N^v,E)$ as before, we can assume that $\deg{N}$ is large enough. Then Lemma \ref{ss:localglobalsatakecomparison} shows that $\Te_n$ yields a natural $\natural$-pushforward morphism
\begin{align*}
 \cH^{\loc,I,\leq s}_{V,nv,E}|_{\De(\ov\eta_v)}\ra\cH^{I,\leq s}_{V,N,E}|_{\De(\ov\eta)},
\end{align*}
where we use Lemma \ref{ss:localglobalsatakecomparison}, \cite[Proposition VII.5.2]{FS21}, and \cite[(5.7.2)]{Hub96} to identify
\begin{align*}
  (f^{\mathrm{S}})^\Diamond_{\De(\ov\eta_v)\natural}\big[(\prescript\prime{}\cF^{(I_1,\dotsc,I_k),\leq s}_{V,N,E})^\Diamond_{\De(\ov\eta_v)}\big] = \cH^{I,\leq s}_{V,N,E}|_{\De(\ov\eta)}.
\end{align*}
Lemma \ref{ss:localglobalpartialfrobenius} and Lemma \ref{ss:FSpartialfrobeniuscomparison} imply that $\cH^{\loc,I,\leq s}_{V,nv,E}|_{\De(\ov\eta_v)}\ra\cH^{I,\leq s}_{V,N,E}|_{\De(\ov\eta)}$ induces a morphism from the composition diagram in Equation (\ref{eqn:local}) to the composition diagram
\begin{align*}\label{eqn:global}
  \cH^{*,\leq s,0}_{\mathbf{1},N,E}|_{\ov\eta}&\lra^x \cH^{*,\leq s,0}_{V|_{\De(\wh{G})},N,E}|_{\ov\eta} = \cH^{I,\leq s,0}_{V,N,E}|_{\De(\ov\eta)}\\
  &\lra^{\ga_\bullet}\cH^{I,\leq s',0}_{V,N,E}|_{\De(\ov\eta)}=\cH^{*,\leq s',0}_{V|_{\De(\wh{G})},N,E}|_{\ov\eta}\ra^\xi\cH^{*,\leq s',0}_{\mathbf{1},N,E}|_{\ov\eta}. \tag{$\star\star$}
\end{align*}
When $I=*$ and $V=\mathbf{1}$, the natural $\natural$-pushforward morphism recovers
\begin{align*}
C_c(G(F_v)/K_n,E)\ra C_c(G(F)\bs G(\bA)/K_N,E)
\end{align*}
on fibers, so we see that the image of $S_{I,V,x,\xi,\ga_\bullet}$ in $\fz_{K_n}(G(F_v),E)$ acts on $\mathbf{1}_{G(F)gK_N}$ via Equation (\ref{eqn:global}). But Theorem \ref{ss:globalexcursion} and \ref{ss:globalcohomologysheaves} indicate that this is precisely how the image of $S_{I,V,x,\xi,\ga_\bullet}$ in $\Exc(W_F,\wh{G})_E$ acts on $\mathbf{1}_{G(F)gK_N}$, as desired.
\end{proof}

\subsection{}
Let us recall the elements of the Bernstein center constructed by Genestier--Lafforgue \cite{GL17}. Write $\fm_E$ for the maximal ideal of $\cO_E$, and let $c$ be a non-negative integer. Write $\fz_{K_n}(G(F_v),\cO_E/\fm_E^c)$ for the center of $C_c(K_n\bs G(F_v)/K_n,\cO_E/\fm_E^c)$. For any finite set $I$, algebraic function $f$ on $\wh{G}\bs(\prescript{L}{}G)^I\!/\wh{G}$, element $\ga_\bullet$ of $W_{F_v}^I$, and positive integer $n$, write $\fz_{n,c,I,f,\ga_\bullet}^{\GL}$ for the element of $\fz_{K_n}(G(F_v),\cO_E/\fm_E^c)$ constructed in \cite[Th\'eor\`eme 1.1]{GL17}\footnote{While \cite[Th\'eor\`eme 1.1]{GL17} is stated for split $G$, the proof adapts for all $G$. Indeed, this is implicitly used in \cite[Th\'eor\`eme 8.1]{GL17}.}.

\subsection{}\label{ss:BernsteincenterGLFS}
We prove that the elements of the Bernstein center constructed by Fargues--Scholze coincide with those constructed by Genestier--Lafforgue. Recall that the image of $\Exc(W_F,\wh{G})$ in $\End_{C_c(K_N\bs G(\bA)/K_N,E)}(C_{\cusp}(G(F)\Xi\bs G(\bA)/K_N,E))$ preserves $C_{\cusp}(G(F)\Xi\bs G(\bA)/K_N,\cO_E)$ \cite[Proposition 13.1]{Laf16}, so \ref{ss:globalexcursionjustG} induces an $\cO_E$-algebra homomorphism
\begin{align*}
\Exc(W_F,\wh{G})\ra\End_{C_c(K_N\bs G(\bA)/K_N,\cO_E)}(C_{\cusp}(G(F)\Xi\bs G(\bA)/K_N,\cO_E)).
\end{align*}
For any object $V$ of $\Rep_{\cO_E}(\prescript{L}{}G)^I$, morphism $x:\mathbf{1}\ra V|_{\De(\wh{G})}$, and morphism $\xi:V|_{\De(\wh{G})}\ra\mathbf{1}$, write $f$ for the algebraic function on $\wh{G}\bs(\prescript{L}{}G)^I\!/\wh{G}$ given by $g_\bullet\mapsto \xi(g_\bullet\cdot x)$.
\begin{thm*}
  The square
  \begin{align*}
    \xymatrix{\Exc(W_{F_v},\wh{G})\ar[r]\ar[d] & \fz_{K_n}(G(F_v),\cO_E)\ar[d] \\
    \Exc(W_F,\wh{G})\ar[r] & \End_{C_c(K_N\bs G(\bA)/K_N,\cO_E)}(C_{\cusp}(G(F)\Xi\bs G(\bA)/K_N,\cO_E))
    }
  \end{align*}
commutes. Consequently, the image of $S_{I,V,x,\xi,\ga_\bullet}$ in $\fz_{K_n}(G(F_v),\cO_E/\fm_E^c)$ equals $\fz^{\GL}_{n,c,I,f,\ga_\bullet}$.
\end{thm*}
\begin{proof}
  Since Theorem \ref{ss:localexcursion} is compatible with changing $\La$, the first claim follows immediately from \ref{ss:globalexcursionjustG} and Theorem \ref{ss:localglobalcompatibilityalgebras}. From here, tensoring with $\cO_E/\fm_E^c$ shows that the image of $S_{I,V,x,\xi,\ga_\bullet}$ in $\fz_{K_n}(G(F_v),\cO_E/\fm_E^c)$ has the same action on
  \begin{align*}
    C_{\cusp}(G(F)\Xi\bs G(\bA)/K_N,\cO_E/\fm_E^c)
  \end{align*}
  as the image of $S_{I,V,x,\xi,\ga_\bullet}$ in $\Exc(W_F,\wh{G})$ does. Now $\fz^{\GL}_{n,c,I,f,\ga_\bullet}$ enjoys the same property by \cite[Proposition 1.3]{GL17}, so they must be equal by \cite[Lemma 1.4]{GL17}.
\end{proof}

\subsection{}\label{ss:localglobalcompatibilityrepresentations}
We conclude this section by proving Theorem A. For us, \emph{cuspidal automorphic representations} of $G(\bA)$ are irreducible summands of $C^\infty_{\cusp}(G(F)\Xi\bs G(\bA),\ov\bQ_\ell)$ lying in a single generalized eigenspace for $\Exc(W_F,\wh{G})_{\ov\bQ_\ell}$, where $\Xi$ is some lattice of $Z(F)\bs Z(\bA)$.
\begin{thm*}
  The square
    \begin{align*}
    \xymatrix{
\left\{
  {\begin{tabular}{c}
    cuspidal automorphic\\
    representations of $G(\bA)$
  \end{tabular}}
\right\}\ar[r]^-{\GLC_G}\ar[d]^-{(-)_v} & \left\{
  {\begin{tabular}{c}
  $L$-parameters \\
  for $G$ over $F$
  \end{tabular}}
    \right\}\ar[d]^-{(-)|_{W_{F_v}}^{\semis}} \\
\left\{
  {\begin{tabular}{c}
    irreducible smooth\\
    representations of $G(F_v)$
  \end{tabular}}
\right\}\ar[r]^-{\LLC^{\semis}_{G_{F_v}}} & \left\{
  {\begin{tabular}{c}
  semisimple $L$-parameters \\
  for $G_{F_v}$ over $F_v$
  \end{tabular}}
\right\}
    }
    \end{align*}
    commutes.
\end{thm*}
\begin{proof}
  Let $\Pi$ be a cuspidal automorphic representation of $G(\bA)$, and let $N$ be large enough such that $\Pi^{K_N}$ is nonzero. Adapt the notation of \ref{ss:localglobalcompatibilityalgebras}, and write $\chi_{\Pi_v}:\fz_{K_n}(G(F_v),\ov\bQ_\ell)\ra\ov\bQ_\ell$ for the $\ov\bQ_\ell$-algebra homomorphism induced by $\Pi_v^{K_n}$.

 By Theorem \ref{ss:BernsteincenterGLFS}, the square
  \begin{align*}
        \xymatrix{\Exc(W_{F_v},\wh{G})_{\ov\bQ_\ell}\ar[r]\ar[d] & \fz_{K_n}(G(F_v),\ov\bQ_\ell)\ar[d] \\
\Exc(W_F,\wh{G})_{\ov\bQ_\ell}\ar[r] & \End_{C_c(K_N\bs G(\bA)/K_N,\ov\bQ_\ell)}(C_{\cusp}(G(F)\Xi\bs G(\bA)/K_N,\ov\bQ_\ell))
    }
  \end{align*}
commutes. The action of $\Exc(W_F,\wh{G})_{\ov\bQ_\ell}$ on $\Pi^{K_N}$ corresponds to $\GLC_G(\Pi)$ under Proposition \ref{ss:excursionalgebra}, and composition with the left arrow corresponds to $\GLC_G(\Pi)|_{W_{F_v}}^{\semis}$ under Proposition \ref{ss:excursionalgebra}. On the other hand, the action of $\fz_{K_n}(G(F_v),\ov\bQ_\ell)$ on $\Pi^{K_N}$ corresponds to $\chi_{\Pi_v}$, and the composition with the top arrow corresponds to to $\LLC_{G_{F_v}}^{\semis}(\Pi_v)$ under Proposition \ref{ss:excursionalgebra}. Hence commutativity of the square yields the desired result.
\end{proof}

\section{Applications}\label{s:applications}
We revert our notation to the local context: let $F$ be a local field of characteristic $p>0$, let $G$ be a connected reductive group over $F$, and write $C$ for its radical. Our goal in this section is to prove Theorem B, Theorem C, and Theorem D. The proofs all proceed by carefully embedding local representations into global ones.

\subsection{}\label{ss:nonsemisimplify}
We now prove Theorem B. Fix an isomorphism $\ov\bQ_\ell\cong\bC$.
\begin{thm*}
The $\LLC^{\semis}_G$ uniquely lifts to a family of maps
    \begin{align*}
    \LLC_G:\left\{
  {\begin{tabular}{c}
    irreducible smooth\\
    representations of $G(F)$
  \end{tabular}}
\right\}\ra\left\{
  {\begin{tabular}{c}
  $L$-parameters \\
  for $G$ over $F$
  \end{tabular}}
\right\},
    \end{align*}
where $G$ runs over connected reductive groups over $F$, that is compatible with twisting by characters, compatible with parabolic induction for essentially $L^2$ representations as in \cite[Conjecture 4.1 (5)]{KT}, and whose value on $L^2$ representations with finite order central character is pure.
\end{thm*}
\begin{proof}
  By compatibility with parabolic induction for essentially $L^2$ representations, $\LLC_G$ is determined by its values on essentially $L^2$ representations $\pi$. By compatibility with twisting by characters, we can assume that $\pi$ also has finite order central character $\om_\pi:C(F)\ra\ov\bQ_\ell^\times$. There exists at most one pure $L$-parameter for $G$ over $F$ whose semisimplification equals $\LLC^{\semis}_G(\pi)$ \cite[Lemma 3.5.(b)]{GHSBP21}, so we just need to construct it.
  
  By \cite[Lemma 3.2]{GL18}, there exists a global field $\mathbf{F}$ of characteristic $p$, a place $v$ of $\mathbf{F}$, a connected reductive group $\mathbf{G}$ over $\mathbf{F}$, and an isomorphism $\mathbf{F}_v\cong F$ such that
  \begin{enumerate}[$\bullet$]
  \item $\mathbf{G}_{\mathbf{F}_v}$ is identified with $G$ as group schemes over $\mathbf{F}_v\cong F$,
  \item the radical $\mathbf{C}$ of $\mathbf{G}$ has $\mathbf{F}$-split rank equal to the $F$-split rank of $C$.
  \end{enumerate}
  Write $\bA_{\mathbf{F}}$ for the adele ring of $\mathbf{F}$. By the Chebotarev density theorem, there exists a place $v'\neq v$ of $\mathbf{F}$ where $\mathbf{G}_{\mathbf{F}_{v'}}$ is split. Write $\bF_{q'}$ for the residue field of $\mathbf{F}_{v'}$, identify $\mathbf{F}_{v'}$ with $\bF_{q'}\lp{\textstyle\frac1z}$, and write $\mathbf{G}_{\mathbf{F}_{v'}}$ as the pullback of a split connected reductive group $\mathbf{H}$ over $\bF_{q'}$.

Let $\phi$ be a generic character for $\mathbf{H}$ as in \cite[Section 1.3]{HNY13}, write $\Pi'$ for the cuspidal automorphic representation of $\mathbf{H}(\bA_{\bF_{q'}(z)})$ associated with the automorphic sheaf $A_\phi$ as in \cite[Definition 2.6]{HNY13}, and write $\rho'$ for the $L$-parameter for $\mathbf{H}$ over $\bF_{q'}(z)$ associated with the $\prescript{L}{}{\mathbf{H}}$-local system $\mathrm{Kl}_{\prescript{L}{}{\mathbf{H}}}(\phi)$ as in \cite[Theorem 1(1)]{HNY13}. Since $A_\phi$ is a Hecke eigensheaf with eigenvalue $\mathrm{Kl}_{\prescript{L}{}{\mathbf{H}}}(\phi)$, we see that $\Pi'$ and $\rho'$ are associated via the Satake isomorphism at cofinitely many places of $\bF_{q'}(z)$. Now $\rho'|_{W_{\bF_{q'}\lp{\frac1z}}}$ is irreducible by \cite[Corollary 5.1(1)]{HNY13} and \cite[Remark 4.5.10(i)]{XZ22}, so $\rho'$ is irreducible. Hence \cite[Th\'eor\`eme 12.3]{Laf16} and \cite[Proposition 6.4]{BHKT19}\footnote{While \cite{BHKT19} only considers split $G$, \cite[Proposition 6.4]{BHKT19} immediately extends to general $G$.} imply that $\rho'=\GLC_{\mathbf{H}}(\Pi')$. From here, Theorem \ref{ss:localglobalcompatibilityrepresentations} shows that $\LLC^{\semis}_{\mathbf{H}_{\bF_{q'}\lp{\frac1z}}}(\Pi'_\infty)=\rho'|_{W_{\bF_{q'}\lp{\frac1z}}}^{\semis}=\rho'|_{W_{\bF_{q'}\lp{\frac1z}}}$.

  By \cite[p.~2829]{GL18}, there exists a finite order character $\om:\mathbf{C}(\mathbf{F})\bs\mathbf{C}(\bA_{\mathbf{F}})\ra\ov\bQ_\ell^\times$ such that $\om_v$ is identified with $\om_\pi$ and $\om_{v'}$ is identified with an unramified twist of $\om_{\Pi'_\infty}$. Note that $\ker\om$ contains a lattice $\Xi$ of $\mathbf{C}(\mathbf{F})\bs\mathbf{C}(\bA_{\mathbf{F}})$. Then \cite[Lemma A.1]{GHSBP21} and \cite[Lemma 8.1]{GL18} yield an irreducible summand $\Pi$ of $C^\infty_{\cusp}(\mathbf{G}(\mathbf{F})\Xi\bs\mathbf{G}(\bA_{\mathbf{F}}),\ov\bQ_\ell)$ such that
  \begin{enumerate}[$\bullet$]
  \item $\Pi_v$ has the same cuspidal support as $\pi$,
  \item $\Pi_{v'}$ is isomorphic to an unramified twist of $\Pi'_\infty$ via $\mathbf{F}_{v'}\cong\bF_{q'}\lp{\textstyle\frac1z}$.
  \end{enumerate}
Theorem \ref{ss:localglobalcompatibilityrepresentations} and \cite[p.~326]{FS21} indicate that $\GLC_{\mathbf{G}}(\Pi)|_{W_{\mathbf{F}_{v'}}}^{\semis}$ equals an unramified twist of $\LLC^{\semis}_{\mathbf{H}_{\bF_{q'}\lp{\frac1z}}}(\Pi'_\infty)$. This shows that $\GLC_{\mathbf{G}}(\Pi)$ is irreducible, so \cite[Lemme 16.2]{Laf16} and \cite[Lemma 11.4]{ST21} imply that $\GLC_{\mathbf{G}}(\Pi)$ is pure.\footnote{Now \cite[Lemme 16.2]{Laf16} and \cite[Lemma 11.4]{ST21} are stated for split $G$, but they hold in general.} Hence $\GLC_{\mathbf{G}}(\Pi)|_{W_{\mathbf{F}_v}}$ is pure as in \cite[Definition 3.3.(b)]{GHSBP21}. Finally, Theorem \ref{ss:localglobalcompatibilityrepresentations} and \cite[Corollary IV.7.3]{FS21} show that $\GLC_{\mathbf{G}}(\Pi)|_{W_{\mathbf{F}_v}}^{\semis}=\LLC^{\semis}_G(\Pi_v)=\LLC^{\semis}_G(\pi)$, so $\GLC_{\mathbf{G}}(\Pi)|_{W_{\mathbf{F}_v}}$ is the unique pure $L$-parameter for $G$ over $F$ whose semisimplification equals $\LLC^{\semis}_G(\pi)$.
\end{proof}

\subsection{}\label{ss:GLFS}
We give the following abstract proof of Theorem C.
\begin{thm*}
  There exists at most one family of maps
  \begin{align*}
    \cL\cL\cC^{\semis}_G:\left\{
  {\begin{tabular}{c}
    irreducible smooth\\
    representations of $G(F)$
  \end{tabular}}
\right\}\ra\left\{
  {\begin{tabular}{c}
  semisimple $L$-parameters \\
  for $G$ over $F$
  \end{tabular}}
\right\},
  \end{align*}
  where $G$ runs over connected reductive groups over $F$, that is compatible with twisting by characters as in \cite[Property 2.8]{Har22}, compatible with parabolic induction as in \cite[Property 2.13]{Har22}, and satisfies the conclusion of Theorem \ref{ss:localglobalcompatibilityrepresentations}.

  Consequently, the Genestier--Lafforgue correspondence agrees with the Fargues--Scholze correspondence.
\end{thm*}

\begin{proof}
  By compatibility with parabolic induction, $\cL\cL\cC^{\semis}_G$ is determined by its values on cuspidal representations $\pi$. By compatibility with twisting by characters, we can assume that $\pi$ also has finite order central character $\om_\pi:C(F)\ra\ov\bQ_\ell^\times$.

Let $\mathbf{F}$, $v$, $\mathbf{G}$, and $\mathbf{C}$ be as in the proof of Theorem \ref{ss:nonsemisimplify}. By \cite[Lemma 3.3]{GL18}, there exists a finite order character $\om:\mathbf{C}(\mathbf{F})\bs\mathbf{C}(\bA_{\mathbf{F}})\ra\ov\bQ_\ell^\times$ such that $\om_v$ is identified with $\om_\pi$. Note that $\ker\om$ contains a lattice $\Xi$ of $\mathbf{C}(\mathbf{F})\bs\mathbf{C}(\bA_{\mathbf{F}})$. Poincar\'e series yield an irreducible summand $\Pi$ of $C^\infty_{\cusp}(\mathbf{G}(\mathbf{F})\Xi\bs\mathbf{G}(\bA),\ov\bQ_\ell)$ such that $\Pi_v$ is identified with $\pi$ \cite[Theorem 1.1]{GL17}, so the conclusion of Theorem \ref{ss:localglobalcompatibilityrepresentations} uniquely determines $\cL\cL\cC^{\semis}_G(\pi)$ as $\GLC_{\mathbf{G}}(\Pi)|_{W_{\mathbf{F}_v}}^{\semis}$.

The Fargues--Scholze correspondence satisfies the aforementioned properties by \cite[p.~326]{FS21}, \cite[Corollary IX.7.3]{FS21}, and Theorem \ref{ss:localglobalcompatibilityrepresentations}. The Genestier--Lafforgue correspondence also satisfies these properties by \cite[Th\'eor\`eme 8.1]{GL17}, so the above shows that it agrees with the Fargues--Scholze correspondence.
\end{proof}

\subsection{}
Finally, we prove Theorem D. Let $D$ be a central simple algebra over $F$ of degree $n$.
\begin{thm*}
  The triangle
  \begin{align*}
    \xymatrixcolsep{-.5in}
    \xymatrix{\left\{
  {\begin{tabular}{c}
    irreducible essentially $L^2$\\
    representations of $D^\times$
  \end{tabular}}
\right\}\ar[rr]^-{\JL}\ar[dr]_-{\LLC^{\semis}_{D^\times}} & & \left\{
  {\begin{tabular}{c}
    irreducible essentially $L^2$\\
     representations of $\GL_n(F)$
  \end{tabular}}
    \right\} \ar[dl]^-{\LLC^{\semis}_{\GL_n}}\\
    & \left\{
  {\begin{tabular}{c}
   $n$-dimensional semisimple \\
    representations of $W_F$
  \end{tabular}}
    \right\} &}
  \end{align*}
commutes, where $\JL$ denotes the local Jacquet--Langlands correspondence as in \cite[(th. 1.1)]{Bad02}.
\end{thm*}
\begin{proof}
  Because both $\JL$ \cite[(th. 1.1)]{Bad02} and $\LLC^{\semis}_G$ \cite[p.~326]{FS21} are compatible with twisting by characters, it suffices to check commutativity on $L^2$ representations $\pi$ with finite order central character $\om_\pi:F^\times\ra\ov\bQ_\ell^\times$.

  Let $\mathbf{F}$ be a global field of characteristic $p$ along with a place $v$ of $\mathbf{F}$ and an isomorphism $\mathbf{F}_v\cong F$, and let $\mathbf{D}$ be a central division algebra over $\mathbf{F}$ such that $\mathbf{D}_{\mathbf{F}_v}$ is identified with $D$ as central simple algebras over $\mathbf{F}_v\cong F$. Using the pseudo-coefficient for $\JL(\pi)$ constructed in \cite[Section 5]{BR17}, the proof of \cite[(15.10)]{LRS93} yields a lattice $\Xi$ of $\mathbf{F}^\times\bs\bA^\times_{\mathbf{F}}$ and an irreducible summand $\wt\Pi$ of $C^\infty_{\cusp}(\GL_n(\mathbf{F})\Xi\bs\GL_n(\bA_{\mathbf{F}}),\ov\bQ_\ell)$ such that
  \begin{enumerate}[$\bullet$]
  \item $\wt\Pi_v$ is isomorphic to $\JL(\pi)$,
  \item for all places $v'\neq v$ of $\mathbf{F}$ where $\mathbf{D}_{\mathbf{F}_{v'}}$ is ramified, $\wt\Pi_{v'}$ is cuspidal.
  \end{enumerate}
Therefore we can apply the global Jacquet--Langlands correspondence \cite[Theorem 3.2]{BR17} to $\wt\Pi$, which yields an irreducible summand $\Pi$ of $ C^\infty_{\cusp}(\mathbf{D}^\times\Xi\bs(\mathbf{D}\otimes_{\mathbf{F}}\bA_{\mathbf{F}})^\times,\ov\bQ_\ell)$ such that
  \begin{enumerate}[$\bullet$]
  \item $\Pi_v$ is isomorphic to $\pi$,
  \item for all places $w$ of $\mathbf{F}$ where $\mathbf{D}_{\mathbf{F}_w}$ is split, $\Pi_w$ is isomorphic to $\wt\Pi_w$.
  \end{enumerate}
  Then \cite[Th\'eor\`eme 12.3]{Laf16} and the Chebotarev density theorem imply that
  \begin{align*}
   \GLC_{\mathbf{D}^\times}(\Pi)=\GLC_{\GL_n}(\wt\Pi),
  \end{align*}
  so Theorem \ref{ss:localglobalcompatibilityrepresentations} enables us to conclude that
  \begin{gather*}
    \LLC^{\semis}_{D^\times}(\pi)=\GLC_{\mathbf{D}^\times}(\Pi)|^{\semis}_{W_{\mathbf{F}_v}}=\GLC_{\GL_n}(\wt\Pi)|^{\semis}_{W_{\mathbf{F}_v}}=\LLC^{\semis}_{\GL_n}(\JL(\pi)).\qedhere
  \end{gather*}
\end{proof}

\bibliographystyle{../habbrv}
\bibliography{biblio}
\end{document}